\documentclass[11pt,letterpaper]{article}

\usepackage[margin=1in]{geometry}

\usepackage[T1]{fontenc}
\usepackage{lmodern}
\usepackage{microtype}

\usepackage{amsmath,amssymb}
\usepackage{amsthm}
\usepackage{bm}

\usepackage{graphicx}
\usepackage{longtable}
\usepackage{float}
\usepackage{booktabs}

\usepackage{comment}

\usepackage[authoryear,round]{natbib}
\bibpunct{(}{)}{;}{a}{,}{,}

\usepackage{url}

\usepackage[algo2e,ruled]{algorithm2e}
\usepackage{array,tabularx}

\usepackage[
    colorlinks=true,
    linkcolor=blue,
    citecolor=blue,
    urlcolor=blue
]{hyperref}

\theoremstyle{plain}
\newtheorem{theorem}{Theorem}[section]
\newtheorem{lemma}[theorem]{Lemma}
\newtheorem{proposition}[theorem]{Proposition}
\newtheorem{corollary}[theorem]{Corollary}

\theoremstyle{definition}
\newtheorem{definition}[theorem]{Definition}
\newtheorem{example}[theorem]{Example}

\newtheorem{assumption}{Assumption}

\theoremstyle{remark}
\newtheorem{remark}[theorem]{Remark}

\newcommand{\F}{\mathcal F}
\newcommand{\N}{\mathbb N}
\newcommand{\iid}{\overset{\mathrm{iid}}{\sim}}

\newcommand{\uiota}             {\mbox{\boldmath$\uiota$}}

\def\beq{\begin{equation}}
\def\eeq{\end{equation}}
\def\beqa{\begin{eqnarray}}
\def\eeqa{\end{eqnarray}}
\def\beqan{\begin{eqnarray*}}
\def\eeqan{\end{eqnarray*}}

\def\bc{\begin{center}}
\def\ec{\end{center}}
\def\btable{\begin{table}[htbp]}
\def\etable{\end{table}}
\def\bfig{\begin{figure}[htbp]}
\def\efig{\end{figure}}

\def\bi{\begin{itemize}}
\def\ei{\end{itemize}}

\def\E{\mathbb{E}}

\def\P{\mathbb{P}}

\def\F{\mathcal{F}}
\def\M{\mathcal{M}}

\def\N{\mathbb{N}}

\def\R{\mathbb{R}}

\renewcommand{\P}{\mathbb{P}}

\newcommand{\RNum}[1]{\uppercase\expandafter{\romannumeral #1\relax}}

\title{Optimal distribution-free sequential change detection\\
with conformal martingales}

\author{
  Swapnaneel Bhattacharyya\\
  {\small Wharton School, University of Pennsylvania}\\
  {\small\texttt{swap@wharton.upenn.edu}}
  \and
  Aaditya Ramdas\\
  {\small Stanford University}\\
  {\small\texttt{aramdas@stanford.edu}}
}

\date{}

\begin{document}

\maketitle

\begin{abstract}
We study the challenging problem of distribution-free sequential changepoint detection for independent observations with unknown and \emph{completely unrestricted} pre- and post-change laws. For over two decades, the primary existing solution for this problem has employed Vovk's conformal test martingales and associated e-detectors, which control the probability of false alarm (PFA) and the average run length (ARL) respectively. The  majority of these works focus on validity, with statistical efficiency usually left for simulations.
We develop a comprehensive theory of how conformal p-values behave under non-exchangeable data with a changepoint at an unknown time $T$. We use this to analyze the post-change growth and resulting detection delays
and find that the standard existing methods are provably suboptimal for PFA and ARL control, and can lead to delays that are $\Omega(T)$ and $\Omega(\sqrt{\text{ARL}})$ respectively. We construct new conformal e-processes and e-detectors with learned scores and a fixed, distribution-independent betting dictionary. For any distribution pair with positive finite KL divergence, our methods asymptotically achieve $\log T/\mathrm{KL}$ delay for PFA control and $\log\mathrm{ARL}/\mathrm{KL}$ delay (with a sharp leading constant of 1), which are both proven to be minimax optimal in a suitable double-array framework.
Simulations illustrate the large empirical benefits of our new approach.
\end{abstract}

\noindent\textbf{Keywords:}
Distribution-free inference; sequential changepoint detection;
conformal martingales; e-detectors; probability of false alarm;
average run length; detection delay.
\par\medskip


\section{Introduction}
Rapid changes in public sentiment frequently occur on online platforms. A product
failure may lead to a sudden surge of negative reviews, while a major event can
abruptly shift the tone of discussions on social media. Detecting when such
changes occur is useful for identifying emerging issues, measuring public
responses, and monitoring systems in real time. In applications of this kind the
observations are rarely scalar, and neither the law generating them before the
change nor the law generating them afterwards is known, even up to a
finite-dimensional parameter. These considerations motivate sequential
changepoint detection on an arbitrary measurable space $\mathcal X$, which may
represent $\mathbb R^d$, images, text embeddings, or other complex data objects.
 
Formally, we describe the data structure under a distribution shift as follows. A
sequence of $\mathcal X$-valued random variables $\{X_n\}_{n\geq1}$ arrives
sequentially and, for some $T\in\mathbb N\cup\{\infty\}$, satisfies
\begin{equation}
\label{eq: data structure}
X_1,\ldots,X_T\stackrel{\mathrm{iid}}{\sim}P_0,
\qquad
X_{T+1},X_{T+2},\ldots\stackrel{\mathrm{iid}}{\sim}P_1,
\end{equation}
with independence across coordinates. The pre- and post-change laws $P_0$ and
$P_1$ are unknown, as is $T$. Throughout the paper, $T$ is the last pre-change
index, so $X_{T+1}$ is the first post-change observation, and $T=\infty$ encodes
the absence of a change. We write $\mathbb P_{P_0,T,P_1}$ for the joint law under
a change after time $T$ and $\mathbb P_{P_0,\infty}$ for the no-change law.
 
A change detector is a stopping time
$\tau\in\mathbb N\cup\{\infty\}$ in a specified filtration. If $\tau<\infty$, a change is
declared; If $\tau=\infty$, no alarm is raised. Under the no-change law, the event
$\{\tau<\infty\}$ is a false alarm. Its probability, $\operatorname{PFA}_{P_0}(\tau)
:=
\mathbb P_{P_0,\infty}(\tau<\infty),$
is a natural criterion when the objective is to control the probability of ever
making a false declaration over an indefinitely long stream. A procedure is
PFA-valid at level $\alpha\in(0,1)$ if $\sup_{P_0}\operatorname{PFA}_{P_0}(\tau)\leq\alpha.$
Detection performance is then assessed by how rapidly the procedure responds after
a change while respecting this anytime error guarantee.
 
A second operating regime, and the one classically adopted in sequential analysis,
is based instead on the average run length to false alarm, $\operatorname{ARL}_{P_0}(\tau)
:=
\mathbb E_{P_0,\infty}\tau.$ A procedure is ARL-valid at level
$\gamma\geq1$ if $\inf_{P_0}\operatorname{ARL}_{P_0}(\tau)\geq\gamma.$
The two criteria describe different statistical objectives: PFA strictly below one cannot coexist with finite null expected stopping time. PFA control at a level below one leaves positive
probability of never stopping under no change and therefore implies infinite ARL.
Conversely, finite ARL implies that the procedure stops almost surely under no
change, so its PFA equals one. PFA is appropriate when false alarms have an unacceptably high cost, whereas ARL is a standard metric when false alarms may be infrequently tolerated. 
However, ARL control alone by itself is not sufficient for having satisfactory control over  false alarms.  A procedure may stop at an early time with substantial probability and still have a large average run length, provided it runs sufficiently long on the remaining paths. 
Our procedures also satisfy $\mathbb P_{P_0,\infty}\{\tau\leq m\} \leq m/\gamma$ in addition to ARL control.

Existing procedures for change detection typically proceed in one of two ways. Either
they aim at a change in a particular functional of the distribution, such as the
mean or the scale, and the monitoring statistic is designed accordingly, in which
case nothing is guaranteed against changes that leave the targeted functional
intact; or they posit a parametric model, or take the pre-change law to be known,
and monitor a likelihood ratio, in which case both validity and efficiency degrade
under misspecification. 

As a prominent exception that avoids both restrictions, \cite{vovk2021testing} constructs a distribution-free detector using conformal $p$-values. For observations $X_1,\cdots,X_n,$ define the randomized conformal $p$-values as
\begin{equation}
\label{eq:intro-randomized-pvalue}
    p_n(X_1,\ldots,X_n;\lambda_n)
    :=
    \frac{\sum_{i=1}^n \mathbf 1(X_i<X_n)+\lambda_n\sum_{i=1}^n\mathbf 1(X_i=X_n)}{n},
\end{equation} 
where $\lambda_n \sim U(0,1)$ independent of the data $\{X_1,\cdots,X_n\}$. Sometimes, if relevant the data is first transformed using a measurable score function $s: \mathcal{X} \to \R$ and the conformal $p$-values are computed on the score-transformed data $\{s(X_1),\cdots,s(X_n)\}$. Henceforth, $p_n(X_1,\cdots,X_n; \lambda_n)$ will be denoted $p_n$. Under the exchangeability of the data, the $p$-values $(p_1,\cdots,p_n)$ can be shown to be independent and identically distributed (i.i.d) uniform random variables (\cite{vovk2021testing}). Then with any sequence of nonnegative Borel-measurable ``calibrator'' functions $\{f_n\}_{n\in\mathbb N}$, satisfying $f_n:[0,1]^n\to[0,\infty]$ and, $\int_0^1 f_n(s_1,\ldots,s_n)\,ds_n
    =
    f_{n-1}(s_1,\ldots,s_{n-1}),$ 
the nonnegative martingale $S_n = f_n(p_1,\cdots,p_n)$ yields the $\alpha$ PFA level change detector $\tau^{\texttt{CTM}}(\frac{1}{\alpha}) := \inf\{ n: S_n \geq \frac{1}{\alpha}\}$. The functions $f_n$ are sometimes called \textit{betting densities}, and the process $\{f_n\}_{n \geq 1}$ is called \emph{betting martingale} and the martingale $\{S_n\}_n$ is called a \textit{Conformal Test Martingale}(CTM). With this construction, Vovk also defines the CUSUM detector $\tau^{\texttt{CUSUM}}(\gamma) := \inf\{n : \max_{i = 0,\cdots,n}\frac{S_n}{S_i} \geq \gamma\}$ and the Shiryaev-Roberts detector $\tau^\texttt{SR}(\gamma) := \inf \{n: \sum_{i=1}^{n-1}\frac{S_n}{S_i} \geq \gamma\}$
which he shows to satisfy $\E_{P_0, \infty}(\tau^\texttt{CUSUM}(\gamma)) \geq \E_{P_0, \infty}(\tau^\texttt{SR}(\gamma)) \geq \gamma$. Thus both the procedures have ARL validity at level $\gamma$.

While the conformal test martingale retains distribution-free PFA and ARL validity, the analysis in \cite{vovk2021testing}
offers no formal optimality guarantees, such as the rate of the detection
delay. As we show in Examples~\ref{eg: CTM log T delay}  and~\ref{ex:arl-single-bet-separation}, these methods can lead to detection delays that are linear in $T$ or of order $\sqrt{\gamma}$, respectively. 
In this work, we will supply a 
distribution-free framework in which both procedures with PFA and ARL guarantees are
accompanied by minimax optimality guarantees and detection delays of order $\log T$ and $\log \gamma$ for the PFA and ARL valid procedures, respectively. While our construction is directly motivated by the conformal martingales of
\cite{vovk2021testing}, our methods
 perform much better in theory and practice.

In the CTM construction, the choice of the betting density heavily influences the optimality of the corresponding detector: a betting density that bets well against one post-change law might bet poorly against another, and since $P_1$ is unknown, choosing the density optimally can be difficult. Therefore, instead of relying on a single sequence of betting densities, We consider a class of betting densities
$\mathcal F=\{f_\theta:\theta\in\Theta\}$ supported on $[0,1]$ and aggregate the corresponding martingale from each density, using a prior $\Pi$ on $\Theta$ to define,
\[M_n \equiv M_n(p_1,\cdots,p_n) := \int_{\Theta} f_\theta(p_1)\cdots f_\theta(p_n) d\Pi(\theta).\]
We call $\{M_n\}_n$ a \emph{conformal mixture martingale} (CMM). Mixing preserves the
martingale property, hence the corresponding detector 
\begin{equation}
    \label{eq: def CMM}
    \tau^{\texttt{CMM}}(b)
:=
\inf\left\{n\geq1:M_n\geq b \right\},
\end{equation}
with $b = \frac1\alpha$
has distribution-free PFA validity of level $\alpha \in (0,1)$, while allowing the
statistic to adapt to whichever $f_\theta$ is favorable for the realized change.
 
Mixing over $\Theta$ addresses the unknown post-change law, but not the unknown
changepoint. A martingale started at time one carries the entire pre-change
history in its capital; So when the change occurs late, the martingale has already accumulated a long pre-change history, and the post-change evidence must first overcome this accumulated dilution before the alarm can be raised. This structural limitation of CTM and CMM can lead to significantly delayed detection.

To overcome this dilution, we instead accumulate local evidence for the distribution change by running the martingale from every past candidate,
\[
M_{k,t} := \int_\Theta \prod_{j=k}^{t} f_\theta(p_j)\,d\Pi(\theta), \quad 1 \leq k, \leq t
\]
which measures only the evidence accrued between times $k$ and $t$, and we combine
these across past times $k$ using a deterministic sequence of nonnegative restart weights
$w=(w_k)_{k\ge1}$. Summable weights provide PFA control, while unit weights provide ARL control. Restarting, rather than monotonicity of the weights, prevents a local product from carrying the entire pre-change record. Two constructions are natural in this setup, a weighted sum that
accumulates evidence over all starts and a weighted maximum that selects the most
promising one,
\begin{equation}
\label{eq:intro-weighted-restarted-statistics}
S_t^{(w)}:=\sum_{k=1}^t w_kM_{k,t}, \hspace{.2cm}S_0^{(w)} := 0;
\qquad
Z_t^{(w)}:=\max_{1\leq k\leq t}w_kM_{k,t},\hspace{.2cm} Z_0^{(w)} := 0.
\end{equation}
Then the associated alarms $\tau_{S,w}(b)$ and $\tau_{Z,w}(b)$ can be defined as,
\begin{equation}
\label{eq:intro-weighted-stopping-times}
\tau_{S,w}(b):=\inf\{t\geq1:S_t^{(w)}\geq b\},
\qquad
\tau_{Z,w}(b):=\inf\{t\geq1:Z_t^{(w)}\geq b\}.
\end{equation}
Interestingly, this restarted process aggregation significantly improves detection delay over Vovk's method using the conformal test martingale, as is evident in Figure~\ref{fig:vovkcomparison}, where we draw a direct comparison of the detection delays of our method and Vovk's method. The figure shows that our method detects the change orders of magnitude significantly faster, and that the gap widens the later the change occurs.

\begin{figure}[h!]
    \centering
    \includegraphics[width=0.47\linewidth]{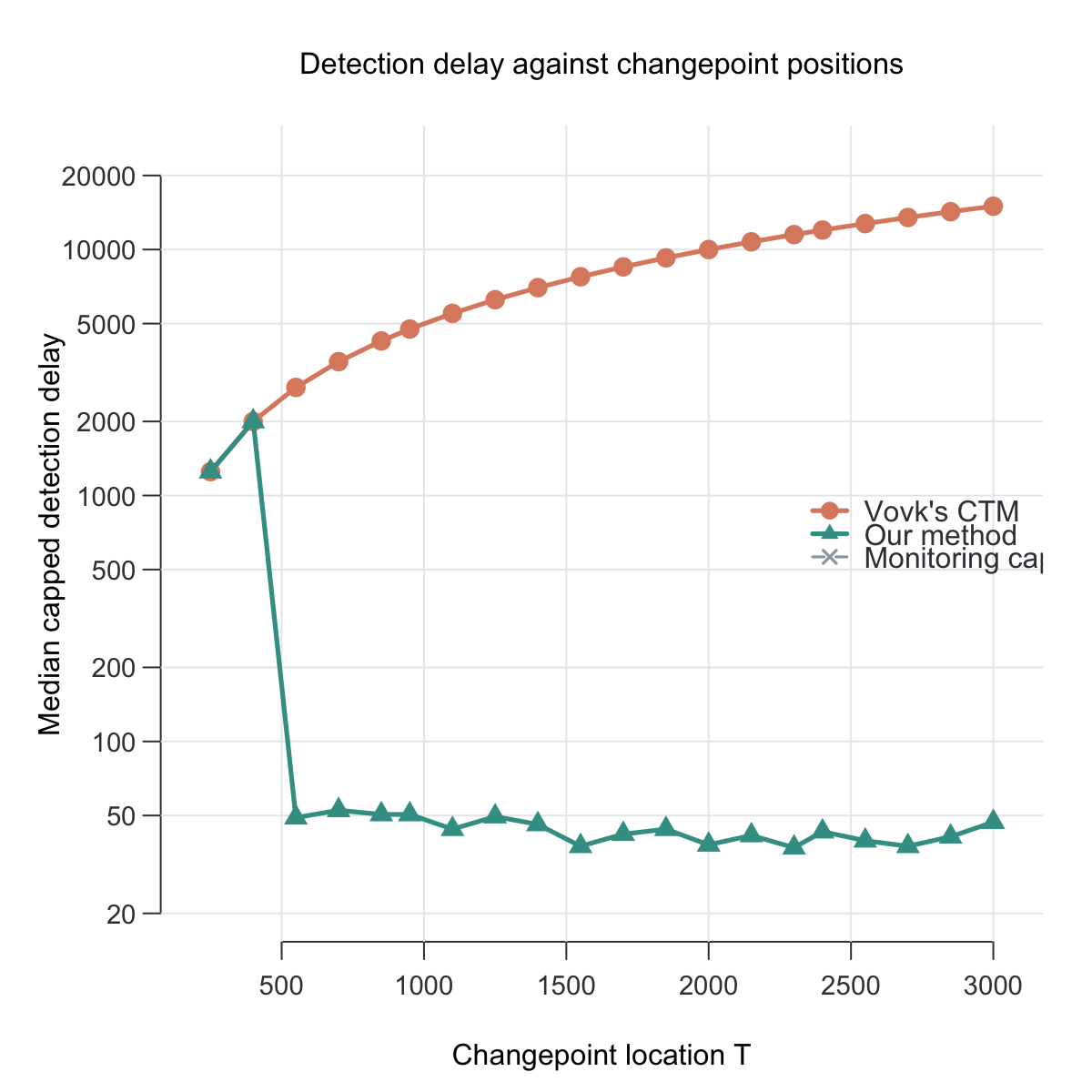}
    \caption{ \small Comparison of the detection delays (in log scale) of Vovk's conformal test martingale
and our method, for PFA control at level $\alpha = 0.05$. For different
changepoint locations, for both the methods detection delay is computed for the same $50$ independent streams generated with $P_0 \equiv N(0,1), P_1 \equiv N(1.2,1),$ 
 betting density $f(p) = 0.15\,p^{-0.85}$and weights $w_k \propto k^{-1.1}$. Each stream is monitored until time $T + 5T$, and the
reported capped delay is $(\tau - T) \wedge 5T$. Then we plot the median of the 
capped delay across the $50$ streams against $T$. In this fixed-score experiment, the restarted method offers a significant improvement in detection delay over Vovk's method (which can be $\Omega(T)$ as demonstrated in example \ref{eg: CTM log T delay}). It reaches $\sim$350 times at the right end of the plot.}
\label{fig:vovkcomparison}
    
\end{figure}

\begin{figure}[h!]
    \centering
    \includegraphics[width=0.9\linewidth]{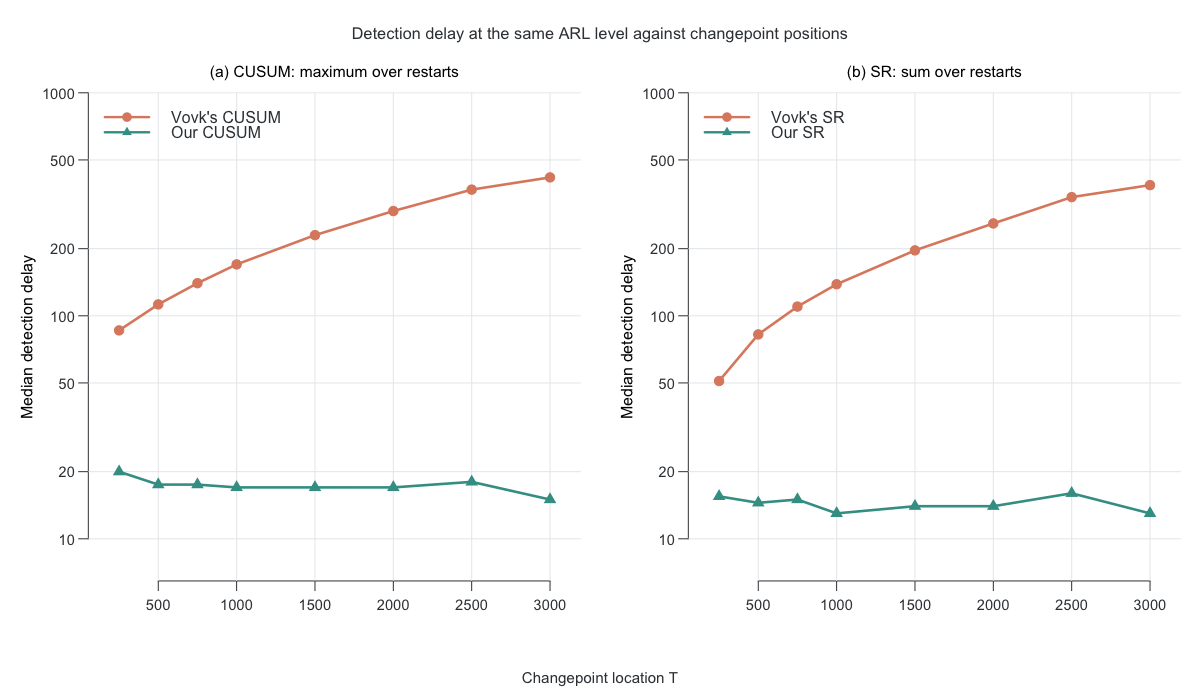}
    \caption{\small Comparison of the detection delays (in log scale) of Vovk's CUSUM and
Shiryaev--Roberts procedures and ours, for ARL control at level $\gamma = 10^{4}$.
For different changepoint locations, both the methods use the same $50$ independent streams, generated
with $P_0 \equiv N(0,1), P_1 \equiv N(1.2,1),$ betting class
$f_\theta(p) = \theta p^{\theta - 1}$ with $\theta$ ranging over
$\{0.15, 0.30, 0.50, 0.75, 1\}$ under a uniform prior and our method uses unit weights $w_k \equiv 1$. Each stream is monitored until time $T + 1500$, and the
reported capped delay is $(\tau - T) \wedge 1500$. The left panel compares the
maximum statistics and the right panel the sum statistics, so that each
procedure is matched with its counterpart. As evident, our method offers a
significant improvement in detection delay over Vovk's method, and the gap
widens as the change occurs later.}
    \label{fig:placeholder}
\end{figure}

The detectors in \eqref{eq:weighted-stopping-times} can be considered as the conformal analogues of the
Shiryaev--Roberts \citep{shiryaev1961problem} and CUSUM \citep{page1954continuous} 
statistics in this setup. For a general sequence of non-negative weights $w := (w_k)_{k \geq 1},$ we also show that the alarms satisfies the optional horizon inequality $\mathbb P_{P_0,\infty}\{\tau_{S,w}(b)\le\sigma, \tau_{S,w}(b)<\infty\} \le \mathbb E_{P_0,\infty}W_{\tau_{S,w}(b)\wedge\sigma}/b,$ for any stopping time $\sigma$, where $W_t := \sum_{k=1}^t w_k$. Thus when there is no change, our procedure is guaranteed not to stop too early. 

The role of the weights unifies the two operating regimes. If the weights
are summable with $\sum_{k =1}^\infty w_k\le1$, then the process $\Tilde{S}^{(w)}$ defined as
\begin{equation}
    \label{eq:conformal e process}
    \Tilde{S}^{(w)}_t := S_t^{(w)} + \sum_{k>t}w_k,
\end{equation}
 is a nonnegative martingale, hence an e-process, and $\tau_{S,w}(1/\alpha)$
and $\tau_{Z,w}(1/\alpha)$ have distribution-free PFA-validity at level $\alpha$; we call the process $\Tilde{S}^{(w)}$ to be a \emph{conformal e-process}, simply to distinguish it by name from conformal test martingales. If instead every weight does not exceeds one, then $S^{(w)}$ and $Z^{(w)}$ are e-detectors following the definition of
\cite{shin2023detectors}, and the corresponding stopping rules $\tau_{S,w}(\gamma)$
and $\tau_{Z,w}(\gamma)$ have distribution-free ARL validity at level $\gamma$; 
to differentiate from Vovk's  methods, we call the processes
\emph{conformal (SR and CUSUM) e-detector} respectively.

 Besides distribution-free validity, we also discuss optimality properties of
our detection schemes, such as detection delay guarantees, and provide a
restricted minimax formulation for both the PFA and the ARL controlled methods. Our detection delay bounds are information-theoretic in nature, with both governed by $D_{\mathrm{KL}}(P_1\|P_0).$

Following these terms, we
show that our level $\alpha$ PFA-controlled process, run with nearly harmonic
weights, has a detection delay of order $\log T/D_{\mathrm{KL}}(P_1\|P_0)$, and that the level
$\gamma$ ARL-controlled process has a detection delay of
order at most $\log\gamma/D_{\mathrm{KL}}(P_1\|P_0)$. Thus, our method offers significant improvement in detection delay over Vovk's method using a conformal test martingale, which can take detection delays linear in changepoint position $T$, as shown in example~\ref{eg: CTM log T delay}. This improvement is also visually evident in Figure~\ref{fig:vovkcomparison}.

As $D_{\mathrm{KL}}(P_1\|P_0)$ increases, the pre- and post-change distributions become more separable, making the detection problem easier, thus leading to a faster detection. This is also in accordance with the delay bounds. We show that, these rates are sharp for both the PFA and ARL regime. In an asymptotic setting with respect to the location of the changepoint, over a restricted class of changepoints and alternatives described later, we prove explicit lower bounds under the PFA and ARL constraint separately and show that our method attains that lower bound, i.e.\ no detector can asymptotically detect faster than ours. In this sense the proposed construction is first order minimax optimal in both regimes.

Finally, we establish a data processing inequality for the score transformed change detection, showing that for among all score transformation only the likelihood ratio of post- and pre-change observation retains all the information as in the raw
data, and our results quantify the additional cost one needs to pay for a poorly chosen score transformation. Taken together, the paper develops a distribution-free framework with unified validity guarantees for PFA and ARL control with explicit delay guarantees and first order minimaxity.

\paragraph{Notation.}
For a probability distribution $P$, $F_P$ denotes its cumulative distribution function (CDF). For  $n \in \N$, let $[n] := \{1,2,\cdots,n\}$. We denote the collection of Borel sets on $\R$ by $\mathcal{B}(\R)$ and define $\mathcal{B}[a,b] := \{A : A \subseteq [a,b], A \in \mathcal{B}(\R)\}$. For data $X_1,\cdots,X_n$, denote their order statistics by $(X_{(1)},\cdots,X_{(n)})$ satisfying $X_{(1)} \leqslant \cdots \leqslant X_{(n)}$. 
$P_0 \ll P_1$ means that $P_0$ is absolutely continuous with respect to $P_1$. For a random variable $X$ having distribution $P_1$, we denote its expectation by $\E_{P_1}[Z]$. For random vectors $\mathbf{Y} := (Y_1,\cdots,Y_m)$ and $\mathbf{Z} := (Z_1,\cdots,Z_n)$, we write $\mathbf{Y} \perp \mathbf{Z}$ if  $\mathbf{Y}$ and $\mathbf{Z}$ are independent. For a stochastic process $\{M_n\}_{n\in\mathbb N}$, we write $M$ to denote the entire sequence, i.e., $M := \{M_n\}_{n\in\mathbb N}.$ We write $\|G\|_\infty:=\sup_x|G(x)|$ for a bounded real function $G$.

\subsection{Summary of contributions}
\begin{itemize}
\item \textbf{Conformal $e$-processes, conformal $e$-detectors, and unified
validity.} We introduce conformal $e$-processes and $e$-detectors,
obtained by aggregating locally restarted conformal mixture martingales with
deterministic restart weights, and use them to construct distribution-free online
change detectors. A single family of statistics serves both error criteria: the
total weight $\|w\|_1$ governs the probability of ever raising a false alarm and
the largest weight $\|w\|_\infty$ governs the average run length. Thus by choosing the weights appropriately, we can control either of PFA or ARL by the same statistic. In addition, our ARL-valid procedure also satisfies the optional horizon inequality \citep{ramdas2026complete} $\P_{P_0, \infty}(\tau_{Z,w}(\gamma)<\infty,\ \tau_{Z,w}(\gamma) \leq \sigma) \leq \frac{\E_{P_0,\infty}(\sigma)}{\gamma}$ for any stopping time $\sigma$; thus, when there is no change, our stopping procedure is guaranteed not to stop too early.

\item \textbf{A fixed-changepoint impossibility result.} Since our aim is not merely validity but optimality, we first identify the regime where optimality can be discussed. With the changepoint $T$ held fixed as
the stream grows, the pre-change observations become a negligible fraction of the
sample and the monitoring statistic adapts to post-change observations. We formalize this as an impossibility theorem for universal supermartingales over a composite class of distributions, which motivates the optimality analysis to be done in a growing-changepoint regime, where $T \to\infty$ along a double array.

\item \textbf{Rigorous bounds of detection delay for the PFA and ARL controlled process} We show that our PFA and ARL controlled procedures with suitably chosen weights mentioned later have detection delay of order at most $\frac{\log T}{D_{\mathrm{KL}}(P_1\|P_0)}$ and $\frac{\log \gamma}{D_{\mathrm{KL}}(P_1\|P_0)}$, respectively where $T$ is the changepoint, $\gamma$ is the ARL level. The rates vastly improve over the delay rates for the existing distribution free procedures for change detection using conformal test martingales \citep{vovk2021testing}; We provide a concrete example showing that CTM may have detection delay rate linear in $T$ for PFA regime and $\sqrt{\gamma}$ for ARL regime --- therefore our method offers exponential improvement of the delay rate while retaining distribution-free validity.
For the learned construction, the late-alarm bound contains an exponential post-change term and a separate calibration-error term, with constants and onset depending on the distribution pair.

\item \textbf{Minimax optimality of our PFA and ARL valid procedures.}
We prove the score data-processing inequality and characterize equality, which quantifies the additional detection delay due to a score transformation to the data. For a fixed pair of pre- and post-change distributions $(P_0,P_1)$ over a late-changepoint regime, we establish a first-order minimax optimality for the proposed procedure.

\item \textbf{Identifying optimal score transformations} Our method is not restricted to only real valued observations --- it applies to any objects such as text or image data. In such cases, typically the data are first reduced to real values by a score transformation, which could reduce the separation between the pre- and post-change laws and hence may increase detection delay. We quantify this loss explicitly and derive a data processing inequality for score transformed change detection and characterize the transformation which does not lose any information.
\end{itemize}

Taken together this work offers a distribution-free method for PFA and ARL controlled change
detection with sharp rates and explicit late alarm bounds and minimax optimality.

\subsection{Related Work}
\label{subsec: related works}

\paragraph{Rank-based methods.}
Nonparametric online changepoint detection has been studied extensively in the univariate setting; \citet{ross15cpmpackage} contains an overview. A central feature of the univariate setting is that distributional information is fully encoded by the empirical distribution function, and hence by ranks of the observations. This observation has led to a class of rank-based sequential procedures, including the methods of \citet{gordon94efficient,hawkins10nonparametric,romano23changedetection}. Extending these ideas beyond one dimension is less straightforward, since there is no unique scalar ordering of multivariate observations and computationally tractable analogues of ranks are difficult to define.

\paragraph{Multivariate nonparametric methods.}
Several approaches have therefore been developed specifically for multivariate online changepoint detection; see \citet{wang24sequential} for a recent survey. Some procedures reduce the data stream to summary statistics motivated by geometric entropy minimization \citep{yilmaz17online,kurt20real}. These methods can be computationally attractive, but typically require knowledge of the pre-change distribution. Other nonparametric approaches rely on interpoint distances \citep{chen19sequential,chu22sequential}. While such methods avoid parametric modeling, their test statistics need not metrize the space of probability distributions, and hence there may exist alternatives under which the change is not detected asymptotically.

\paragraph{Kernel-based methods.}
Kernel-based two-sample methods provide another approach to nonparametric changepoint detection. Characteristic kernels \citep{fukumizu07kernel,sriperumbudur10hilbert} can separate probability distributions through their kernel mean embeddings, making them sensitive to broad classes of distributional changes. In an online setting, however, the direct use of kernel two-sample statistics can be computationally expensive, since standard empirical estimators require pairwise comparisons among observations. A line of work has addressed this computational barrier through approximations, block-based statistics, and recursive updates. \citet{zaremba13btest} introduced block-based two-sample statistics that reduce the cost of kernel computations by partitioning the data into smaller blocks. Building on this idea, \citet{li15mstatistic,li19scanbstatistics} proposed online changepoint procedures based on recursive estimation over sliding windows, and \citet{wei22online} extended this strategy to a grid of window sizes. Related streaming methods, such as NEWMA \citep{keriven20newma}, use recursive smoothing and random-feature approximations to obtain scalable online statistics. These approaches improve computational scalability, but their performance depends on tuning parameters such as window lengths, block sizes, smoothing rates, or feature approximations. The most recent method in this class is the Online RFF-MMD procedure of \citet{kalinke2026optimal}, which develops an online, window-free kernel-based detector using random Fourier features.

\paragraph{Conformal and $e$-value based methods.}
The use of conformal ranks for sequential testing of exchangeability goes back to
\citet{vovk2003testing}, with more flexible plug-in betting martingales developed by
\citet{fedorova2012plugin}. Their use for changepoint detection was made explicit by
\citet{volkhonskiy2017inductive}, who proposed inductive conformal martingales for
quickest changepoint detection. \citet{vovk2021testing} gives a modern treatment of
conformal test martingales for distribution-free online testing, while
\citet{vovk2021binary,nouretdinov2021continuous,vovk2022markov} investigate their
efficiency in specific changepoint and nonexchangeable model situations. Most directly
related to the motivation of the present work, \citet{vovk2021retrain} proposed
conformal versions of the Ville, CUSUM, and Shiryaev--Roberts procedures, observing
that the ordinary Ville procedure becomes less efficient as the monitoring horizon
grows. A related practical limitation was identified by \citet{eliades2022betting}:
after a long stationary period, a conformal martingale may become very small and
therefore require a long time to recover after a change; they address this issue by
modifying the betting function. \citet{vovk2023forgetting} studies more generally the
role of forgetting observations in conformal testing. More recently,
\citet{vovk2025cusum} established theoretical validity and efficiency properties for
the conformal CUSUM procedure, and \citet{prinster2025watch} introduced
weighted-conformal test martingales for adaptive monitoring under covariate and
concept shift. The weights in the latter construction modify conformal calibration
to accommodate distribution shift, whereas the restart weights in our construction
allocate evidence across candidate changepoints.

A complementary line of work constructs changepoint detectors directly from
$e$-values and $e$-processes. The $e$-detector framework of
\citet{shin2023detectors} aggregates $e$-processes started at successive candidate
changepoints and yields Shiryaev--Roberts and CUSUM $e$-detectors with nonasymptotic
average-run-length guarantees. \citet{vovk2025conformaletesting} develops conformal $e$-testing by replacing
conformal $p$-values with conformal $e$-values, including conformal CUSUM and reverse
Shiryaev--Roberts $e$-procedures. Finally, since conformal evidence is 
adapted to the filtration generated by the conformal p-values,
\citet{choe2026combining} studies the transfer and combination of $e$-processes across
different filtrations.

\paragraph{Outline of the paper.}
The rest of the paper is organized as follows.
Section~\ref{sec:Distribution-free online changepoint detection using conformal martingales}
introduces the necessary preliminaries, discusses the limitations of existing conformal-martingale
methods, and develops our restarted conformal procedures, including the learned-score construction
and their distribution-free validity guarantees.
Section~\ref{sec:optimality} studies the statistical efficiency of the learned-score procedures,
establishing detection-delay bounds and restricted minimax optimality results for both the PFA-
and ARL-controlled regimes.
Section~\ref{sec:simulations} presents simulation studies illustrating the empirical performance
of the proposed procedures, and Section~\ref{sec:conclusion} concludes with a discussion.

The appendices contain supporting results and additional analyses.
Appendix~\ref{sec: change detection using CMM} develops the fixed-score CMM theory,
while Appendix~\ref{apd:regularity of betting class} gives sufficient regularity conditions
for the corresponding betting classes.
Appendix~\ref{apd:impossibility theorem} contains the canonical-filtration impossibility result,
Appendix~\ref{Appendix: Lemmas} collects auxiliary lemmas, and
Appendix~\ref{Appendix: Proofs} contains the proofs of the main results.
Finally, Appendix~\ref{apd:fixed-score-algorithm} gives the fixed-score implementation used
in the simulations.

\section{Background and Main Results}
\label{sec:Distribution-free online changepoint detection using conformal martingales}

\subsection{Preliminaries}
\label{sebsec: preliminaries}

Before turning to our constructions, we define all the necessary terms used in
this paper. We work throughout on a filtered probability space
$(\Omega,\mathcal B,\{\mathcal F_n\}_{n\geq0},\mathbb P)$, where $\mathcal B$ is a
$\sigma$-field on $\Omega$ and $\{\mathcal F_n\}_{n\geq0}$ is a \emph{filtration},
that is, an increasing sequence of sub-$\sigma$-fields
$\mathcal F_0\subseteq\mathcal F_1\subseteq\cdots\subseteq\mathcal B$. The online monitoring system observes $X_1,X_2,\ldots$ which are $\mathcal
B$-measurable maps from $\Omega$ into the arbitrary measurable space $\mathcal X$. A sequence of random variables
$\{Y_n\}_{n\geq1}$ is \emph{adapted} to the filtration $\{\mathcal F_n\}_{n \geq 1}$ if $Y_n$ is $\mathcal
F_n$-measurable for every $n$. A random variable
$\tau:\Omega\to\mathbb N_0\cup\{\infty\}$ is a \emph{stopping time} with respect to
$\{\mathcal F_n\}_{n\in\mathbb N}$ if $\{\tau\leq n\}\in\mathcal F_n$ for every
$n$. For a nonnegative adapted process $D$ and a stopping time $\tau$ we adopt the
convention $D_\tau:=\liminf_{m\to\infty}D_{\tau\wedge m}$ on the event
$\{\tau=\infty\}$.
 
$(X_1,\ldots,X_n)$ is \emph{exchangeable} if 
for any permutation $\sigma$ of $\{1,2,\cdots,n\}$, the random vectors $(X_1,\cdots,X_n)$ and $(X_{\sigma(1)},\cdots,X_{\sigma(n)})$ are identically distributed. In this paper, we typically assume the data sequence $\{X_n\}_{n \geq 1}$ to be consisting of independent observations and we say a change in the distribution has occured if \eqref{eq: data structure} holds. In case the joint distribution of $\{X_n\}_{n \geq 1}$ is denoted by $\P_{P_0,T,P_1}$. An adapted, integrable process $\{M_n\}_{n\geq0}$ is a \emph{martingale} with
respect to $\{\mathcal F_n\}$ under $\mathbb P$ if
$\mathbb E_{\mathbb P}[M_n\mid\mathcal F_{n-1}]=M_{n-1}$ almost surely with
respect to $\mathbb P$ (henceforth abbreviated $\mathbb P$-a.s.) for every
$n\geq1$, and a \emph{supermartingale} if
$\mathbb E_{\mathbb P}[M_n\mid\mathcal F_{n-1}]\leq M_{n-1}$ $\mathbb P$-a.s. for
every $n\geq1$; it is \emph{nonnegative} if $M_n\geq0$ $\mathbb P$-a.s. for every
$n$. Since conditional expectations are defined only up to null sets, each such
identity or inequality is asserted outside a $\mathbb P$-null set. A
\emph{test martingale} under $\mathbb P$ is a nonnegative martingale with $M_0=1$
$\mathbb P$-a.s.

Let $\mathcal P$ be a class of probability measures on $(\Omega,\mathcal B)$. A
nonnegative random variable $E$ is an \emph{$e$-value} for $\mathcal P$ if
$\mathbb E_{\mathbb P}E\leq1$ for every $\mathbb P\in\mathcal P$; A nonnegative adapted process
$\{E_n\}_{n\geq0}$ is an \emph{$e$-process} for $\mathcal P$ if
$\mathbb E_{\mathbb P}E_\tau\leq1$ for every $\mathbb P\in\mathcal P$ and every
stopping time $\tau$, and a nonnegative adapted process $\{D_n\}_{n\geq0}$ is an
\emph{$e$-detector} for $\mathcal P$, if
$\mathbb E_{\mathbb P}D_\tau\leq\mathbb E_{\mathbb P}\tau$ for every
$\mathbb P\in\mathcal P$ and every stopping time $\tau$ (\cite{shin2023detectors}). Every nonnegative
supermartingale with $M_0\leq1$ is an $e$-process, by optional stopping followed by
Fatou's lemma, but the converse is not true.
 
If $\{M_n\}_{n\geq0}$ is a nonnegative supermartingale with respect to
$\{\mathcal F_n\}$ under $\mathbb P$, then \emph{Ville's inequality} states that
$\mathbb P\{\sup_{n\geq0}M_n\geq b\}\leq\mathbb E_{\mathbb P}M_0/b$ for every
$b>0$; in particular $\mathbb P\{\exists\,n\geq1:M_n\geq b\}\leq1/b$ for a test
martingale. The same bound holds for an $e$-process $E$, since
$\tau:=\inf\{n\geq1:E_n\geq b\}$ gives
$b\,\mathbb P\{\tau<\infty\}\leq\mathbb E_{\mathbb P}E_\tau\leq1$. Finally, for probability measures $P_1$ and $P_0$ on a common space, the
Kullback--Leibler divergence is defined as
$D_{\mathrm{KL}}(P_1\|P_0):=\int\log\frac{dP_1}{dP_0}\,dP_1$ if $P_1\ll P_0$ or $+\infty$
otherwise.

\subsection{Change detection using conformal test martingale and limitations}
\label{subsec: change detection using CTM}

Before we state our construction explicitly, we first discuss the provable suboptimality of existing distribution-free online change detection method using conformal martingales, for both PFA and ARL control.   The idea of change-point detection using conformal test martingale (CTM) is introduced in \cite{vovk2021testing}. For random variables $X_1,\cdots,X_n$ defined over a probability space $(\Omega, \mathcal{B})$ taking values in an arbitrary set $\mathcal{X}$, conformal $p$-value $p_n$ based on $X_1,\cdots,X_n$ is defined as,
\begin{equation}
\label{eq: randomized pval}
    p_n(X_1,\ldots,X_n;\lambda_n)
    :=
    \frac{\sum_{i=1}^n \mathbf 1(X_i<X_n)+\lambda_n\sum_{i=1}^n\mathbf 1(X_i=X_n)}{n}.
\end{equation}
Throughout this paper, $\lambda_n$ will be referred to as \emph{randomizers} and will always be assumed that $\lambda_n\sim U(0,1)$ is independent of the data and other randomizers. Sometimes instead of computing the $p$-value directly from $X_i$s, data is transformed as $Z_i = s(X_i),$ where $s: \mathcal{X} \to \R$ is a measurable function, and the $p$-value is computed using $Z_i$s. Such functions are commonly referred as the \textit{score function}. When the score function is fixed or clear from the context, we suppress the notation $s$ and simply denote the conformal $p$-value by $p_n$. The following result is the basic distribution-free ingredient used throughout the paper.
\begin{proposition}[\cite{vovk2021testing}]
\label{prop: pval iid}
    Let $X_1,\cdots,X_n$ be exchangeable and for $k \in [n]$, define $p_k := p_k(X_1,\cdots,X_k;\lambda_k),$ where $\lambda_1,\cdots,\lambda_n \iid U(0,1)$ and $(\lambda_1,\cdots,\lambda_n) \perp (X_1,\cdots,X_n)$. Then $p_1,\cdots,p_n \iid U(0,1)$. 
\end{proposition}

Proposition~\ref{prop: pval iid} turns exchangeability into iid uniform conformal $p$-values. Throughout the main development, we use the filtration $\mathcal G_n:=\sigma(p_1,\ldots,p_n),
n\geq1,$
and all stopping times are understood with respect to $(\mathcal G_n)_{n\geq0}$ unless stated otherwise. This filtration is the natural one for conformal martingale validity.

A sequence of nonnegative Borel-measurable functions $\{f_n\}_{n\in\mathbb N}$, with $f_n:[0,1]^n\to[0,\infty]$, is a \emph{betting martingale} if, for every $n\geq1$,
\begin{equation}
    \label{eq: betting martingale def}
    \int_0^1 f_n(s_1,\ldots,s_n)\,ds_n
    =
    f_{n-1}(s_1,\ldots,s_{n-1}).
\end{equation}
For example, one may take $f_n(s_1,\ldots,s_n)=\prod_{i=1}^n f(s_i)$ for a density $f$ on $[0,1]$. If $S_n=f_n(p_1,\ldots,p_n)$, with $S_0 := 1$ then $(S_n,\mathcal G_n)_{n \geq 0}$ is a nonnegative martingale under every exchangeable no-change law, and it is called a \textit{conformal test martingale} (CTM). For a threshold $b>0$, define
\begin{equation}
    \label{eq: CTM CUSUM SR}
    \tau^{\texttt{CTM}}(b)
    :=
    \inf\{n\geq1:S_n\geq b\}.
\end{equation}
Ville's inequality gives $\sup_{P_0}\mathbb P_{P_0,\infty}\{\tau^{\texttt{CTM}}(b)<\infty\}
\leq
\frac1b.$
In particular, setting $b=1/\alpha$ one can obtain a distribution-free PFA valid change detection for the distribution shift. Based on the CTM $S_n$, define the CUSUM \citep{page1954continuous} and Shirtayev-Roberts \cite{shiryaev1961problem} procedures as,
\begin{equation}
    \label{eq:VovkSR}
    \tau^{\texttt{CUSUM}}(\gamma) := \inf \bigg\{n : \max_{i = 0,\cdots,n}\frac{S_n}{S_i} \geq \gamma \bigg\}, \quad \tau^\texttt{SR}(\gamma) := \inf \bigg\{n: \sum_{i=1}^{n-1}\frac{S_n}{S_i} \geq \gamma \bigg\}. 
\end{equation}
As shown in \citep{vovk2021testing}, $\E_{P_0, \infty}\tau^\texttt{SR}(\gamma) \geq \gamma$ and since $\tau^\texttt{CUSUM}(\gamma) \geq \tau^\texttt{SR}(\gamma),$ hence $\E_{P_0, \infty}\tau^\texttt{SR}(\gamma) \geq \gamma$. Thus both of the procedures have ARL validity at level $\gamma$. Although this construction was introduced in \cite{vovk2021testing} for distribution-free sequential testing, no optimality guarantee of the procedure was discussed. As illustrated in the following examples, change detection using CTM can lead to long delays for both the PFA and ARL regimes.

\begin{example}
\label{eg: CTM log T delay}
    Consider a growing changepoint regime in a double array $\{X_{n,j}\}_{n,j \geq 1}$ where the $n$th row of this array consists of independent observations satisfying $X_{n,j} \iid U(0,1)$ for $j \leq n$ and $X_{n,j} \iid P_1$ for $j > n$ where $P_1(dx)=f_1(x)\,dx$ with $f_\theta(u)
=
\exp\left\{\theta\left(u-\frac12\right)-\psi(\theta)\right\}$ to be one member of the one-parameter natural exponential family with sufficient statistic $u - \frac{1}{2}$, natural parameter \(\theta\), and log-partition function

$$ \psi(\theta) = \log\int_0^1 e^{\theta(u-1/2)}\,du = \begin{cases} \displaystyle \log\left(\frac{2\sinh(\theta/2)}{\theta}\right), &\theta\neq0,\\[1.2ex] 0,&\theta=0. \end{cases} $$ For a fixed $\alpha \in (0,1)$, let $\tau_n^{\texttt{CTM}}(\frac{1}{\alpha})$ be $\tau^{\texttt{CTM}}(\frac{1}{\alpha})$ as defined in \eqref{eq: CTM CUSUM SR} computed on $\{X_{n,j}\}_{j \geq 1}$ with corresponding betting martingale $f_k(p_1,\cdots,p_k) := \prod_{i = 1}^k f_1(p_i)$. Let $\kappa:=\psi(1)=\log\left(2\sinh\left(\frac12\right)\right)$
and $c_{\mathrm{CTM}}
:=
\frac{\kappa}{\frac12-\kappa}
\approx 0.090096.$
Then, for every fixed $c\in(0,c_{\mathrm{CTM}})$, as $n\to\infty$, $$\mathbb P_{P_0,n,P_1}(
\tau_n^{\texttt{CTM}}(\frac{1}{\alpha})
>
n+\lfloor cn\rfloor \mid
\tau_n^{\texttt{CTM}}(\frac{1}{\alpha})>n
)
\longrightarrow 1.
$$ 
Thus, for all deterministic nonnegative 
$r_n=o(n)$, $\mathbb P_{P_0,n,P_1}\left(
\tau_n^{\texttt{CTM}}(\frac{1}{\alpha})-n>r_n
\;\middle|\;
\tau_n^{\texttt{CTM}}(\frac{1}{\alpha})>n
\right)
\longrightarrow 1.$ In particular, for every $K>0$, $\mathbb P_T (
\frac{\tau_n^{\texttt{CTM}}(\frac{1}{\alpha})-T}{\log T}>K
\mid
\tau_n^{\texttt{CTM}}(\frac{1}{\alpha})>T)
\longrightarrow 1,$ as $T \to \infty.$ 
\end{example}

Thus, conditional on stopping after the true changepoint (at time $n$), the detection delay $\tau_n^{\texttt{CTM}}(\frac{1}{\alpha})-n$ must grow at least linearly in the changepoint location $T$ while the benchmark in the change detection literature is a delay of order $\log T$. The following example shows that the ARL-valid method using CTM can also lead to longer delay than $\log \gamma.$

\begin{example}\label{ex:arl-single-bet-separation}
Let $(\gamma_n)_{n\ge1}$ be a deterministic sequence with $\gamma_n>1$ and $\gamma_n\to\infty$. As before, Consider a growing changepoint regime in a double array $\{X_{n,j}\}_{n,j \geq 1}$ where the $n$th row $\{X_{n,j}\}_{j \geq 1}$ is as in \ref{eq: data structure}, with $P_0 \equiv U(0,1), T_n:=\lfloor\sqrt{\gamma_n}\rfloor$ and 
$P_1(dx)=h(x)\mathbf{1}_{[0,1]}(x)\,dx$ with $h(u):=\frac12+3u(1-u)$. Consider the betting density $f(u):=u+\frac12$,
 and $\tau_{n}^\texttt{CUSUM}$ and $\tau_{n}^\texttt{SR}$ be $\tau^\texttt{CUSUM}(\gamma_n)$ and $\tau^\texttt{SR}(\gamma_n)$ respectively as in \eqref{eq:VovkSR}, computed on $\{X_{n,j}\}_{j \geq 1}$. 
Then there exists a constant $c_*>0$, independent of $n$, such that
 \begin{align*}
     \mathbb P_{P_0,T_n,P_1}\!\left(\tau_n^{\texttt{SR}}(\gamma_n)>T_n+\lfloor c_*T_n\rfloor\mid\tau_n^{\texttt{SR}}(\gamma_n)>T_n\right)\to1, \\
     \liminf_{n\to\infty}\mathbb E_{P_0,T_n,P_1}\!\left[\tau_n^{\texttt{SR}}(\gamma_n)-T_n\mid\tau_n^{\texttt{SR}}(\gamma_n)>T_n\right]/\sqrt{\gamma_n}\ge c_*.
 \end{align*}
 
  In particular, for every $K>0$, $\mathbb P_{P_0,T_n,P_1}\!\left((\tau_n^{\texttt{SR}}(\gamma_n)-T_n)/\log\gamma_n>K\mid\tau_n^{\texttt{SR}}(\gamma_n)>T_n\right)\to1$. The same conclusions hold for $\tau_n^{\texttt{CUSUM}}(\gamma_n)$.
\end{example}

Thus, conditional on no alarm before or at the changepoint, the detection delay of $\tau_n^{\texttt{SR}}(\gamma_n)$ is at least of order $\sqrt{\gamma_n}$ with probability tending to one. The proof of both the examples can be found in Appendix \ref{Appendix: Proofs}. As we will show, our new procedure will improve the detection delay to be of order $\log T,$ and $\log \gamma$ for PFA and ARL valid procedures respectively, therefore offering substantial improvement over CTM.

\subsection{Our method for distribution-free change detection}
\label{subsec: our method}
We now describe the construction that underlies all the procedures in this paper. In the CTM construction, the choice of the betting density heavily influences the optimality of the corresponding detector: a betting density that bets well against one post-change law might bet poorly against another, and since $P_1$ is unknown, choosing the density optimally can be difficult. Therefore, rather than choosing a single conformal betting function in advance, we instead aggregate the evidence over a class of betting functions. This mixture construction preserves the distribution-free calibration of conformal test martingales while giving the statistic enough flexibility to respond to different departures from the null behavior of the conformal $p$-values.

Formally, let $\mathcal{F} = \{f_\theta : \theta \in \Theta\}$ be a class of density functions on $[0,1]$. Each $f_\theta$ defines a valid conformal test martingale when evaluated along the conformal $p$-value sequence. We average the resulting martingales with respect to a probability measure $\Pi$ on $\Theta$. Onwards, we refer to $\Pi$ as the \textit{prior distribution} on $\mathcal F$. This produces a single martingale that accumulates evidence across the whole class of betting functions. This leads us to formally define,  
\begin{equation}
    \label{eq: conformal mixture martingale}
    M_n \equiv M_n(p_1,\cdots,p_n) := \int_{\Theta} f_\theta(p_1)\cdots f_\theta(p_n) d\Pi(\theta),
\end{equation}
where $p_1,\cdots,p_n$ are conformal $p$-values as in Proposition~\ref{prop: pval iid}. Using Tonelli's theorem, it follows that $\{(M_n,\mathcal G_n)\}_{n\in\mathbb N}$ is a nonnegative martingale under every exchangeable no-change law, with $\mathcal G_n=\sigma(p_1,\ldots,p_n)$. We call it the \textit{conformal mixture martingale} (CMM). 
With this construction, for any threshold $b>0$, define
\begin{equation}
    \label{eq: CMM Stopping time}
    \tau^{\texttt{CMM}}(b)
    :=
    \inf\{n\geq1:M_n\geq b\}.
\end{equation}
Under no change, Ville's inequality gives $\sup_{P_0}\mathbb P_{P_0,\infty}\{\tau^{\texttt{CMM}}(b)<\infty\}
\leq
\frac1b.$
Thus for $\alpha\in(0,1)$, $\tau^{\texttt{CMM}}(\frac{1}{\alpha})$ has distribution-free PFA validity of level $\alpha$. If $b>1$, then $\mathbb P_{P_0,\infty}\{\tau^{\texttt{CMM}}(b)=\infty\}\geq1-1/b$, so its ARL is infinite. The ordinary CMM is therefore a bounded-PFA baseline rather than a finite-ARL procedure. Since $M_n$ is a non-negative test martingale, Vovk's construction can be applied in this setting to define the ARL-valid procedures as, 
\begin{equation}
    \label{eq:CMM CUSUM SR}
    \tau^{\texttt{CMM-CUSUM}}(\gamma) := \inf \bigg\{n: \max_{i=0,\cdots,n} \frac{M_n}{M_i} \geq \gamma\bigg\}, \quad \tau^{\texttt{CMM-SR}}(\gamma) := \inf\bigg\{n: \sum_{i=0}^{n-1} \frac{M_n}{M_i} \geq \gamma \bigg\},
\end{equation}
which satisfies $\E_{P_0, \infty}\tau^\texttt{CMM-CUSUM}(\gamma) \geq \E_{P_0, \infty}\tau^\texttt{CMM-SR}(\gamma) \geq \gamma$. However, as evident in the following two examples, mixing over different betting functions alone does not suffice for faster detection. 

\begin{example}\label{ex:pfa-cmm-separation}
Fix $\alpha\in(0,1)$. Consider a growing changepoint regime in a double array $\{X_{n,j}\}_{n,j\ge1}$ whose $n$th row is as in \eqref{eq: data structure}, with $P_0\equiv U(0,1)$, $T_n:=n$, and $P_1(dx)=e^x(e-1)^{-1}\mathbf{1}_{[0,1]}(x)\,dx$. In \eqref{eq: conformal mixture martingale}, take $\Theta:=\{-1/2,1\}$, $f_\theta(u):=e^{\theta u}/\int_0^1e^{\theta v}\,dv$ for $u\in[0,1]$, and $\Pi:=\frac12\delta_{-1/2}+\frac12\delta_1$. Let $\tau_n^{\texttt{CMM}}(1/\alpha)$ be $\tau^{\texttt{CMM}}(1/\alpha)$ as in \eqref{eq: CMM Stopping time}, computed on $\{X_{n,j}\}_{j\ge1}$ with this mixture. Then there exists a constant $c_*>0$, independent of $n$ and $\alpha$, such that $$\mathbb P_{P_0,T_n,P_1}(\tau_n^{\texttt{CMM}}(1/\alpha)>T_n+\lfloor c_*T_n\rfloor\mid\tau_n^{\texttt{CMM}}(1/\alpha)>T_n)\to1.$$ Consequently, for every deterministic nonnegative sequence $r_n=o(T_n)$, $$\mathbb P_{P_0,T_n,P_1}(\tau_n^{\texttt{CMM}}(1/\alpha)-T_n>r_n\mid\tau_n^{\texttt{CMM}}(1/\alpha)>T_n)\to1.$$ In particular, for every $K>0$, $\mathbb P_{P_0,T_n,P_1}((\tau_n^{\texttt{CMM}}(1/\alpha)-T_n)/\log T_n>K\mid\tau_n^{\texttt{CMM}}(1/\alpha)>T_n)\to1$.
\end{example}

\begin{example}\label{ex:arl-cmm-separation}
Let $(\gamma_n)_{n\ge1}$ be a deterministic sequence with $\gamma_n>1$ and $\gamma_n\to\infty$. Consider a growing changepoint regime in a double array $\{X_{n,j}\}_{n,j\ge1}$ whose $n$th row is as in \eqref{eq: data structure}, with $T_n:=\lfloor\sqrt{\gamma_n}\rfloor$ and the same $P_0$, $P_1$, betting family $(f_\theta)_{\theta\in\Theta}$, and fixed prior $\Pi$ as in Example~\ref{ex:pfa-cmm-separation}. Let $\tau_n^{\texttt{CMM-SR}}(\gamma_n)$ and $\tau_n^{\texttt{CMM-CUSUM}}(\gamma_n)$ be the stopping times in \eqref{eq:CMM CUSUM SR}, computed on $\{X_{n,j}\}_{j\ge1}$ using the CMM in \eqref{eq: conformal mixture martingale}, with initial value $M_0=1$. Then there exists a constant $c_*>0$, independent of $n$, such that $$\mathbb P_{P_0,T_n,P_1}(\tau_n^{\texttt{CMM-SR}}(\gamma_n)>T_n+\lfloor c_*T_n\rfloor\mid\tau_n^{\texttt{CMM-SR}}(\gamma_n)>T_n)\to1,$$ and $$\liminf_{n\to\infty}\mathbb E_{P_0,T_n,P_1}[\tau_n^{\texttt{CMM-SR}}(\gamma_n)-T_n\mid\tau_n^{\texttt{CMM-SR}}(\gamma_n)>T_n]/\sqrt{\gamma_n}\ge c_*.$$ In particular, for every $K>0$, $\mathbb P_{P_0,T_n,P_1}((\tau_n^{\texttt{CMM-SR}}(\gamma_n)-T_n)/\log\gamma_n>K\mid\tau_n^{\texttt{CMM-SR}}(\gamma_n)>T_n)\to1$. The same conclusions hold for $\tau_n^{\texttt{CMM-CUSUM}}(\gamma_n)$.
\end{example}

The above example shows that the worst-case delays for the PFA and ARL-valid processes obtained using the aggregation of different betting densities can be of order $\Omega(T)$ and $\Omega(\sqrt{\gamma})$. Mixing over $\Theta$ addresses the unknown post-change law, but not the unknown
changepoint. A martingale started at time one carries the entire pre-change
history in its capital; So when the change occurs late, the martingale has already accumulated a long pre-change history, and the post-change evidence must first overcome this accumulated dilution before the alarm can be raised. This structural limitation of CMM can lead to significantly delayed detection. To overcome this dilution, we instead accumulate local evidence for the distribution change by running the martingale from every past candidate,
\begin{equation}
    \label{eq:M_kt}
    M_{k,t} := \int_\Theta \prod_{j=k}^{t} f_\theta(p_j)\,d\Pi(\theta), \quad 1 \leq k \leq t,
\end{equation}
which measures only the evidence accrued between times $k$ and $t$. For fixed $k\geq1$, we can extend the above process to all times by setting
\[
  \Lambda^{(k)}_t :=
  \begin{cases}
    1, & t<k,\\[2pt]
    M_{k,t}, & t\geq k.
  \end{cases}
\]
By
Proposition~\ref{prop: pval iid}, joint measurability of $(\theta,s)\mapsto
f_\theta(s)$, and conditional Tonelli, the process
$\{(\Lambda^{(k)}_t,\mathcal G_t)\}_{t\geq0}$ is a nonnegative martingale with
$\Lambda^{(k)}_{k-1}=1$ under every exchangeable no-change law, in particular under each iid law $\mathbb P_{P_0,\infty}$; Thus by defining, $\Tilde{S}_t^{(w)} := \sum_{k=1}^\infty w_k \Lambda_t^{(k)}$ with the sum of weights $||w||_1 := \sum_{k=1}^\infty w_k \leq 1$, $\Tilde{S}_t^{(w)}$ is a non-negative martingale (hence an $e$-process), which leads to distribution-free PFA controlled change detection. 

On the other hand, if we take the all the weights to be one, then defining $S_t^{(w)} := \sum_{k=1}^t w_k M_{k,t} = \sum_{k=1}^t M_{k,t}$, the process $S^{(w)}$ is an $e$-detector, which leads to distribution-free ARL controlled chnage detection \citep{shin2023detectors}. Thus the choice of weights can lead to PFA or ARL control in our construction. Definition~\ref{def:weighted-restarted-statistics} specifies this aggregation.
\begin{definition}[Weighted restarted statistics and stopping times]
\label{def:weighted-restarted-statistics}
For $t\geq1$, define the weighted restarted sum and maximum
\begin{equation}
\label{eq:weighted-restarted-statistics}
S_t^{(w)}:=\sum_{k=1}^t w_kM_{k,t},
\qquad
Z_t^{(w)}:=\max_{1\leq k\leq t}w_kM_{k,t}.
\end{equation}
Set $S_0^{(w)}=Z_0^{(w)}=W_0:=0$.
For a threshold $b>0$, define
\begin{equation}
\label{eq:weighted-stopping-times}
\tau_{S,w}(b):=\inf\{t\geq1:S_t^{(w)}\geq b\},
\qquad
\tau_{Z,w}(b):=\inf\{t\geq1:Z_t^{(w)}\geq b\}.
\end{equation}
Throughout, $\inf\varnothing:=\infty$.
\end{definition}

\paragraph{Optimal score transformation.}
In case the observations take values in an arbitrary measurable space $\mathcal X$, they need to be transformed using a measurable score function $s: \mathcal{X} \to \R$ before computing the conformal $p$-values using \eqref{eq: randomized pval}. Typically, such a transformation needs to be chosen particularly, otherwise than can lead to a significant loss of information, which makes the detection problem harder. The next result quantifies this loss and provides the optimal score transformation.

\begin{proposition}[Information retained by the score]
\label{prop:score-data-processing}
Let $P_1\ll P_0$, and
$D_{\mathrm{KL}}(P_1\|P_0)<\infty$.
For a measurable $s:\mathcal X\to\mathbb R$, define the transformation
\begin{equation}\label{eq:score-transform}
T_s(x,v):=F_{s,P_0}(s(x)-)
+v\{F_{s,P_0}(s(x))-F_{s,P_0}(s(x)-)\},\qquad v\in(0,1),
\end{equation}
where $F_{s,P_0}$ is the CDF of $s(X)$ under $P_0$.
Let $H_s$ be the distribution of $T_s(X,\lambda)$ under
$X\sim P_1$, $\lambda\sim U(0,1)$ independently.
Then,
\[
D_{\mathrm{KL}}(H_s\|U)
=D_{\mathrm{KL}}(P_1\circ s^{-1}\|P_0\circ s^{-1})
\leq D_{\mathrm{KL}}(P_1\|P_0).
\]
Equality holds if and only if $dP_1/dP_0$ is
$\sigma(s)$-measurable up to $P_0$-null sets. 
\end{proposition}

Thus, for an arbitrary score function $s,$ the separation between pre- and post-change observations decreases than that in the original observations, which makes the detection problem more difficult. The information loss for a score transformation is $D_{\mathrm{KL}}(P_1\|P_0)-D_{\mathrm{KL}}(H_s\|U)$. It vanishes precisely when the likelihood ratio is measurable with respect to the score. However, in practice the likelihood ratio is unknown. So in the following section, we describe an adaptive method for change detection which does not use any knowledge of the likelihood-ratio score, yet achieves the least possible delay, as shown in the next section.

\paragraph{Learning the score from preceding ranks.} We propose a data-driven construction that estimates the unknown likelihood-ratio score from the preceding conformal ranks and re-ranks the current observation accordingly, while retaining the distribution-free validity of the detection method. Fix a measurable score $s:\mathcal X\to\mathbb R$ independently of
$P_0$, $P_1$, and $T$. Throughout this construction,
$p_t(s(X_1),\ldots,s(X_t);\lambda_t)$ denotes the original randomized
rank in \eqref{eq: randomized pval}; the scalar $p_t$ will denote
its learned transformation. The randomizers remain iid $U(0,1)$
and independent of the observations. The construction has three steps. The stream is divided into deterministic blocks. A histogram of the original ranks is fitted on each block. On the next block, this histogram determines a permutation of the rank positions, and the permutation is applied to the current rank.

First, the blocks and the histogram resolutions are fixed in advance. For $B,a\geq1$, define
\begin{equation}\label{eq:learning-K}
K_B(a):=2^{\left\lfloor\log_2\!\left[
\max\{1,\min(B^{1/8},(a/B)^{1/8})\}\right]\right\rfloor}.
\end{equation}
Starting at $a_0=1$, choose the deterministic block endpoints and
bin counts by
\begin{equation}\label{eq:learning-schedules}
(a_{j+1},K_j):=
\begin{cases}
\left(a_j+\lceil\{\log(e+a_j)\}^{1/2}\rceil,
K_{\log(e+a_j)}(a_j)\right),&\texttt{PFA},\\[1ex]
\left(a_j+\lceil\{\max(1,\log\gamma)\}^{1/2}\rceil,
K_{\max(1,\log\gamma)}(a_j)\right),&\texttt{ARL}.
\end{cases}
\end{equation}
Thus the ARL block length is $\lceil(\log\gamma)^{1/2}\rceil$
when $\gamma>e$. Neither schedule uses either unknown distribution
or the changepoint. In both regimes the blocks are short compared with the logarithmic delays of Section~\ref{sec:optimality}, so the histogram is refitted frequently. The dyadic bin count $K_j$, which sets the histogram resolution, is nondecreasing in $a_j$ and grows only slowly.

Second, for $K\geq1$, let $I_{K,v}=[(v-1)/K,v/K)$ for $v<K$ and
$I_{K,K}=[(K-1)/K,1]$. On the block $a_j\leq t<a_{j+1}$,
$j\geq1$, use the histogram of the original ranks fitted to the preceding block:
\begin{equation}\label{eq:learned-histogram}
\widehat h_j(u):=\frac{K_j}{a_j-a_{j-1}}
\sum_{v=1}^{K_j}\sum_{i=a_{j-1}}^{a_j-1}
\mathbf1\{p_i(s(X_1),\ldots,s(X_i);\lambda_i)\in I_{K_j,v}\}
\mathbf1\{u\in I_{K_j,v}\}.
\end{equation}
This nonnegative function integrates to one and is determined
before the block begins. After a change, the density of the original ranks approximates the likelihood ratio of the score laws, evaluated at the corresponding $P_0$-quantile. Hence $\widehat h_j$ is an estimate of the unknown likelihood ratio on the rank scale. Since it uses only the preceding block, it is fitted solely to post-change ranks as soon as one complete block has elapsed after the change.

Third, the histogram is converted into a reordering of the rank positions. Set $\pi_t(r)=r$ for $t<a_1$, where no histogram is yet available.
For $a_j\leq t<a_{j+1}$ with $j\geq1$, define the permutation
\begin{equation}\label{eq:learned-permutation}
\begin{split}
\pi_t(r):=1
&+\sum_{v=1}^t\mathbf1\left\{
\widehat h_j\!\left(\frac{v-1/2}{t}\right)
<\widehat h_j\!\left(\frac{r-1/2}{t}\right)\right\}\\
&+\sum_{v=1}^{r-1}\mathbf1\left\{
\widehat h_j\!\left(\frac{v-1/2}{t}\right)
=\widehat h_j\!\left(\frac{r-1/2}{t}\right)\right\},
\qquad r\in[t].
\end{split}
\end{equation}
The permutation orders the histogram heights increasingly and
breaks ties by the original rank. Positions with a larger estimated likelihood ratio therefore receive larger ranks. The permutation is known before observing
$X_t$ or drawing $\lambda_t$. The learned rank moves the original randomized rank to its permuted position:
\begin{equation}\label{eq:learned-pvalues}
\begin{split}
p_t:= {}&p_t(s(X_1),\ldots,s(X_t);\lambda_t)\\
&+\frac{
\pi_t\!\left(1+\left\lfloor t p_t(s(X_1),\ldots,s(X_t);\lambda_t)\right\rfloor\right)
-1-\left\lfloor t p_t(s(X_1),\ldots,s(X_t);\lambda_t)\right\rfloor}{t}.
\end{split}
\end{equation}
The fractional position inside the rank interval is unchanged, and only the integer rank is relocated.

Although it is defined through a permutation, \eqref{eq:learned-pvalues} is itself a randomized conformal rank. To see this, assign the current rank positions sequentially. Starting with the empty ordering, insert label $t$ at position
$1+\lfloor t p_t(s(X_1),\ldots,s(X_t);\lambda_t)\rfloor$, and
shift every preceding label at or above that position one place
to the right. This specifies every current rank position, including
when base scores tie. For $j\geq1$ and current rank position $r$,
the common learned score
\[
\widehat h_j\!\left(\frac{r-1/2}{t}\right)
+\frac{K_j}{2(a_j-a_{j-1})}\frac{r-1/2}{t}
\]
has precisely the ordering \eqref{eq:learned-permutation}:
distinct histogram heights differ by at least
$K_j/(a_j-a_{j-1})$, while the second term varies by less than
half this amount. Thus \eqref{eq:learned-pvalues} is the randomized
rank after applying a learned common score to the current rank
positions, not an unrelated evidence statistic. Because $\pi_t$ depends only on preceding original ranks, the learned ranks remain iid uniform under exchangeability and generate the same filtration, as the next lemma shows. The martingale constructions above therefore apply to them unchanged. Remark~\ref{remark:learned-score-filtration} explains why the fit must use preceding ranks rather than preceding observations.

\begin{lemma}[Predictable permutations of randomized ranks]
\label{lem:predictable-rank-rearrangement}
Let $(X_t)_{t\geq1}$ be exchangeable and let the independent
randomizers be iid $U(0,1)$. For each $t$, let $\pi_t$ be a
permutation of $[t]$ measurable with respect to
$\sigma(p_i(s(X_1),\ldots,s(X_i);\lambda_i):i<t)$.
Then the transformed ranks $(p_t)_{t \geq 1}$ in \eqref{eq:learned-pvalues} are iid
$U(0,1)$, and, up to null sets,
\[
\sigma(p_1,\ldots,p_t)
=\sigma(p_i(s(X_1),\ldots,s(X_i);\lambda_i):i\leq t).
\]
\end{lemma}

\begin{remark}[Restriction on score fitting]
\label{remark:learned-score-filtration}
The estimator \eqref{eq:learned-histogram} uses preceding original
ranks, not arbitrary information from the preceding observations.
The lemma identifies this information with the preceding
transformed-rank filtration. Predictability in the full observation
filtration alone is insufficient. For example, under a continuous
iid law, $s_t(x)=\mathbf1\{x\notin\{X_1,\ldots,X_{t-1}\}\}$
assigns zero to each preceding observation and one to $X_t$
almost surely. Its current randomized rank is
$(t-1+\lambda_t)/t$, which is not uniform for $t\geq2$.
\end{remark}

\paragraph{A finite, distribution-independent betting implementation.}It remains to choose the class of betting densities and the prior in \eqref{eq:M_kt} without knowledge of $P_0$ or $P_1$. After a change, the learned ranks have an approximately nondecreasing density, because $\pi_t$ assigns larger ranks to positions with larger histogram heights. Their divergence from uniformity is unchanged, because $\pi_t$ only permutes rank positions. It therefore suffices to bet against nondecreasing densities. Fix an effective enumeration $(\varphi_\theta)_{\theta\geq1}$ of
all continuous nondecreasing piecewise-linear functions on $[0,1]$
with rational breakpoints and rational values at those breakpoints.
Take $\Theta=\mathbb N_0$ and use the fixed class of betting densities
\begin{equation}\label{eq:nonoracle-betting-dictionary}
f_0(u):=1,\qquad
f_\theta(u):=\frac{e^{\varphi_\theta(u)}}
{\int_0^1e^{\varphi_\theta(v)}\,dv},\qquad\theta\geq1.
\end{equation}
Each normalizing integral is a finite sum of integrals of
exponentials of linear functions. To perform only finitely many
updates at every time, at restart $k$ replace the common prior
$\Pi$ in \eqref{eq:M_kt} by the deterministic probability measure
$2^{-k}\delta_0+\sum_{\theta=1}^k2^{-\theta}\delta_\theta$, which places the remaining mass on the constant density.
No new prior notation is needed: the resulting local process is
\begin{equation}\label{eq:finite-restart-mixture}
M_{k,t}:=2^{-k}+\sum_{\theta=1}^k2^{-\theta}
\prod_{i=k}^t f_\theta(p_i),\qquad M_{k,k-1}:=1.
\end{equation}
The weighted statistics and stopping times remain those of
Definition~\ref{def:weighted-restarted-statistics}. The next proposition shows that this restriction costs nothing to first order. The class attains the optimal growth rate against every nondecreasing density. On the logarithmic scale, each $M_{k,t}$ is within a constant of the product martingale of any $f_\theta$ with $\theta\leq k$.

\begin{proposition}[Approximation by a fixed betting dictionary]
\label{prop:nonoracle-betting-class}
Every density in \eqref{eq:nonoracle-betting-dictionary} is strictly
positive with bounded Lipschitz logarithm. For every nondecreasing
density $g$ on $[0,1]$ with finite entropy,
\[
\sup_{\theta\geq1}\int_0^1g(u)\log f_\theta(u)\,du
=\int_0^1g(u)\log g(u)\,du.
\]
For every $\theta\leq k\leq t$, the finite mixture satisfies
$\log M_{k,t}\geq\sum_{i=k}^t\log f_\theta(p_i)-\theta\log2$.
\end{proposition}

For the default PFA implementation, we use the exactly normalized
near-harmonic weights
\begin{equation}\label{eq:near-harmonic-weights}
w_k:=\frac1{\log(e+k-1)}-\frac1{\log(e+k)},\qquad k\geq1.
\end{equation}
They are positive and sum to one. They decay only slightly faster than the harmonic sequence, so $\log(1/w_k)$ exceeds $\log k$ only by lower-order terms; this yields the sharp delay constant in Remark~\ref{remark:near-harmonic-KL}. Other deterministic summable
nonnegative weights are also permitted. For a finite-mean ARL
implementation, take unit weights and fix $\rho\in(0,1)$. An ARL lower bound alone does not guarantee a finite expected run length under no change. Each local process is therefore mixed with the constant density with weight $\rho$.
For general weights, we use the same modification defined by
\begin{equation}\label{eq:learned-null-regularization}
\begin{split}
\overline M_{k,t}^{\rho}&:=\rho+(1-\rho)M_{k,t},\\
\overline S_t^{(w,\rho)}&:=\sum_{k=1}^t w_k\overline M_{k,t}^{\rho}
=\rho W_t+(1-\rho)S_t^{(w)},\\
\Bar{\tau}_{S,w}(b)&:=\inf\{t\geq1:\overline S_t^{(w,\rho)}\geq b\}.
\end{split}
\end{equation}
The constant density ensures $\overline S_t^{(w,\rho)}\geq\rho W_t$.
In particular, with unit weights the regularized alarm is at most
$\lceil\gamma/\rho\rceil$ on every sample path. The next theorem shows that the resulting procedures are therefore valid under every exchangeable no-change law, with the total weight controlling the PFA and the largest weight controlling the ARL.

\begin{theorem}[Distribution-free Unified PFA and ARL validity]
\label{th:weighted-null-validity}
Suppose $(X_j)_{j\geq1}$ is exchangeable and the randomizers are iid uniform and independent of the observations. Let $\mathbb P$ denote this joint null law. Consider $p_t$ as defined in \eqref{eq:learned-pvalues} and $M_{k.t}$ as in \eqref{eq:finite-restart-mixture}. Let $w=(w_k)_{k\geq1}$ be deterministic with $w_k\in[0,\infty)$, and let $b>0$. Write $\|w\|_1:=\sum_{k=1}^{\infty}w_k$, $\|w\|_\infty:=\sup_{k\geq1}w_k$, $W_t:=\sum_{k=1}^t w_k$, and $W_\infty:=\|w\|_1$. All stopping times below are with respect to $(\mathcal G_t)_{t\geq0}$, with $\mathcal G_0$ trivial. Then the following holds.
\begin{enumerate}
\item If $\|w\|_1<\infty$, the process $\widetilde S_t^{(w)}
:=
S_t^{(w)}+\sum_{k>t}w_k,
\hspace{.2cm}t\geq0,$
is a nonnegative martingale with initial value $\widetilde S_0^{(w)}=\|w\|_1$. Consequently,
\begin{equation}
\label{eq:weighted-PFA-bound}
\mathbb P\{\tau_{S,w}(b)<\infty\}
\leq
\frac{\|w\|_1}{b},
\qquad
\mathbb P\{\tau_{Z,w}(b)<\infty\}
\leq
\frac{\|w\|_1}{b}.
\end{equation}
\item For every weight sequence and every stopping time $\tau$ satisfying $\mathbb E_{\mathbb P}\tau<\infty$,
\begin{equation}
\label{eq:weighted-exact-edetector}
\mathbb E_{\mathbb P}S_\tau^{(w)}
\leq
\mathbb E_{\mathbb P}W_\tau,
\qquad
\mathbb E_{\mathbb P}Z_\tau^{(w)}
\leq
\mathbb E_{\mathbb P}W_\tau,
\end{equation}
where $\mathbb E_{\mathbb P}W_\tau$ is allowed to be infinite. Consequently, for any stopping time $\sigma$,
\begin{equation}
\label{eq:weighted-finite-horizon-bound}
\begin{split}
    \mathbb P\{\tau_{S,w}(b)<\infty,\ \tau_{S,w}(b)\le\sigma\} \le \frac{\mathbb E_{\mathbb P}W_{\tau_{S,w}(b)\wedge\sigma}}{b},\\
\mathbb P\{\tau_{Z,w}(b)<\infty,\ \tau_{Z,w}(b)\le\sigma\} \le \frac{\mathbb E_{\mathbb P}W_{\tau_{Z,w}(b)\wedge\sigma}}{b}.
\end{split}
\end{equation}
If, in addition, $0<\|w\|_\infty<\infty$, then
$\mathbb E_{\mathbb P}W_\tau
\leq
\|w\|_\infty\mathbb E_{\mathbb P}\tau,$
so $S^{(w)}/\|w\|_\infty$ and $Z^{(w)}/\|w\|_\infty$ are e-detectors, and
\begin{equation}
\label{eq:weighted-ARL-bound}
\mathbb E_{\mathbb P}\tau_{S,w}(b)
\geq
\frac{b}{\|w\|_\infty},
\qquad
\mathbb E_{\mathbb P}\tau_{Z,w}(b)
\geq
\frac{b}{\|w\|_\infty}.
\end{equation}
In this case, \eqref{eq:weighted-finite-horizon-bound} further implies, for every stopping time $\sigma$,
$
\mathbb P\{\tau_{S,w}(b)<\infty,\ \tau_{S,w}(b)\leq \sigma\}
\vee
\mathbb P\{\tau_{Z,w}(b)<\infty,\ \tau_{Z,w}(b)\leq \sigma\}
\leq
\frac{\|w\|_\infty\,\mathbb E_{\mathbb P}[\sigma]}{b}.
$
\end{enumerate}
All conclusions hold in particular under every iid no-change law $\mathbb P_{P_0,\infty}$.\end{theorem}

Thus, from \eqref{eq:weighted-finite-horizon-bound}, setting $\sigma$ to be any natural number yields that, when there is no change in the data distribution, our procedure satisfies the optional horizon inequality and therefore is guaranteed not to raise a false alarm too early. When $||w||_1 \leq 1,$ the resulting weighted process $\Tilde{S}^{(w)}$ is a test martingale, and hence an $e$-process. We call this a \textbf{conformal $e$-process}, just to distinguish it from Vovk's conformal test martingales. When $||w||_\infty \leq 1$, then the resulting weighted process is an $e$-detector following the definition of \cite{shin2023detectors}. We call that to be \textbf{conformal $e$-detector}. The above results provide the distribution-free validity of our change detectors. In the next section, we discuss their optimality properties.

\begin{algorithm2e}[h!]
\LinesNumbered
\footnotesize
\caption{Learned-score conformal change detection in the PFA and ARL regimes}
\label{alg: RCMM}
\KwIn{Observations $X_1,X_2,\ldots$; a fixed score $s$;
the dictionary \eqref{eq:nonoracle-betting-dictionary}; either
$\texttt{PFA}$ with $\alpha\in(0,1)$ and deterministic
$w_k\geq0$, $\sum_kw_k\leq1$, or $\texttt{ARL}$ with
$\gamma\geq1$ and $\rho\in(0,1)$.}
\KwOut{$\tau_{S,w}(b)$ and $\tau_{Z,w}(b)$; in the ARL regime,
also $\Bar{\tau}_{S,w}(\gamma)$.}
\uIf{the regime is $\texttt{PFA}$}{Set $b\leftarrow1/\alpha$.
Use \eqref{eq:near-harmonic-weights} unless other weights are specified\;}
\Else{Set $b\leftarrow\gamma$ and $w_k\leftarrow1$ for every $k\geq1$\;}
Set $a_0\leftarrow1$, $j\leftarrow0$, and compute $a_1,K_0$
using \eqref{eq:learning-schedules}. Initialize every requested alarm to $\infty$\;
\For{$t=1,2,\ldots$}{
\If{$t=a_{j+1}$}{Set $j\leftarrow j+1$ and compute $a_{j+1},K_j$
using \eqref{eq:learning-schedules}. Fit $\widehat h_j$ from the
preceding block using \eqref{eq:learned-histogram}\;}
Before observing $X_t$, compute $\pi_t$ from
\eqref{eq:learned-permutation}; use the identity permutation when $j=0$\;
Observe $X_t$ and independently draw $\lambda_t\sim U(0,1)$.
Compute and store $p_t(s(X_1),\ldots,s(X_t);\lambda_t)$
using \eqref{eq: randomized pval}\;
Compute the scalar $p_t$ from \eqref{eq:learned-pvalues}\;
Initialize $G_{t,t-1}(\theta)\leftarrow1$ for $1\leq\theta\leq t$\;
\For{$k=1,\ldots,t$}{
\For{$\theta=1,\ldots,k$}{Set
$G_{k,t}(\theta)\leftarrow G_{k,t-1}(\theta)f_\theta(p_t)$\;}
Set $M_{k,t}\leftarrow2^{-k}+\sum_{\theta=1}^k2^{-\theta}G_{k,t}(\theta)$\;}
Compute $S_t^{(w)}$ and $Z_t^{(w)}$ using
\eqref{eq:weighted-restarted-statistics}\;
\lIf{$S_t^{(w)}\geq b$ and $\tau_{S,w}(b)=\infty$}{Record and report $\tau_{S,w}(b)\leftarrow t$}
\lIf{$Z_t^{(w)}\geq b$ and $\tau_{Z,w}(b)=\infty$}{Record and report $\tau_{Z,w}(b)\leftarrow t$}
\If{the regime is $\texttt{ARL}$}{Set
$\overline S_t^{(w,\rho)}\leftarrow\rho t+(1-\rho)S_t^{(w)}$\;
\lIf{$\overline S_t^{(w,\rho)}\geq\gamma$ and
$\Bar{\tau}_{S,w}(\gamma)=\infty$}{Record and report $\Bar{\tau}_{S,w}(\gamma)\leftarrow t$}}
\lIf{all requested alarms are finite}{Stop and return the recorded alarm times}
}
\end{algorithm2e}

\section{Optimality properties of learnt score procedure}
\label{sec:optimality}
We now discuss the detection delay and minimax optimality of our procedures. The two fundamental questions often addressed assessing efficiency of any detector is, after a distribution change occurs, whether it can detect the change (consistency of a detector) and how long afterwards the alarm is raised (detection delay). Since the alarm is raised only when the monitoring statistic crosses the threshold depending on PFA or ARL level, consequently to detect the change at any level, the statistic needs to be arbitrary large in finite time.

However, in the usual setup, when the data satisfies (\ref{eq: data structure}), we can not immediately expect the statistic to be arbitrarily large almost surely. Because, as there is a finite number of observations before the distribution shift, the proportion of the observations from the pre-change distribution $P_0$ decreases over time, and eventually their influence on the overall sample becomes negligible. As a result, the statistic eventually behaves as though the observations were identically distributed as $P_1$, thereby diminishing the effect of the distribution shift. The next result formalizes this impossibility and thus motivates a growing changepoint regime. 

\subsection{An impossibility result} For a PFA-valid detector $\tau$, the consistency of the detector for a fixed changepoint requires $\mathbb{P}_{P_0,T,P_1}(\tau\le T+r\mid\tau>T)\to1$ as $r\to\infty$, for every fixed $T\in\mathbb{N}$ and every distinct pair $P_0,P_1$. The following proposition shows that this requirement cannot hold for a fixed changepoint $T$, when $P_0, P_1$ lie in a composite class of distributions. The result applies to arbitrary PFA-valid detectors, rather than only to conformal procedures or full-filtration supermartingales.

\begin{proposition}[Impossibility of universal consistency at a fixed changepoint]
\label{prop:fixed-changepoint-impossibility}
Let $\M$ be a convex class of probability distributions on $\mathcal{X}$ containing at least two distinct distributions, and fix $T\in\mathbb{N}$.

\begin{enumerate}
    \item Fix $\alpha\in[0,1)$, and let $\tau$ be a stopping time satisfying
    $
    \sup_{P\in\M}\mathbb{P}_{P,\infty}(\tau<\infty)\le\alpha.
    $
    Then, for every $P_0\in\M$,
    $
    \inf_{P_1\in\M\setminus\{P_0\}}
    \mathbb{P}_{P_0,T,P_1}(\tau<\infty)\le\alpha.
    $
    Moreover, for every $P_0,P_1\in\M$ such that $P_1\ll P_0$,
    $
    \mathbb{P}_{P_0,T,P_1}(\tau<\infty)<1.
    $

    \item Let $(E_n)_{n\ge0}$ be a nonnegative adapted process satisfying
    $
    \mathbb{E}_{P,\infty}[E_\sigma]\le1
    $
    for every $P\in\M$ and every bounded stopping time $\sigma$. Then, for every $P_0,P_1\in\M$ such that $P_0\ll P_1$,
    $
    \mathbb{P}_{P_0,T,P_1}\left(\sup_{n\ge0}E_n<\infty\right)=1.
    $
\end{enumerate}
\end{proposition}

Thus, for every fixed $P_0\in\M$, no distribution-free PFA-valid detector has power one against every $P_1\in\M\setminus\{P_0\}$ at a fixed changepoint, and no e-process valid under all no-change laws can be almost surely unbounded against every such alternative. The obstruction is that $\mathbb{P}_{P_0,T,P_1}$ and the no-change law $\mathbb{P}_{P_1,\infty}$ differ only in their first $T$ observations. Increasing the monitoring horizon supplies more observations from $P_1$, but no additional observations from the unknown pre-change law $P_0$. 

\begin{remark}
    \label{remark: impossibility theorem}
    The full data filtration does not circumvent this obstruction for universal test supermartingales as well; Theorem~\ref{th: impossibility theorem} in Appendix~\ref{apd:impossibility theorem} shows that any process $M$ which is a non-negative supermartingale with respect to the canonical filtration $\sigma(X_1,\cdots,X_n)$ under every $P\in\M$ remains a supermartingale under the fixed-changepoint alternative, thus $M_n \to M_\infty < \infty$ almost surely under the fixed changepoint alternative and in particular 
$\mathbb P_{P_0,T,P_1}
\left\{
\sup_{n\geq0}M_n\geq\frac1\alpha
\right\}
\leq\alpha.$
Thus, a level-$\alpha$ alarm based on a universal test supermartingale with respect to the canonical filtration has detection probability at most $\alpha$ under every fixed changepoint alternative. We therefore work with the reduced filtration of the conformal $p$-values, $\mathcal{G}_n:=\sigma(p_1,\ldots,p_n)$.
\end{remark}

This motivates studying consistency and detection delay for fixed separated alternatives while allowing the pre-change sample size to grow. We therefore take the data in the form of a double array $\{X_{n,j}\}_{n,j
\in \N}$, the $n$th row being an online sample with changepoint $T_n$, where
$T_n \in \N$ satisfies $T_n \to \infty$ as $n \to \infty$. This, in particular, ensures that the contribution of the pre-change observations remains non-negligible relative to the entire sample. 

\paragraph{Double-array setup.}
For an increasing sequence $\{T_n\}_{n\geq1}$ of natural numbers with $T_n\to\infty$, we assume that
\begin{equation}
\label{eqn:data structure}
X_{n,j}\stackrel{\mathrm{iid}}{\sim}P_0,\quad j\leq T_n;
\qquad
X_{n,j}\stackrel{\mathrm{iid}}{\sim}P_1,\quad j>T_n, \quad (X_{n,1},\cdots,X_{n,T_n}) \perp \{X_{n,j}\}_{j \geqslant T_n + 1}.
\end{equation}

Onwards, each statistic discussed so far will be applied to the $n$th row of the array to illustrate the case when the changepoint of the online data lies at $T_n$. For example, $S_{n,t}^{(w)}$ and $Z_{n,t}^{(w)}$ denote $S_{t}^{(w)}$ and $Z_{t}^{(w)}$ computed on $\{X_{n,j}\}_{j \geq 1}$ following definition~\ref{def:weighted-restarted-statistics} and define
\[
\tau_{n,S,w}(b):=\inf\{t\geq1:S_{n,t}^{(w)}\geq b\},
\qquad
\tau_{n,Z,w}(b):=\inf\{t\geq1:Z_{n,t}^{(w)}\geq b\}.
\]

\subsection{Upper bounds on detection delay} We first state the learned-score upper bound for general nonnegative weights under explicit deterministic learning-scale conditions. These conditions ensure that the block length is negligible relative to the effective boundary and that prefix contamination remains small at the chosen histogram resolution. Polynomial or near-harmonic PFA weights and unit ARL weights satisfy the conditions in the regimes below. The later minimax statements use a fixed pair and a restricted interval of late changepoints. 

The general bound is stated in terms of two
quantities: the effective boundary $\mathcal B_{n,w}(b_n)$, where $b_n$ is the threshold the
process must cross to raise the alarm, and the pre-change survival lower bound $q_{n,w}(b_n)$,
namely
\begin{equation}
\label{eq:weighted-survival-lower}
\mathcal B_{n,w}(b_n)
:=
\log\frac{b_n}{w_{T_n+1}},
\hspace{.3cm}
q_{n,w}(b_n)
:=
1-\min\left\{1,\frac{W_{T_n}}{b_n}\right\}
\leq
\mathbb P_{P_0,T_n,P_1}\{\tau_{n,S,w}(b_n)>T_n\},
\end{equation}
where the displayed inequality holds for $\tau_{n,Z,w}$ as well. The following result states the detection delay bound for a general weight sequence $w$.

\begin{theorem}[Delay upper bound for a general weight sequence]
\label{th:weighted-KL-upper}
Consider \eqref{eqn:data structure} with $P_1\ll P_0$ and
$0<D_{\mathrm{KL}}(P_1\|P_0)<\infty$, where $H=H_s$ is the distribution of $T_s(X,\lambda)$ as defined
by \eqref{eq:score-transform}. Use
\eqref{eq:learned-pvalues}, \eqref{eq:nonoracle-betting-dictionary},
and \eqref{eq:finite-restart-mixture}. Let $w_k\geq0$ be
deterministic, $w_{T_n+1}>0$, and
$\mathcal B_{n,w}(b_n)\to\infty$.
For the PFA schedule in \eqref{eq:learning-schedules}, suppose
\[
(\log T_n)^{1/2}=o\{\mathcal B_{n,w}(b_n)\},\qquad
\mathcal B_{n,w}(b_n)(\log T_n)^{1/8}=o(T_n).
\]
Alternatively, for the ARL schedule with $\gamma=b_n$, suppose
\[
\log b_n\to\infty,\quad \log b_n=o(T_n),\quad
(\log b_n)^{1/2}=o\{\mathcal B_{n,w}(b_n)\},\quad
\mathcal B_{n,w}(b_n)K_{\log b_n}(2T_n)=o(T_n).
\]
Then, for every $A>1/D_{\mathrm{KL}}(P_1\|P_0)$,
\[
\mathbb P_{P_0,T_n,P_1}\{
\tau_{n,Z,w}(b_n)\leq T_n+A\mathcal B_{n,w}(b_n)\}\to1.
\]
If $\liminf_nq_{n,w}(b_n)>0$, then
\[
\mathbb P_{P_0,T_n,P_1}\{
\tau_{n,Z,w}(b_n)>T_n+A\mathcal B_{n,w}(b_n)
\mid\tau_{n,Z,w}(b_n)>T_n\}\to0
\]
and
\[
\liminf_n\mathbb P_{P_0,T_n,P_1}\{
T_n<\tau_{n,Z,w}(b_n)\leq T_n+A\mathcal B_{n,w}(b_n)\}
\geq\liminf_nq_{n,w}(b_n).
\]
All conclusions hold for $\tau_{n,S,w}(b_n)$ as well.
\end{theorem}

 Thus, if $\liminf_n q_{n,w}(b_n)>0$, the corresponding conditional late-alarm probabilities tend to zero and the correct-window probabilities have liminf at least $\liminf_n q_{n,w}(b_n)$.

The above result is for general weight sequence $w := (w_k)_{k \geq 1}$. Therefore, choosing the weights appropriately and satisfying $\sum_{k=1}^\infty w_k \leq 1$ or $\| w\|_\infty \leq 1$ yields the corresponding delay bounds for the PFA or ARL controlled processes, respectively, outlined in the next two results.

\begin{corollary}[Detection delay for the PFA-controlled process]
\label{cor:RCMM-pointwise-KL}
Assume \eqref{eqn:data structure}, $P_1\ll P_0$, and
$0<D_{\mathrm{KL}}(P_1\|P_0)<\infty$, with $H=H_s$ from
\eqref{eq:score-transform}. Use the PFA schedule and the learned
construction in Algorithm~\ref{alg: RCMM}. Let
$w_k\propto k^{-(1+\eta)}$, $\eta>0$, with
$0<\sum_{k\geq1}w_k\leq1$, and fix $\alpha\in(0,1)$.
For every $A>(1+\eta)/D_{\mathrm{KL}}(P_1\|P_0)$ and either
$\tau_n=\tau_{n,Z,w}(1/\alpha)$ or
$\tau_n=\tau_{n,S,w}(1/\alpha)$,
\[
\mathbb P_{P_0,T_n,P_1}\{\tau_n\leq T_n+A\log T_n\}\to1,
\]
\[
\liminf_n\mathbb P_{P_0,T_n,P_1}\{T_n<\tau_n\leq T_n+A\log T_n\}
\geq1-\alpha,
\]
and
\[
\mathbb P_{P_0,T_n,P_1}\{\tau_n>T_n+A\log T_n\mid\tau_n>T_n\}\to0.
\]
\end{corollary}

Thus, for appropriately chosen weights, the detection delay for the PFA-controlled process is of order $O_{\mathbb P}\!\left(\frac{\log T_n}{D_{\mathrm{KL}}(P_1\|P_0)}\right).$ The following choices of weights can improve the constant in the delay from $\frac{1+\eta}{D_{\mathrm{KL}}(P_1\|P_0)}$ to $\frac{1}{D_{\mathrm{KL}}(P_1\|P_0)}$.
\begin{remark}[Near-harmonic restart weights]
\label{remark:near-harmonic-KL}
Under the distributional and construction assumptions of
Corollary~\ref{cor:RCMM-pointwise-KL}, replace its weights by
\eqref{eq:near-harmonic-weights}. For every
$A>1/D_{\mathrm{KL}}(P_1\|P_0)$ and either
$\tau_n=\tau_{n,Z,w}(1/\alpha)$ or
$\tau_n=\tau_{n,S,w}(1/\alpha)$,
\[
\mathbb P_{P_0,T_n,P_1}\{
\tau_n\leq T_n+A(\log T_n+2\log\log T_n)\}\to1,
\]
\[
\liminf_n\mathbb P_{P_0,T_n,P_1}\{
T_n<\tau_n\leq T_n+A(\log T_n+2\log\log T_n)\}\geq1-\alpha,
\]
and the probability of exceeding this boundary conditional on
$\tau_n>T_n$ tends to zero. The same statements hold for
$w_k=c_{\mathrm{nh}}[k\{\log(e+k)\}^2]^{-1}$ with the previous
normalization. In either case the first-order constant is
$1/D_{\mathrm{KL}}(P_1\|P_0)$ on the $\log T_n$ scale.
\end{remark}

The next result establishes the upper bound on detection delay on the ARL-valid procedure of Algorithm~\ref{alg: RCMM} with appropriately chosen weights.
\begin{corollary}[Detection delay for the ARL-controlled process]
\label{cor:ARL-KL-sharp}
Assume \eqref{eqn:data structure}, $P_1\ll P_0$, and
$0<D_{\mathrm{KL}}(P_1\|P_0)<\infty$, with $H=H_s$ from
\eqref{eq:score-transform}. Use the ARL schedule and learned
construction in Algorithm~\ref{alg: RCMM}, with $w_k=1$.
Suppose $\gamma_n\to\infty$ and $\log\gamma_n=o(T_n)$.
For every $A>1/D_{\mathrm{KL}}(P_1\|P_0)$ and either
$\tau_n=\tau_{n,Z,w}(\gamma_n)$ or
$\tau_n=\tau_{n,S,w}(\gamma_n)$,
\[
\mathbb P_{P_0,T_n,P_1}\{\tau_n\leq T_n+A\log\gamma_n\}\to1.
\]
If also $T_n=o(\gamma_n)$, then
\[
\mathbb P_{P_0,T_n,P_1}(\tau_n>T_n)\geq1-T_n/\gamma_n\to1,
\]
\[
\mathbb P_{P_0,T_n,P_1}\{T_n<\tau_n\leq T_n+A\log\gamma_n\}\to1,
\quad
\mathbb P_{P_0,T_n,P_1}\{\tau_n>T_n+A\log\gamma_n\mid\tau_n>T_n\}\to0.
\]
For fixed $\rho\in(0,1)$, the regularized sum alarm in
\eqref{eq:learned-null-regularization} further satisfies
\[
\limsup_n\frac{\mathbb E_{P_0,T_n,P_1}[
\Bar{\tau}_{n,S,w}(\gamma_n)-T_n\mid
\Bar{\tau}_{n,S,w}(\gamma_n)>T_n]}{\log\gamma_n}
\leq\frac1{D_{\mathrm{KL}}(P_1\|P_0)}
\]
when $T_n=o(\gamma_n)$. Under \eqref{eq:restricted-arl-regime},
this bound is uniform over the specified late-change interval:
\[
\limsup_{\gamma\to\infty}\frac{
\sup_{T\in I_{N_\gamma}(\kappa)}
\mathbb E_{P_0,T,P_1}[\Bar{\tau}_{S,w}(\gamma)-T
\mid\Bar{\tau}_{S,w}(\gamma)>T]}{\log\gamma}
\leq\frac1{D_{\mathrm{KL}}(P_1\|P_0)}.
\]
\end{corollary}

Thus the detection delay for the ARL-controlled process with appropriately chosen weights is of order $O_\P\bigg(\frac{\log \gamma_n}{D_{\mathrm{KL}}(P_1\|P_0)}\bigg)$. As shown below, these bounds are first-order sharp: they match the corresponding information-theoretic lower bounds, and the proposed procedure attains the optimal leading detection-delay constant.

\subsection{A restricted minimax guarantee for PFA}

To formulate the minimax optimality of the proposed change detection method, we
need to understand the best detection window achievable by any valid detection
procedure. Let $P_1\ll P_0$ on $\mathcal X$ and write $L:=\frac{dP_1}{dP_0}$ and
$D:=D_{\mathrm{KL}}(P_1\|P_0)=\mathbb E_{P_1}\log L(X)\in(0,\infty)$; Let $\mathfrak T_{\rm raw}$ and $\mathfrak T_{\rm rand}$ be the set of all stopping times with respect to the filtrations $\mathcal A_t:=\sigma(X_1,\ldots,X_t)$ and $\mathcal
H_t:=\sigma(X_1,\ldots,X_t,\lambda_1,\ldots,\lambda_t)$ respectively, where the $\lambda_t$ are
external randomizers. Finally, fix $\alpha,\delta,\kappa\in(0,1)$ and let
$I_N(\kappa):=\{\lceil\kappa N\rceil,\ldots,N\}$ be the range of late
changepoints.

Take any procedure $\tau$ that is PFA-valid at level $\alpha$. Its required
detection window is the smallest $r$ such that, uniformly over every changepoint
$T\in I_N(\kappa)$, the procedure detects within $r$ observations with probability
at least $1-\delta$, namely $\min\{r\in\mathbb N:\inf_{T\in I_N(\kappa)}\mathbb
P_{P_0,T,P_1}\{T<\tau\leq T+r\}\geq1-\delta\}$. Minimizing this window over all
PFA-valid procedures gives the restricted minimax detection window: for
$\mathfrak T\in\{\mathfrak T_{\rm raw},\mathfrak T_{\rm rand}\}$,
\begin{align}
\label{eq:restricted-minimax-radius}
d_N^*(\mathfrak T;\alpha,\delta,\kappa)
:={}&
\inf_{\substack{\tau\in\mathfrak T:\\
\mathbb P_{P_0,\infty}(\tau<\infty)\leq\alpha}}
\min\left\{r\in\mathbb N:
\inf_{T\in I_N(\kappa)}
\mathbb P_{P_0,T,P_1}\{T<\tau\leq T+r\}
\geq1-\delta
\right\},
\end{align}
where the minimum of the empty set is infinity. The next result gives the
lower bound of $d_N^*$ as $N\to\infty$, and shows that the proposed method of PFA-valid change detection in algorithm~\ref{alg: RCMM} attains the lower bound in first order, thus establishing the first order restricted minimax optimality of our method for the late changepoints.



\begin{theorem}[Restricted PFA minimaxity for late changepoints]
\label{th:restricted-PFA-minimax}
Consider \eqref{eq: data structure}. Fix
$\alpha,\delta,\kappa\in(0,1)$ and $P_1\ll P_0$ with $0<D<\infty$.
\begin{enumerate}
\item For both nonrandomized and randomized stopping times,
\[
\liminf_{N\to\infty}\frac{d_N^*(\mathfrak T_{\rm raw};\alpha,\delta,\kappa)}{\log N}
\geq\frac{1}{D_{\mathrm{KL}}(P_1\|P_0)},\qquad
\liminf_{N\to\infty}\frac{d_N^*(\mathfrak T_{\rm rand};\alpha,\delta,\kappa)}{\log N}
\geq\frac{1}{D_{\mathrm{KL}}(P_1\|P_0)}.
\]
\item Suppose $\mathcal X$ is standard Borel, and fix a Borel
injection $s:\mathcal X\to\mathbb R$ independently of the
unknown distributions. Choose $0<\bar\alpha\leq\alpha$ with
$\bar\alpha<\delta$. Run Algorithm~\ref{alg: RCMM} in the PFA
regime at level $\bar\alpha$, with \eqref{eq:near-harmonic-weights}.
For either $\tau=\tau_{Z,w}(1/\bar\alpha)$ or
$\tau=\tau_{S,w}(1/\bar\alpha)$, its PFA is at most $\bar\alpha$
under every iid no-change law, and for every $c>1/D$,
\[
\sup_{T\in I_N(\kappa)}\mathbb P_{P_0,T,P_1}
\{\tau>T+\lceil c\log N\rceil\}\longrightarrow0.
\]
Consequently,
\[
\frac{d_N^*(\mathfrak T_{\rm rand};\alpha,\delta,\kappa)}{\log N}
\longrightarrow\frac{1}{D_{\mathrm{KL}}(P_1\|P_0)},
\]
and either learned conformal alarm attains this limit to first order.
\end{enumerate}
\end{theorem}

\subsection{A restricted minimax guarantee for ARL}
We now formulate the corresponding optimality question for the ARL regime.
Retain the simple-pair setup, $L$, $D$, $I_N(\kappa)$, and the
stopping-time classes introduced before
Theorem~\ref{th:restricted-PFA-minimax}. In the PFA regime, the constraint
$\mathbb P_{P_0,\infty}(\tau<\infty)\leq\alpha$ already restricts the probability
of an alarm at any early time, since any alarm under no change is a false alarm.
The ARL constraint does not have this consequence: a procedure may stop at an
early time with substantial probability and still have a large average run length,
provided it runs sufficiently long on the remaining paths. A minimax statement
under the ARL constraint alone would therefore admit procedures whose small delay
is attributable to frequent alarms raised before the change. An additional
restriction on the false-alarm probability at every finite horizon is thus
required, and the unit-weight bounds in Theorem~\ref{th:weighted-null-validity}
provide one of the linear form $m/\gamma$. Accordingly, we impose both an ARL
constraint and linear control of early false alarms: for $\gamma>1$ and
$\mathfrak T\in\{\mathfrak T_{\rm raw},\mathfrak T_{\rm rand}\}$, define
\begin{equation}
\label{eq:restricted-arl-class}
\mathcal C_{\gamma}^{\rm lin}(\mathfrak T)
:=
\left\{
\tau\in\mathfrak T:
\gamma\leq\mathbb E_{P_0,\infty}\tau<\infty,\quad
\mathbb P_{P_0,\infty}(\tau\leq m)\leq\frac{m}{\gamma}
\quad\forall m\in\mathbb N
\right\}.
\end{equation}
This class contains the regularized unit-weight sum in \eqref{eq:learned-null-regularization}. Its constant-density component ensures finite null expectation; an ARL lower bound alone does not establish this membership for every unregularized alarm. The linear false-alarm condition prevents small delay achieved merely through frequent pre-change alarms. Fix $\kappa\in(0,1)$
and let $N_\gamma\in\mathbb N$ satisfy
\begin{equation}
\label{eq:restricted-arl-regime}
N_\gamma\longrightarrow\infty,
\qquad
\log\gamma=o(N_\gamma),
\qquad
N_\gamma=o(\gamma).
\end{equation}
For $N_\gamma<\gamma$, which holds eventually, define
\begin{align}
\label{eq:restricted-arl-risk}
\mathcal J_{\gamma}^{\kappa}(\tau)
&:=
\sup_{\nu\in I_{N_\gamma}(\kappa)}
\mathbb E_{P_0,\nu,P_1}[\tau-\nu\mid\tau>\nu],
\qquad \tau\in\mathcal C_{\gamma}^{\rm lin}(\mathfrak T),\\
\label{eq:restricted-arl-minimax}
\mathcal R_{\gamma}^{\rm lin}(\mathfrak T;\kappa)
&:=
\inf_{\tau\in\mathcal C_{\gamma}^{\rm lin}(\mathfrak T)}
\mathcal J_{\gamma}^{\kappa}(\tau).
\end{align}
The risk \eqref{eq:restricted-arl-risk} is conditioned on survival to the
changepoint, since a procedure that has already stopped cannot be regarded as
having detected the change, and the conditioning probabilities are at least
$1-N_\gamma/\gamma>0$. As in the PFA case, the changepoint is restricted to the
late window $I_{N_\gamma}(\kappa)$, and the two scale restrictions in
\eqref{eq:restricted-arl-regime} make a logarithmic detection window negligible
relative to the calibration sample and make pre-change false alarms asymptotically
negligible. The next result gives a lower bound on the minimax risk
$\mathcal R_{\gamma}^{\rm lin}(\mathfrak T;\kappa)$ and shows that our ARL
controlled procedure attains it to first order.

\begin{theorem}[Restricted ARL minimaxity for late changepoints]
\label{thm:restricted-arl-lower}
Assume \eqref{eq:restricted-arl-regime} and $P_1\ll P_0$ with
$0<D<\infty$.
\begin{enumerate}
\item For either $\mathfrak T\in\{\mathfrak T_{\rm raw},\mathfrak T_{\rm rand}\}$,
\[
\liminf_{\gamma\to\infty}
\frac{\mathcal R_\gamma^{\rm lin}(\mathfrak T;\kappa)}{\log\gamma}
\geq\frac{1}{D_{\mathrm{KL}}(P_1\|P_0)}.
\]
\item Suppose $\mathcal X$ is standard Borel, and fix a Borel
injection $s:\mathcal X\to\mathbb R$ independently of the
unknown distributions. Fix $\rho\in(0,1)$. Use the ARL regime of
Algorithm~\ref{alg: RCMM} and its regularized sum alarm
\eqref{eq:learned-null-regularization}. Then
\[
\Bar{\tau}_{S,w}(\gamma)\in
\mathcal C_\gamma^{\rm lin}(\mathfrak T_{\rm rand}),\qquad
\gamma\leq\mathbb E_{P_0,\infty}\Bar{\tau}_{S,w}(\gamma)
\leq\lceil\gamma/\rho\rceil,
\]
and
\[
\limsup_{\gamma\to\infty}\frac{
\sup_{\nu\in I_{N_\gamma}(\kappa)}
\mathbb E_{P_0,\nu,P_1}[\Bar{\tau}_{S,w}(\gamma)-\nu
\mid\Bar{\tau}_{S,w}(\gamma)>\nu]}{\log\gamma}\leq\frac{1}{D_{\mathrm{KL}}(P_1\|P_0)}.
\]
Consequently,
\[
\lim_{\gamma\to\infty}
\frac{\mathcal R_\gamma^{\rm lin}(\mathfrak T_{\rm rand};\kappa)}{\log\gamma}
=\lim_{\gamma\to\infty}
\frac{\mathcal J_\gamma^\kappa\{\Bar{\tau}_{S,w}(\gamma)\}}{\log\gamma}
=\frac{1}{D_{\mathrm{KL}}(P_1\|P_0)}.
\]
\end{enumerate}
\end{theorem}

The regularized unit-weight construction therefore attains the restricted lower bound. Its deterministic component $\rho t$ forces threshold crossing by $\lceil\gamma/\rho\rceil$ while retaining the same ARL and finite-horizon calibration. For fixed $\rho\in(0,1)$ the local log-boundary penalty is constant and does not change the first-order delay. Proposition~\ref{prop:weighted-null-atom} gives the general-weight version.n Proposition~\ref{prop:weighted-second-order} in Appendix~\ref{Appendix: Lemmas} gives the corresponding learned-score tail bound retaining both the exponential post-change term and the calibration-error term.

\section{Simulations}
\label{sec:simulations}

This section reports the existing fixed-score simulations, corresponding to Algorithm~\ref{alg:fixed-score-RCMM}. We abbreviate our method as RCMM in this section for brevity. We first examine false-alarm control under the no-change model for the ordinary CMM and the maximum-type restarted statistic with summable weights. We then turn to changepoint alternatives and compare the delay of the ordinary CMM with this PFA-controlled RCMM instantiation. All simulations are carried out at the raw-data level. That is, observations are first generated from the stated pre- and post-change distributions, randomized conformal $p$-values are then computed sequentially from the observed data following \eqref{eq: randomized pval}, and the resulting $p$-value sequence is supplied to the CMM and RCMM stopping rules. Since the data distributions used below are continuous, ties occur with probability zero, but the randomized form is retained to match the definition of the conformal $p$-values. In the simulations, both procedures use the same left-tail betting density $f(p)=0.15p^{-0.85}$. The restarted procedure is the PFA-controlled maximum statistic $Z_t^{(w)}$ with weights proportional to $k^{-(1+\eta)}$, $\eta=0.10$, normalized to have total mass one, and threshold $1/\alpha$.

We first examine false-alarm behavior under the no-change model. In each Monte Carlo replication, observations are generated as $X_i\iid N(0,1)$ for all $1\le i\le U$. Since there is no changepoint, any alarm before the monitoring horizon $U$ is counted as a false alarm. For a stopping time $\tau$, the empirical finite-horizon PFA is computed from $B=1000$ independent replications as
$\widehat{\mathrm{PFA}}
=
\frac{1}{B}
\sum_{b=1}^{B}\mathbf 1\{\tau^{(b)}\le U\},$
where $\tau^{(b)}$ denotes the stopping time found in replication $b$. The Monte Carlo standard error is computed using the binomial approximation, $\widehat{\mathrm{se}}
=
\sqrt{
\frac{\widehat{\mathrm{PFA}}\{1-\widehat{\mathrm{PFA}}\}}{B}
},$
and the plotted $95\%$ Monte Carlo intervals are $\widehat{\mathrm{PFA}}
\pm
1.96\,\widehat{\mathrm{se}},$
truncated to the interval $[0,1]$ when necessary.

Figure~\ref{fig1:empirical PFA alpha = 0.05} reports the empirical PFA for the nominal level $\alpha=0.05$ using horizon $U=20000$. Figure~\ref{fig2:empirical PFA several alpha} and Table~\ref{tab:pfa-nominal-levels} repeat the same experiment over several nominal levels using horizon $U=2000$. Across all cases, the empirical false-alarm probabilities remain below the corresponding nominal levels. This agrees with the distribution-free validity of RCMM in Theorem~\ref{th:weighted-null-validity} and the corresponding martingale validity of the ordinary CMM. The lower finite-horizon PFA observed for RCMM should be interpreted as conservativeness induced by the restart weights in this simulation, not as a stronger theoretical false-alarm guarantee; both methods target the same distribution-free PFA control.

\begin{figure}[h!]
    \centering
    \includegraphics[width=0.6\linewidth]{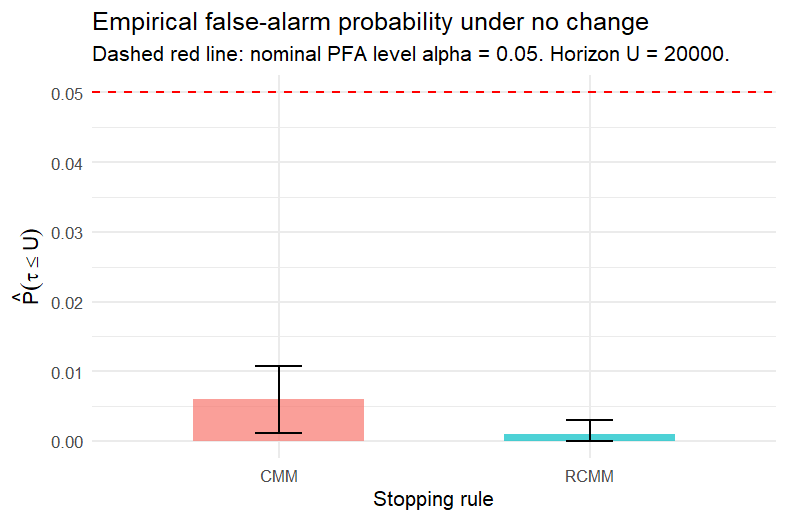}
\caption{\footnotesize Empirical false-alarm probability under the no-change model at PFA level $\alpha=0.05$. 
In each of the 1000 replications, the data are generated as $X_i\iid N(0,1)$ for $1\le i\le U$, with monitoring horizon $U=20000$ and no changepoint. 
The plotted bars report the empirical probability $\widehat{\mathbb P}(\tau\le U)$ over the simulation replications, and the vertical intervals give approximate Monte Carlo 95\% confidence intervals. 
The dashed red line is the target PFA level. 
Both CMM and RCMM remain below the nominal level, consistent with the distribution-free false-alarm control of CMM and that of RCMM as shown in Theorem~\ref{th:weighted-null-validity}.}
    \label{fig1:empirical PFA alpha = 0.05}
\end{figure}

\begin{figure}[h!]
    \centering
    \includegraphics[width=0.6\linewidth]{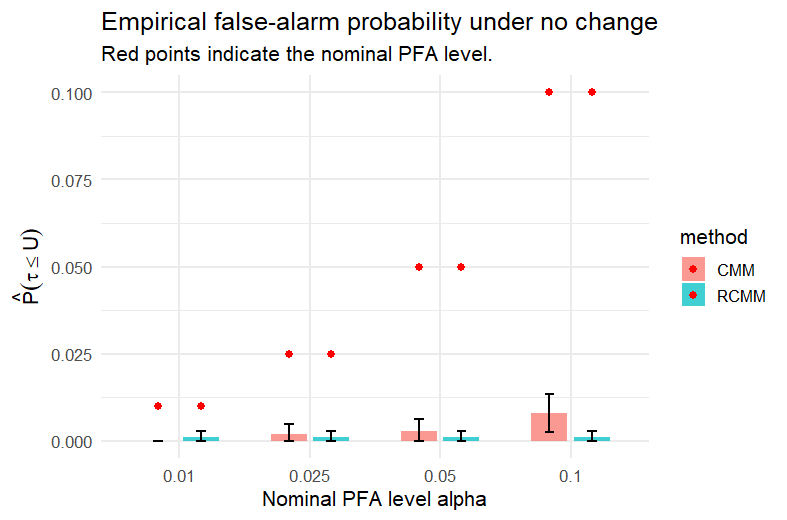}
\caption{\footnotesize Empirical false-alarm probabilities across nominal PFA levels under the no-change model. 
For each value of $\alpha$, 1000 independent replications are generated with $X_i\iid N(0,1)$ for $1\le i\le U$, using monitoring horizon $U=2000$ and no changepoint. 
The bars show $\widehat{\mathbb P}(\tau\le U)$ for CMM and RCMM, the black intervals are approximate Monte Carlo 95\% confidence intervals, and the red points mark the nominal PFA levels. 
Across all nominal levels considered, the empirical false-alarm probabilities remain below the target levels, in agreement with the distribution-free validity for our method (RCMM) as shown in Theorem~\ref{th:weighted-null-validity}.}
    \label{fig2:empirical PFA several alpha}
\end{figure}

\begin{table}[h!]
\centering
\caption{\footnotesize Empirical false-alarm probabilities under the no-change model. 
In each of 1000 Monte Carlo (MC) replications, the data are generated under the no-change model, with
$X_i\iid N(0,1)$ for all $1\le i\le U$, where $U = 2000$ is the finite monitoring horizon. 
The column $\widehat{\mathrm{PFA}}$ reports the empirical probability of stopping by time $U$, and the final column gives an approximate Monte Carlo 95\% confidence interval. 
Across all PFA levels, both CMM and RCMM remain below the target false-alarm level, in accordance with the distribution-free PFA control validity.}
\label{tab:pfa-nominal-levels}
\begin{tabular}{lcccc}
\toprule
Method & PFA bound $\alpha$ & $\widehat{\mathrm{PFA}}$ & Replications & 95\% MC interval \\
\midrule
CMM  & 0.010 & 0.000 & 1000 & $[0.000,\;0.000]$ \\
RCMM & 0.010 & 0.001 & 1000 & $[0.000,\;0.003]$ \\
CMM  & 0.025 & 0.002 & 1000 & $[0.000,\;0.005]$ \\
RCMM & 0.025 & 0.001 & 1000 & $[0.000,\;0.003]$ \\
CMM  & 0.050 & 0.003 & 1000 & $[0.000,\;0.006]$ \\
RCMM & 0.050 & 0.001 & 1000 & $[0.000,\;0.003]$ \\
CMM  & 0.100 & 0.008 & 1000 & $[0.002,\;0.014]$ \\
RCMM & 0.100 & 0.001 & 1000 & $[0.000,\;0.003]$ \\
\bottomrule
\end{tabular}
\end{table}

We next compare detection delay. For each changepoint location $T$, data are generated according to $X_i\iid N(0,1), i \leq T$ and $X_i\iid N(1.2,1), i >T.$ For each value of $T$, we use $50$ independent Monte Carlo replications. The stream is monitored until time $T+5T=6T$. Thus the largest observable post-change delay is $5T$. For a replication $b = 1,2,\cdots,B,$ with no false alarm, the capped delay is
\[
D_{\mathrm{cap}}^{(b)}(T)
=
\begin{cases}
\tau^{(b)}-T, & T\le \tau^{(b)}\le 6T,\\
5T, & \text{if no alarm occurs by time }6T.
\end{cases}
\]
False-alarm replications are recorded separately and are not used in the post-change delay summaries. The reported detection probability is the fraction of no-false-alarm replications satisfying $T\le \tau^{(b)}\le 6T$. The median capped delay is the median of $D_{\mathrm{cap}}^{(b)}(T)$ over the no-false-alarm replications, and the normalized quantity to compare that with $\log T$ in Figure~\ref{fig:vovkcomparisonlogT} and~\ref{fig3:detection delay scaling} is
$\frac{\operatorname{median}\{D_{\mathrm{cap}}^{(b)}(T)\}}{\log T}.$

Figure~\ref{fig:vovkcomparisonlogT} shows a clear difference in the detection delay rates of Vovk's method using CTM and our method (RCMM). For our method (RCMM), the median capped detection delay divided by $\log T$ stays bounded over a large range of the changepoint positions, while for Vovk's method it diverges, validating that the detection delay for our PFA-valid method has delay of order $\log T,$ while for Vovk's method the rate can be $\gg \log T.$ Figure~\ref{fig3:detection delay scaling} draws the same comparison between CMM and our method RCMM, and makes the source of the improvement clearer. Aggregating over the betting class does help relative to committing to a single betting function, but it is not what drives the key improvement: the delay of CMM also degrades as the changepoint location grows. The improvement therefore comes from the weighted restarted construction rather than from the mixture.

 For the ordinary CMM, the median capped delay grows rapidly on the $\log T$ scale and, in the representative settings of Table~\ref{tab:cmm-rcmm-delay-brief}, equals the cap. This reflects the fact that the ordinary CMM carries the accumulated pre-change history and therefore does not operate on the logarithmic delay scale in these experiments. In contrast, RCMM remains stable after normalization by $\log T$ and detects before the cap with increasing probability as $T$ grows. This is an empirical observation for the fixed-score implementation, not an evaluation of the learned-score theorem. The comparison illustrates the role of restarting: by allowing a local martingale to begin near the changepoint, RCMM avoids the dilution that slows the ordinary CMM.



\begin{figure}[h!]
    \centering
    \includegraphics[width=.9\linewidth]{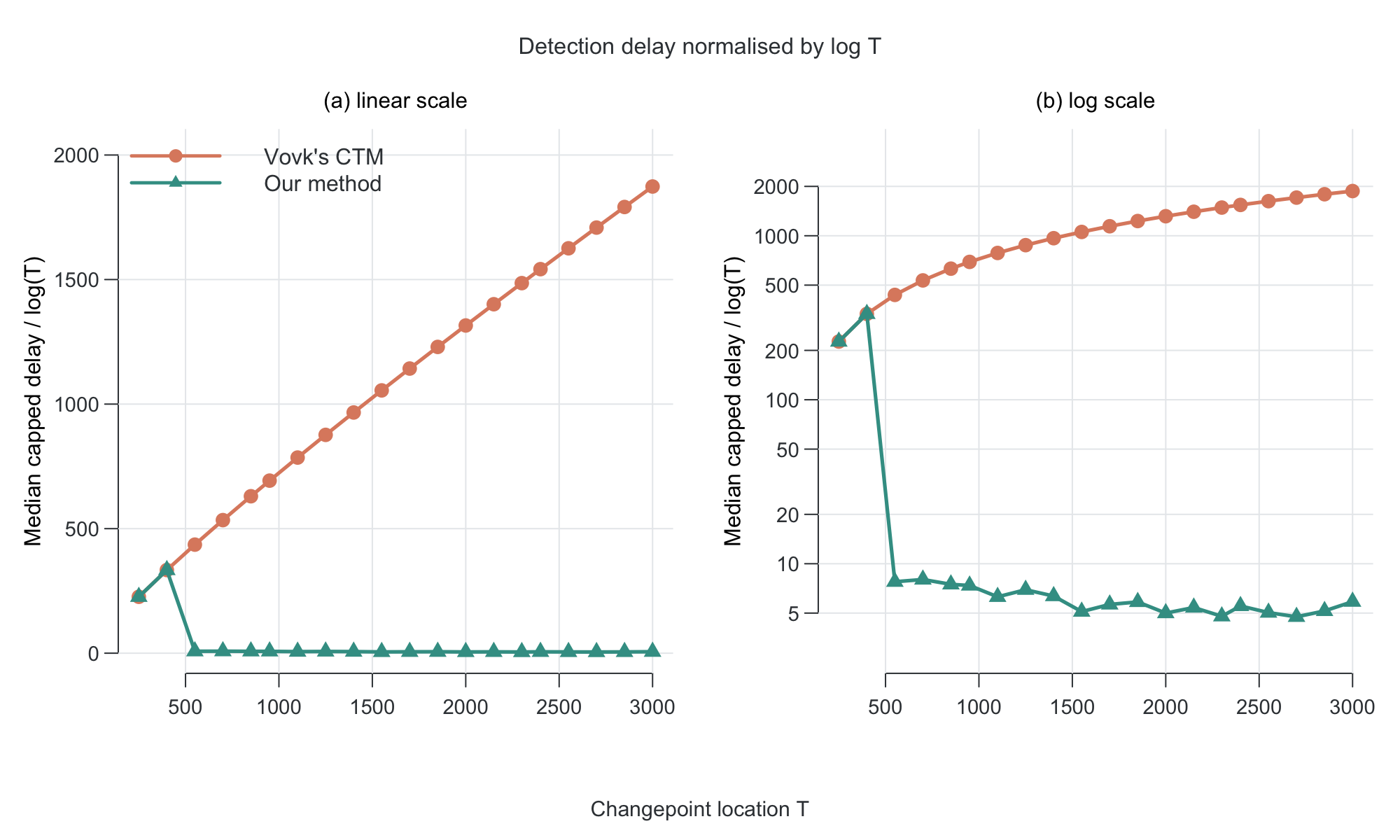}
    \caption{\small Detection-delay scaling of Vovk's conformal test martingale and our
method. For different changepoint locations, $50$ independent streams are
generated with $P_0 \equiv N(0,1), P_1 \equiv N(1.2,1),$ PFA validity
$\alpha = 0.05,$ betting density $f(p) = 0.15\,p^{-0.85}$ and weights
$w_k \propto k^{-1.1}$. Both panels plot the median capped delay divided by
$\log T$ against the changepoint location $T$, on a linear scale (left) and a
logarithmic scale (right), the latter being needed since the two curves differ
by more than two orders of magnitude. The normalized delay of our method stays
flat over the whole range, giving empirical evidence that its detection delay is
of order $\log T$ over the simulated range. That of
Vovk's method instead diverges, reflecting the dilution caused by the
accumulated pre-change history.}
    \label{fig:vovkcomparisonlogT}
\end{figure}

\begin{figure}[h!]
    \centering
    \includegraphics[width=0.5\linewidth]{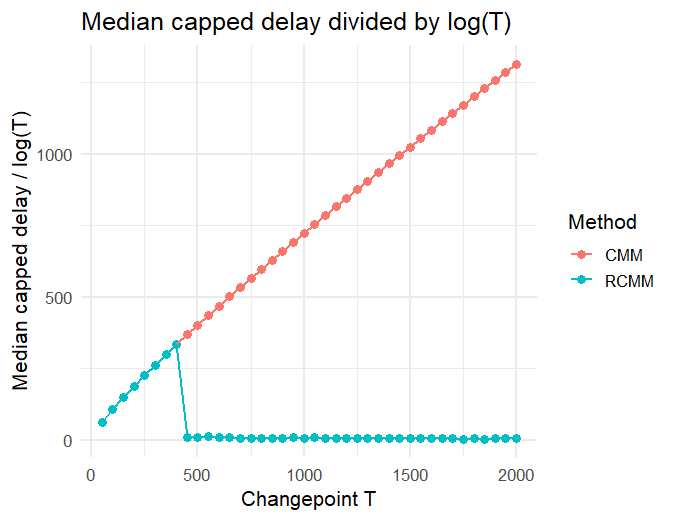}
\caption{\footnotesize Detection-delay scaling for CMM and RCMM in a raw-data changepoint simulation. 
For each changepoint location $T$, independent $50$ replications are generated with $X_i\iid N(0,1)$ for $i\le T$ and $X_i\iid N(1.2,1)$ for $i>T$. 
Sequential randomized conformal $p$-values are computed from the raw observations and then supplied to the CMM and RCMM stopping rules. 
The stream is monitored until time $T+5T$, and the reported capped delay is $(\tau-T)\wedge 5T$. 
The figure plots the median capped delay (across $50$ simulations) divided by $\log T$ as $T$ varies. 
The fixed-score RCMM remains stable on the logarithmic scale in this experiment, whereas the ordinary CMM grows rapidly on this normalization, reflecting the slower delay behavior caused by accumulated pre-change history.}
    \label{fig3:detection delay scaling}
\end{figure}

\begin{table}[h!]
\centering
\caption{\footnotesize Detection-delay comparison for CMM and RCMM in a raw-data simulation. 
For each changepoint location $T$, 50 independent replications are generated with
$X_i\iid N(0,1)$ for $i\leq T$ and $X_i\iid N(1.2,1)$ for $i>T$. For each of the replications the stream is monitored until time $6T$, and the reported capped delay is $(\tau-T)\wedge 5T$; thus, if no alarm occurs before the monitoring horizon, the delay is recorded as $5T$. The CMM does not detect before the cap in these representative settings, so its median capped delay equals the cap. 
In contrast, RCMM detects with increasing probability and has median delay on the scale of $\log T$ over these simulated changepoints; the contrast is consistent with the slower delay behavior of the CMM.}
\label{tab:cmm-rcmm-delay-brief}
\begin{tabular}{rrrrrrrr}
\toprule
& & \multicolumn{3}{c}{CMM} & \multicolumn{3}{c}{RCMM} \\
\cmidrule(lr){3-5}\cmidrule(lr){6-8}
$T$ & Cap 
& Det. & Med. cap & Med. cap$/\log T$
& Det. & Med. cap & Med. cap$/\log T$ \\
\midrule
250  & 1250  & 0.00 & 1250  & 226.39  & 0.28 & 1250.0 & 226.39 \\
500  & 2500  & 0.00 & 2500  & 402.28  & 0.68 & 37.5   & 6.03 \\
1000 & 5000  & 0.00 & 5000  & 723.82  & 0.86 & 37.0   & 5.36 \\
1500 & 7500  & 0.00 & 7500  & 1025.54 & 0.94 & 38.5   & 5.29 \\
2000 & 10000 & 0.00 & 10000 & 1315.63 & 1.00 & 34.5   & 4.54 \\
\bottomrule
\end{tabular}
\end{table}

Overall, the simulations support the two main theoretical conclusions. Under the no-change model, both CMM and RCMM remain calibrated at the prescribed PFA levels. Under the changepoint model, restarting substantially improves the detection scale: RCMM exhibits logarithmic-delay behavior in the changepoint location, whereas the ordinary CMM is slowed by the pre-change history accumulated before the change occurs.

\section{Conclusion}
\label{sec:conclusion}

We have developed a distribution-free online changepoint detection method based on
conformal martingales. Since the efficiency of conformal martingales can depend heavily on the betting function, instead of relying on a single betting function, we instead aggregrade them over a class of betting functions. To avoid the dilution caused by a long pre-change record, the resulting martingale is initiated afresh at every time point, so that each instance accumulates evidence local to a candidate change, and these instances are then aggregated through a weight
sequence. This construction retains the distribution-free validity of the process while offers a major upgrade in the performance. 

The weighted aggregation gives a common validity framework: the total weight controls PFA, while the largest weight controls the ARL lower bound. The learned implementation uses the deterministic PFA and ARL schedules in \eqref{eq:learning-schedules}. Its general-weight delay theorem imposes explicit scale conditions on the block length, effective boundary, and histogram resolution. Separate fixed-score results quantify the guarantees for restricted directional or other betting classes.

For near-harmonic PFA weights and unit ARL weights, the learned construction attains the first-order constants $1/D_{\mathrm{KL}}(P_1\|P_0)$ on the $\log T$ and $\log\gamma$ scales. An information-preserving fixed base score turns this into the raw-data KL constant without oracle inputs. The regularized ARL sum also attains the untruncated conditional-mean bound. Matching lower bounds apply to the stated restricted late-changepoint criteria, pointwise in the pair. The learned tail estimate includes a separate calibration-error term and has a pair-dependent onset. Existing fixed-score simulations complement, but do not test, the score-learning construction.

\paragraph{Acknowledgments.} Positive results were developed solely by the authors, and based on these, ChatGPT helped develop examples showing that existing methods are suboptimal. It was also used at final stages to assist with writing, proofchecking and simulations.




\small
\bibliographystyle{plainnat}
\bibliography{pmlr-sample}

@article{vovk2021testing,
  title={Testing randomness online},
  author={Vovk, Vladimir},
  journal={Statistical Science},
  volume={36},
  number={4},
  pages={595--611},
  year={2021},
  publisher={JSTOR}
}

@article{page1954continuous,
  title={Continuous inspection schemes},
  author={Page, Ewan S},
  journal={Biometrika},
  volume={41},
  number={1/2},
  pages={100--115},
  year={1954},
  publisher={JSTOR}
}

@inproceedings{shiryaev1961problem,
  title={The problem of the most rapid detection of a disturbance in a stationary process},
  author={Shiryaev, Albert N},
  booktitle={Soviet Math. Dokl},
  volume={2},
  number={795-799},
  pages={103},
  year={1961}
}

@article{ramdas2022testing,
  title={Testing exchangeability: Fork-convexity, supermartingales and e-processes},
  author={Ramdas, Aaditya and Ruf, Johannes and Larsson, Martin and Koolen, Wouter M},
  journal={International Journal of Approximate Reasoning},
  volume={141},
  pages={83--109},
  year={2022},
  publisher={Elsevier}
}

@article{chen19sequential,
  title={Sequential change-point detection based on nearest neighbors},
  author={Hao Chen},
  journal={The Annals of Statistics},
  volume={47},
  number={3},
  pages={1381--1407},
  year={2019}
}

@article{chu22sequential,
  title={Sequential change-point detection for high-dimensional and non-{E}uclidean data},
  author={Lynna Chu and Hao Chen},
  journal={IEEE Transactions on Signal Processing},
  volume={70},
  pages={4498--4511},
  year={2022}
}

@inproceedings{fukumizu07kernel,
  title={Kernel measures of conditional dependence},
  author={Kenji Fukumizu and Arthur Gretton and Xiaohai Sun and Bernhard Sch{\"o}lkopf},
  booktitle={Advances in Neural Information Processing Systems (NeurIPS)},
  pages={489--496},
  year={2007}
}

@article{gordon94efficient,
  title={An efficient sequential nonparametric scheme for detecting a change of distribution},
  author={Louis Gordon and Moshe Pollak},
  journal={The Annals of Statistics},
  volume={22},
  number={2},
  pages={763--804},
  year={1994}
}

@article{hawkins10nonparametric,
  title={A nonparametric change-point control chart},
  author={Douglas M. Hawkins and Qiqi Deng},
  journal={Journal of Quality Technology},
  volume={42},
  number={2},
  pages={165--173},
  year={2010}
}

@article{keriven20newma,
  title={N{EWMA}: {A} new method for scalable model-free online change-point detection},
  author={Nicolas Keriven and Damien Garreau and Iacopo Poli},
  journal={IEEE Transactions on Signal Processing},
  volume={68},
  pages={3515--3528},
  year={2020}
}

@article{kurt20real,
  title={Real-time nonparametric anomaly detection in high-dimensional settings},
  author={Mehmet N. Kurt and Yasin Yilmaz and Xiaodong Wang},
  journal={IEEE Transactions on Pattern Analysis and Machine Intelligence},
  volume={43},
  number={7},
  pages={2463--2479},
  year={2020}
}

@inproceedings{li15mstatistic,
  title={M-statistic for kernel change-point detection},
  author={Shuang Li and Yao Xie and Hanjun Dai and Le Song},
  booktitle={Advances in Neural Information Processing Systems ({NeurIPS})},
  pages={3366--3374},
  year={2015}
}

@article{li19scanbstatistics,
  title={Scan {$B$}-statistic for kernel change-point detection},
  author={Shuang Li and Yao Xie and Hanjun Dai and Le Song},
  journal={Sequential Analysis},
  volume={38},
  number={4},
  pages={503--544},
  year={2019}
}

@article{romano23changedetection,
  title={A log-linear non-parametric online changepoint detection algorithm based on functional pruning},
  author={Gaetano Romano and Idris A. Eckley and Paul Fearnhead},
  journal={IEEE Transactions on Signal Processing},
  volume={72},
  pages={594--606},
  year={2024}
}

@article{ross15cpmpackage,
  title={Parametric and nonparametric sequential change detection in {R}: The cpm package},
  author={Gordon J. Ross},
  journal={Journal of Statistical Software},
  volume={66},
  pages={1--20},
  year={2015}
}

@article{sriperumbudur10hilbert,
  title={Hilbert space embeddings and metrics on probability measures},
  author={Bharath K. Sriperumbudur and Arthur Gretton and Kenji Fukumizu and Bernhard Sch{\"o}lkopf and Gert Lanckriet},
  journal={Journal of Machine Learning Research},
  volume={11},
  number={50},
  pages={1517--1561},
  year={2010}
}

@article{wang24sequential,
  title={Sequential change-point detection: Computation versus statistical performance},
  author={Haoyun Wang and Yao Xie},
  journal={Wiley Interdisciplinary Reviews: Computational Statistics},
  volume={16},
  number={1},
  pages={e1628},
  year={2024}
}

@techreport{wei22online,
  title={Online kernel {CUSUM} for change-point detection},
  author={Song Wei and Yao Xie},
  type={Technical report},
  year={2022},
  url={https://arxiv.org/abs/2211.15070}
}

@inproceedings{yilmaz17online,
  title={Online nonparametric anomaly detection based on geometric entropy minimization},
  author={Yasin Yilmaz},
  booktitle={{IEEE} International Symposium on Information Theory ({ISIT})},
  pages={3010--3014},
  year={2017}
}

@inproceedings{zaremba13btest,
  title={B-test: {A} non-parametric, low variance kernel two-sample test},
  author={Wojciech Zaremba and Arthur Gretton and Matthew B. Blaschko},
  booktitle={Advances in Neural Information Processing Systems (NeurIPS)},
  pages={755--763},
  year={2013}
}

@article{kalinke2026optimal,
  title={Optimal Online Change Detection via Random Fourier Features},
  author={Kalinke, Florian and Gavioli-Akilagun, Shakeel},
  journal={Advances in Neural Information Processing Systems},
  volume={38},
  pages={9866--9901},
  year={2025}
}

@article{shin2023detectors,
  title={E-detectors: A nonparametric framework for sequential change detection},
  author={Shin, Jaehyeok and Ramdas, Aaditya and Rinaldo, Alessandro},
  journal={The New England Journal of Statistics in Data Science},
  volume={2},
  number={2},
  pages={229--260},
  year={2024},
  doi={10.51387/23-NEJSDS51},
  publisher={New England Statistical Society}
}

@inproceedings{vovk2003testing,
  title     = {Testing exchangeability on-line},
  author    = {Vovk, Vladimir and Nouretdinov, Ilia and Gammerman, Alex},
  booktitle = {Proceedings of the Twentieth International Conference on Machine Learning},
  pages     = {768--775},
  year      = {2003},
  publisher = {AAAI Press}
}

@inproceedings{fedorova2012plugin,
  title     = {Plug-in martingales for testing exchangeability on-line},
  author    = {Fedorova, Valentina and Gammerman, Alex and Nouretdinov, Ilia and Vovk, Vladimir},
  booktitle = {Proceedings of the 29th International Conference on Machine Learning},
  pages     = {1639--1646},
  year      = {2012},
  publisher = {Omnipress}
}

@inproceedings{volkhonskiy2017inductive,
  title     = {Inductive Conformal Martingales for Change-Point Detection},
  author    = {Volkhonskiy, Denis and Burnaev, Evgeny and Nouretdinov, Ilia and Gammerman, Alexander and Vovk, Vladimir},
  booktitle = {Proceedings of the Sixth Workshop on Conformal and Probabilistic Prediction and Applications},
  pages     = {132--153},
  year      = {2017},
  volume    = {60},
  series    = {Proceedings of Machine Learning Research},
  publisher = {PMLR}
}

@inproceedings{vovk2021binary,
  title     = {Conformal testing in a binary model situation},
  author    = {Vovk, Vladimir},
  booktitle = {Proceedings of the Tenth Symposium on Conformal and Probabilistic Prediction and Applications},
  pages     = {131--150},
  year      = {2021},
  volume    = {152},
  series    = {Proceedings of Machine Learning Research},
  publisher = {PMLR}
}

@inproceedings{nouretdinov2021continuous,
  title     = {Conformal changepoint detection in continuous model situations},
  author    = {Nouretdinov, Ilia and Vovk, Vladimir and Gammerman, Alex},
  booktitle = {Proceedings of the Tenth Symposium on Conformal and Probabilistic Prediction and Applications},
  pages     = {300--302},
  year      = {2021},
  volume    = {152},
  series    = {Proceedings of Machine Learning Research},
  publisher = {PMLR}
}

@inproceedings{vovk2021retrain,
  title     = {Retrain or not retrain: conformal test martingales for change-point detection},
  author    = {Vovk, Vladimir and Petej, Ivan and Nouretdinov, Ilia and Ahlberg, Ernst and Carlsson, Lars and Gammerman, Alex},
  booktitle = {Proceedings of the Tenth Symposium on Conformal and Probabilistic Prediction and Applications},
  pages     = {191--210},
  year      = {2021},
  volume    = {152},
  series    = {Proceedings of Machine Learning Research},
  publisher = {PMLR}
}

@inproceedings{vovk2022markov,
  title     = {Conformal testing: binary case with Markov alternatives},
  author    = {Vovk, Vladimir and Nouretdinov, Ilia and Gammerman, Alex},
  booktitle = {Proceedings of the Eleventh Symposium on Conformal and Probabilistic Prediction with Applications},
  pages     = {207--218},
  year      = {2022},
  volume    = {179},
  series    = {Proceedings of Machine Learning Research},
  publisher = {PMLR}
}

@inproceedings{eliades2022betting,
  title     = {A Betting Function for Addressing Concept Drift with Conformal Martingales},
  author    = {Eliades, Charalambos and Papadopoulos, Harris},
  booktitle = {Proceedings of the Eleventh Symposium on Conformal and Probabilistic Prediction with Applications},
  pages     = {219--238},
  year      = {2022},
  volume    = {179},
  series    = {Proceedings of Machine Learning Research},
  publisher = {PMLR}
}

@inproceedings{vovk2023forgetting,
  title     = {The power of forgetting in statistical hypothesis testing},
  author    = {Vovk, Vladimir},
  booktitle = {Proceedings of the Twelfth Symposium on Conformal and Probabilistic Prediction with Applications},
  pages     = {347--366},
  year      = {2023},
  volume    = {204},
  series    = {Proceedings of Machine Learning Research},
  publisher = {PMLR}
}

@inproceedings{vovk2025cusum,
  title     = {Validity and efficiency of the conformal {CUSUM} procedure},
  author    = {Vovk, Vladimir and Nouretdinov, Ilia and Gammerman, Alexander},
  booktitle = {Proceedings of the Fourteenth Symposium on Conformal and Probabilistic Prediction with Applications},
  pages     = {576--594},
  year      = {2025},
  volume    = {266},
  series    = {Proceedings of Machine Learning Research},
  publisher = {PMLR}
}

@inproceedings{prinster2025watch,
  title     = {{WATCH}: Adaptive Monitoring for {AI} Deployments via Weighted-Conformal Martingales},
  author    = {Prinster, Drew and Han, Xing and Liu, Anqi and Saria, Suchi},
  booktitle = {Proceedings of the 42nd International Conference on Machine Learning},
  pages     = {49830--49859},
  year      = {2025},
  volume    = {267},
  series    = {Proceedings of Machine Learning Research},
  publisher = {PMLR}
}

@article{vovk2025conformaletesting,
  title   = {Conformal e-testing},
  author  = {Vovk, Vladimir and Nouretdinov, Ilia and Gammerman, Alex},
  journal = {Pattern Recognition},
  volume  = {168},
  pages   = {111841},
  year    = {2025},
  doi     = {10.1016/j.patcog.2025.111841}
}

@article{choe2026combining,
  title   = {Combining Evidence Across Filtrations},
  author  = {Choe, Yo Joong and Ramdas, Aaditya},
  journal = {Journal of the Royal Statistical Society Series B: Statistical Methodology},
  pages   = {},
  year    = {2026},
  doi     = {}
}

@article{ramdas2026complete,
  title={A complete characterization of sequential testability and change detectability in iid models},
  author={Ramdas, Aaditya},
  journal={arXiv preprint arXiv:2609.05752},
  year={2026}
}
\normalsize

\clearpage
\appendix

\begin{center}
    \textbf{\Large Appendices}
\end{center}

\paragraph{Organization of the Appendices.}
Appendix \ref{sec: change detection using CMM} discusses the delay guarantees for the change detection using conformal test martingale of \cite{vovk2003testing}.
Appendix~\ref{apd:regularity of betting class} records optional conditions and examples for the fixed-score betting classes required for the analysis in Appendix \ref{sec: change detection using CMM}. Appendix~\ref{apd:impossibility theorem} contains the universal-supermartingale impossibility result. Appendix~\ref{Appendix: Lemmas} states auxiliary lemmas and the learned-score tail bound; Appendix~\ref{Appendix: Proofs} contains the proofs. Appendix~\ref{apd:fixed-score-algorithm} preserves the implementation used in the fixed-score simulations.

\section{Change detection using conformal test martingale of \cite{vovk2003testing}}
\label{sec: change detection using CMM}

This section concerns the original fixed-score conformal ranks and Vovk's ordinary CMM with common prior $\Pi$. These results are of independent interest and are not needed for the learned-score proofs in Section~\ref{sec:optimality}. In particular, $p_{n,j}$ in this section has the original meaning explicitly specified below, not the transformed meaning in \eqref{eq:learned-pvalues}. We also need the following assumption.

\begin{assumption}[Post-change score-transform separation]
\label{assn:post-change-pvalue-separation}
The CDF $F_{P_0}$ and the function $H(t):=\mathbb P_{Y\sim P_1}\{F_{P_0}(Y)\leq t\},
\hspace{.2cm}t\in[0,1],$
are continuous, $H(0)=0$, and
\begin{equation}
\label{eq:Delta-separation}
\Delta:=\|H-F_{\mathrm U}\|_\infty
=\sup_{t\in[0,1]}|H(t)-t|>0.
\end{equation}
\end{assumption}

\begin{remark}
The continuity assumption on $H$ is mild. For example, if $F_{P_0}$ is continuous and $P_1\{x:F_{P_0}(x)=a\}=0$
for every $a\in[0,1],$
then $H(t):=\mathbb P_{Y\sim P_1}\{F_{P_0}(Y)\leq t\}$ is continuous. Thus the no-mass-on-level-sets condition ensures that $F_{P_0}(Y)$ has no atoms and hence that $H$ is continuous. When $F_{P_0}$ is strictly increasing, this condition follows from continuity of the post-change score law $P_1$.
\end{remark}

When $F_{P_0}$ is continuous, the variable $F_{P_0}(Y)$ is uniform for $Y\sim P_0$. If $F_{P_0}$ is continuous and strictly increasing, then $H=F_{\mathrm U}$ if and only if $P_1=P_0$, so $\Delta>0$ is equivalent to a genuine change in the score-transformed observations. 

\begin{theorem}[Local empirical separation of conformal $p$-values]
\label{thm:local-post-change-pvalues}
Let $\lambda_{n,j}\stackrel{\mathrm{iid}}{\sim}U(0,1)$ be independent of the array and define
$p_{n,j}:=p_j(X_{n,1},\ldots,X_{n,j};\lambda_{n,j})$. Let $r_n$ be a deterministic sequence with
$r_n\to\infty, \hspace{.2cm}
r_n=o(T_n), \hspace{.2cm}
\frac{r_n}{\sqrt{T_n}}\to\infty;$
put $m_n:=T_n+r_n$, and let $\widehat F_{n,m_n}$ be the empirical CDF of $p_{n,1},\ldots,p_{n,m_n}$. Then under 
\eqref{eqn:data structure} and Assumption~\ref{assn:post-change-pvalue-separation},
\begin{equation}
\label{eq:conv to Delta}
\frac{m_n}{r_n}
\sup_{t\in[0,1]}
\left|\widehat F_{n,m_n}(t)-t\right|
\xrightarrow{\mathbb P}
\Delta.
\end{equation}
Consequently, $\sqrt{m_n}\,
\sup_{t\in[0,1]}
\left|\widehat F_{n,m_n}(t)-t\right|
\xrightarrow{\mathbb P}
\infty.$
If, in addition, $\frac{r_n^2}{T_n\log n}\to\infty,$ then both convergences hold almost surely.
\end{theorem}

The proof of the above result proceeds by decomposing $\widehat F_{n,m_n}=\frac{T_n}{m_n}\widehat F_n^{\,0}
+\frac{r_n}{m_n}\widehat F_{n,r_n}^{\,1}$ i.e., into its pre- and post-change
empirical CDFs; the first term has uniform fluctuations of order $T_n^{-1/2}$ while
the second term converges to $H$, so after multiplication by $m_n/r_n$ the
deterministic separation $\sup_t|H(t)-t|=\Delta$ dominates precisely when
$r_n/\sqrt{T_n}\to\infty$, and the condition involving $\frac{r_n^2}{T_n \log n}$ makes the
concentration errors summable across rows leading to almost sure convergence.

Thus the empirical distribution of the post-change conformal $p$-values departs from uniformity at the exact
order $r_n/m_n$; the departure exceeds the $m_n^{-1/2}$ fluctuation of an exchangeable
sample, so the change is detectable in principle. This does not by itself bound
the delay. Since the mixture martingale accumulates log-capital at the rate at
which the $p$-values depart from uniformity, a deviation of order $r_n/m_n$ over
$r_n$ post-change observations yields log-capital of order $r_n^2/T_n$, so
crossing a fixed threshold requires $r_n$ of order $\sqrt{T_n}$. We make this
precise in the next result. For the precise rates as discussed in the next results, we also require some additional assumption on the betting class which is discussed in detailed in appendix \ref{apd:regularity of betting class}.

\begin{theorem}[Growth of the conformal mixture martingale after the changepoint]
\label{thm:CMM-growth-corrected}
Assume \eqref{eqn:data structure} and Assumption~\ref{assn:post-change-pvalue-separation}. Define $p_{n,j}$ as above and
$M_{n,j}:=M_j(p_{n,1},\ldots,p_{n,j})$ using \eqref{eq: conformal mixture martingale}. Let the sequence $r_n$ satisfy
\begin{equation}
\label{eq:r_n}
r_n\to\infty,
\qquad
r_n=o(T_n),
\qquad
\frac{r_n^2}{T_n\log T_n}\to\infty.
\end{equation}
Under either Assumption~\ref{assn:uniform-local-mixture-richness} or Assumption~\ref{assn:directional-local-detectability}, there exists $c_0>0$, depending only on $P_0,P_1$ and the constants of the fixed pair $(\mathcal F,\Pi)$ and not on the sequence $(r_n)$, such that $\mathbb P\left\{
\log M_{n,T_n+r_n}
\geq
c_0\frac{r_n^2}{T_n}
\right\}
\to1.$
In particular, $\log M_{n,T_n+r_n}\to\infty$ in probability. If, in addition,
\begin{equation}
\label{eq:CMM-growth-as-condition}
\frac{r_n^2}{T_n\log n}\to\infty,
\end{equation}
then $\mathbf{1}(
\log M_{n,T_n+r_n}
\geq
c_0\frac{r_n^2}{T_n}
)$ converges almost surely to one, and $\log M_{n,T_n+r_n}\to\infty$ almost surely as $n \to \infty.$
\end{theorem}

The above result provides the growth rate and the explicit threshold that the CMM crosses after the distribution shift. In the double-array setup, define
\begin{equation}
    \label{eq: CMM Stopping time double}
    \tau_n^{\texttt{CMM}}(b)
    :=
    \inf\{k\geq1:M_{n,k}\geq b\},
\end{equation}
and use $\tau_n^{\texttt{CMM}}(\frac{1}{\alpha})$ for a PFA level $\alpha\in(0,1)$. Since the log-capital at time $T_n+r_n$ is of order $r_n^2/T_n$ for every $r_n$ satisfying \eqref{eq:r_n}, the threshold enters the delay only through $\log b_n$. Theorem~\ref{thm:CMM-growth-corrected} is thus the main tool for bounding the detection delay of CMM alarms, which we develop in stages: a smaller $\alpha$ demands more accumulated evidence before stopping and hence a longer post-change delay, and the next result makes this dependence explicit.

\begin{theorem}[Bound on CMM detection delay]
\label{th: upper bound on detection delay}
Assume the setup of Theorem~\ref{thm:CMM-growth-corrected}. Let $b_n\to\infty$ satisfy $\log T_n=o(\log b_n),
\hspace{.2cm}
\log b_n=o(T_n).$
Then there exists a constant $C>0$ such that
\begin{equation}
    \label{eq:prob conv}
    \mathbb P\left\{
    T_n<\tau_n^{\texttt{CMM}}(b_n)
    \leq
    T_n+C\sqrt{T_n\log b_n}
    \right\}
    \to1.
\end{equation}
Consequently, $(\tau_n^{\texttt{CMM}}(b_n)-T_n)^+
=
O_{\mathbb P}\{\sqrt{T_n\log b_n}\},
\hspace{.2cm}
\frac{(\tau_n^{\texttt{CMM}}(b_n)-T_n)^+}{T_n}
\xrightarrow{\mathbb P}0.$
\end{theorem}

Inverting the growth rate $r^2/T_n$ of Theorem~\ref{thm:CMM-growth-corrected} against the boundary $\log b_n$ gives the delay scale $\sqrt{T_n\log b_n}$; the condition $\log T_n=o(\log b_n)$ ensures that the threshold dominates the logarithmic stochastic terms, so that the boundary alone determines the scale, while $\log b_n=o(T_n)$ keeps the crossing window negligible relative to the calibration sample. This upper bound is sublinear in the changepoint location in the stated growing-threshold regime. It does not imply a matching lower bound or a fixed-threshold $\sqrt{T_n}$ rate. The learned restarted construction has the separately established logarithmic upper bounds in Section~\ref{sec:optimality}.

The previous result gives an upper bound on the delay. The next result quantifies how
quickly the probability of a late alarm decays beyond the delay scale.

\begin{theorem}[Tail decay bound for the CMM stopping time]
\label{th:CMM-late-alarm-tail}
Assume the setup of Theorem~\ref{thm:CMM-growth-corrected} and suppose Assumption~\ref{assn:directional-local-detectability} holds. Let $b_n>1$ be deterministic. There exist constants $C_0,C_1,C_2>0$ and $\delta_0\in(0,1)$ such that, for all sufficiently large $n$ and every integer $r$ satisfying $C_0\sqrt{T_n\{\log T_n+\log b_n\}}
\leq r\leq\delta_0T_n,$
we have $\mathbb P\{\tau_n^{\texttt{CMM}}(b_n)>T_n+r\}
    \leq
    C_1\exp\left(-C_2\frac{r^2}{T_n}\right).$
\end{theorem}

The preceding results describe the delay when the false-alarm level is allowed to vary. In many applications, the level $\alpha$ is fixed in advance, and the relevant question is whether the alarm occurs within a deterministic post-change window whose length is negligible relative to $T_n$. The following growth result for the CMM gives such a statement directly: any sequence $r_n$ satisfying \eqref{eq:r_n} is sufficient for eventual threshold crossing.

\begin{corollary}[CMM delay at a fixed threshold]
\label{cor:CMM-fixed-threshold}
In the setup of Theorem~\ref{thm:CMM-growth-corrected}, fix $b>1$. For every deterministic sequence $r_n$ satisfying \eqref{eq:r_n}, $\mathbb P\{\tau_n^{\texttt{CMM}}(b)\leq T_n+r_n\}
\to1,$
and the false-alarm bound implies $\liminf_{n\to\infty}
\mathbb P\{T_n<\tau_n^{\texttt{CMM}}(b)\leq T_n+r_n\}
\geq1-\frac1b.$
If \eqref{eq:CMM-growth-as-condition} also holds, then $\mathbf 1\{(\tau_n^{\texttt{CMM}}(b)-T_n)^+\leq r_n\}
\longrightarrow1$ almost surely.
In particular, if $T_n=\Theta(n)$ and
$r_n=a_n\sqrt{n\log n}$ with
$a_n\to\infty$ and $a_n=o\{\sqrt{n/\log n}\}$, then $(\tau_n^{\texttt{CMM}}(b)-T_n)^+
\leq
a_n\sqrt{n\log n}$ eventually almost surely. 
\end{corollary}

The fixed-threshold statement separates threshold crossing from a correct post-change alarm. Crossing by $T_n+r_n$ occurs with probability tending to one along the array, whereas the unconditional correct-window probability is limited by the possibility of a pre-change false alarm and is therefore guaranteed only up to the PFA level. This distinction is necessary for any bounded-PFA procedure.

\paragraph{Deterministic monitoring horizon.}
So far, the CMM stopping time has been defined on an indefinitely long stream. In an implemented monitoring system, however, the procedure is run only up to a deterministic operating horizon, possibly large but finite. It is therefore useful to verify that truncating the alarm at this horizon does not alter the local delay analysis on the scale considered above.

For an integer-valued deterministic monitoring horizon $U_n\geq T_n$, the truncated stopping time is $\tau_n^{\texttt{CMM}}(b_n)\wedge U_n.$
The next lemma shows that, under a mild growth restriction on $U_n$, the probability of an alarm later than a suitable local window remains negligible even after multiplication by the horizon. This is the quantity needed to control the contribution of late alarms to truncated moments.

\begin{lemma}[Negligible contribution of late CMM alarms]
\label{lem:late-alarm-negligible}
Assume the setup of Theorem~\ref{th:CMM-late-alarm-tail}. Fix a constant threshold $b>1$ and $\gamma\in(0,1/2)$. Let $U_n\geq T_n$ satisfy
\begin{equation}
\label{eq:threshold size}
\log(T_nU_n)=o(T_n^{1-2\gamma}).
\end{equation}
Then there exists a deterministic integer sequence $d_n$ satisfying $d_n\to\infty, \hspace{.2cm}
d_n=o(T_n), \hspace{.2cm}
\frac{d_n}{T_n^{1-\gamma}}\to0,$
such that $U_n\,\mathbb P\{\tau_n^{\texttt{CMM}}(b)>T_n+d_n\}
\to0.$
\end{lemma}

The condition permits the monitoring horizon to grow substantially faster than any polynomial, provided its logarithm is negligible relative to $T_n^{1-2\gamma}$. The lemma also gives a direct truncated-delay consequence:
\[
\begin{split}
\mathbb E\left[(\{\tau_n^{\texttt{CMM}}(b)\wedge U_n\}-T_n)^+\right]
&\leq
 d_n
 +U_n\,\mathbb P\{\tau_n^{\texttt{CMM}}(b)>T_n+d_n\}\\
&=
 d_n+o(1)
=
o(T_n^{1-\gamma}).
\end{split}
\]
Thus a deterministic operating horizon satisfying \eqref{eq:threshold size} does not change the local CMM delay scale, while making the monitored stopping time finite by construction. This result also validates the simulation studies presented in the next section.

\section{Regularity of the betting class}
\label{apd:regularity of betting class}
For the detection delay results discussed in the appendix \ref{sec: change detection using CMM}, we require some regularity condition on the underlying betting class $\F$ and the prior $\Pi$.
We discuss those assumptions in this appendix.

\begin{assumption}[Uniform local mixture richness]
\label{assn:uniform-local-mixture-richness}
The fixed measurable pair $(\mathcal F,\Pi)$ defining the CMM has constants $\kappa>0$, $K\geq0$, and $\epsilon_0>0$ such that, for every $m\geq1$ and every $z_1,\ldots,z_m\in[0,1]$, with $G_m(t):=\frac1m\sum_{j=1}^m\mathbf 1\{z_j\leq t\}$ and $F_{\mathrm U}(t):=t$,
whenever $0<\|G_m-F_{\mathrm U}\|_\infty\leq\epsilon_0,$
we have
\[
\log\int_{\mathcal F}
\exp\left\{
m\int_0^1\log f(u)\,dG_m(u)
\right\}
d\Pi(f)
\geq
\kappa m\|G_m-F_{\mathrm U}\|_\infty^2
-
K\log\frac{1}{\|G_m-F_{\mathrm U}\|_\infty}
-
K.
\]
\end{assumption}

The left-hand side in the above inequality is the log-mixture capital evaluated at an empirical distribution $G_m$ of $p$-values. The condition requires that, whenever $G_m$ is close to but distinct from uniform, the mixture places enough mass on betting functions that exploit this discrepancy. The leading term is quadratic in the Kolmogorov distance from uniformity, matching the local scale of many goodness-of-fit alternatives, while the logarithmic correction accounts for the prior mass needed to average over the function class. This assumption is used to convert post-change separation of the conformal $p$-values into deterministic growth of the conformal mixture martingale. The following proposition gives a precise example of such a betting family. 

\begin{proposition}[Threshold betting family]
\label{ex:threshold-betting-family}
For $a\in[0,1]$ and $\theta\in[-1/2,1/2]$, define $g_a(u):=\mathbf 1\{u\leq a\}-a, \hspace{.2cm}
f_{a,\theta}(u):=1+\theta g_a(u), \hspace{.2cm}u\in[0,1].$
Let $\mathcal F_{\mathrm{thr}}
:=
\{f_{a,\theta}:a\in[0,1],\theta\in[-1/2,1/2]\},$
equip $[0,1]\times[-1/2,1/2]$ with its Borel sigma-field, and let $\Pi_{\mathrm{thr}}$ be its uniform probability measure. The kernel $(a,\theta,u)\mapsto f_{a,\theta}(u)$ is jointly measurable, and the pair $(\mathcal F_{\mathrm{thr}},\Pi_{\mathrm{thr}})$ satisfies
Assumption~\ref{assn:uniform-local-mixture-richness}.
\end{proposition}

The threshold family gives a concrete omnibus construction: by varying $a$, the bet can target departures of the empirical $p$-value distribution at different parts of $[0,1]$, while $\theta$ controls the direction and strength of the bet. For some results, however, it is useful to impose a more local condition, requiring only that the betting class contain a smooth one-dimensional path that is aligned with the post-change departure from uniformity.

\begin{assumption}[Directional local detectability]
\label{assn:directional-local-detectability}
The fixed measurable pair $(\mathcal F,\Pi)$ defining the CMM is such that there exists constants $\theta_0>0$, $C<\infty$, $B<\infty$, $c_\Pi>0$, a bounded Lipschitz function $g:[0,1]\to\mathbb R$, and a collection of densities $\{f_\theta:|\theta|\leq\theta_0\}\subseteq\mathcal F$ such that, with $H$ as in Assumption~\ref{assn:post-change-pvalue-separation}, $f_0\equiv1,
\hspace{.2cm}
\int_0^1g(u)\,du=0, \hspace{.2cm}
J:=\int_0^1g(u)\,dH(u)\neq0,$
and, for every $|\theta|\leq\theta_0$, $\sup_{u\in[0,1]}
\left|
\log f_\theta(u)-\theta g(u)
\right|
\leq C\theta^2,$
and $\sup_{u\in[0,1]}
\left|
\frac{\partial}{\partial\theta}\log f_\theta(u)
\right|
\leq B.$
Furthermore, for all sufficiently small $\epsilon>0$ and all $\theta\in[-\theta_0/2,\theta_0/2]$, $\Pi\{f_\vartheta:|\vartheta-\theta|\leq\epsilon\}\geq c_\Pi\epsilon .$
\end{assumption}

This condition is weaker in spirit than the preceding richness assumption. Rather than requiring the mixture to exploit every small Kolmogorov departure from uniformity, it requires one regular direction $g$ for which the transformed post-change variable $F_{P_0}(Y)$, $Y\sim P_1$, has nonzero mean shift relative to uniformity. The expansion of $\log f_\theta$ gives a controlled first-order approximation to the log-bet along this direction, and the sign of $J=\int g\,dH$ determines which side of the path yields positive growth. The derivative bound and local prior mass condition ensure that the mixture martingale retains this growth after averaging over nearby betting functions. This is the form used in the delay analysis, where the post-change evidence rate is reduced to a scalar quantity determined by $J$ and the local quadratic remainder. The next result gives a precise construction for the setup. 

\begin{proposition}[Countable exponential-tilt family]
\label{ex:countable-exp-tilt-family}
Fix $\theta_0>0$. Let $\mathcal G=\{g_\ell:\ell\geq1\}$ be a countable collection of Lipschitz functions on $[0,1]$ such that $\int_0^1g_\ell(u)\,du=0,
\hspace{.2cm}
\|g_\ell\|_\infty\leq1,$
and such that $\mathcal G$ is dense, in the uniform norm, in the unit ball $\mathcal C_{0,1}
:=
\left\{g\in C[0,1]:\int_0^1g(u)\,du=0,\ \|g\|_\infty\leq1\right\}.$
For $\ell\geq1$ and $\theta\in[-\theta_0,\theta_0]$, define $f_{\ell,\theta}(u)
:=
\exp\{\theta g_\ell(u)-\psi_\ell(\theta)\},
$ and $
\psi_\ell(\theta)
:=
\log\int_0^1\exp\{\theta g_\ell(u)\}\,du.$
Let $\{w_\ell\}_{\ell\geq1}$ be strictly positive weights satisfying
$\sum_{\ell=1}^{\infty}w_\ell=1$, and define
$\mathcal F_{\exp}
:=
\{f_{\ell,\theta}:\ell\geq1,\ |\theta|\leq\theta_0\}.$
Equip $\mathbb N\times[-\theta_0,\theta_0]$ with the product sigma-field, and let $\Pi_{\exp}$ be the probability measure that first draws $\ell$ with probability $w_\ell$ and then, conditional on $\ell$, draws $\theta$ uniformly on $[-\theta_0,\theta_0]$. The betting kernel is jointly measurable. Then, for any continuous CDF $H$ on $[0,1]$ such that $H\neq F_{\mathrm U}$, the pair
$(\mathcal F_{\exp},\Pi_{\exp})$ satisfies Assumption~\ref{assn:directional-local-detectability}.
\end{proposition}

The proposition gives a concrete way to build a betting class that is not tailored to a particular post-change distribution. Since the exponential tilts range over a dense set of mean-zero directions, any continuous non-uniform distribution of $F_{P_0}(Y)$, $Y\sim P_1$, creates positive drift for at least one path in the mixture. We now use these conditions to establish growth results of the CMM after the changepoint.

\section{Impossibility of divergence for universal nonparametric test supermartingales under canonical filtration}
\label{apd:impossibility theorem}

The following result establishes that no universally valid nonparametric supermartingale can be consistent against a fixed changepoint alternative.

\begin{theorem}(Impossibility of divergence for universal nonparametric test supermartingales)
\label{th: impossibility theorem}
Let $(\Omega,\mathcal B)$ be a measurable space, let $\mathcal M$ be a convex composite class of probability measures on it, and $\{X_n\}_{n \in \N}$ be a sequence of independent random variables defined on the canonical product space $(\Omega^{\mathbb N},\mathcal B^{\otimes\mathbb N})$. Consider the hypotheses
\begin{align}
\label{eq: hypothesis}
H_0 &: X_i\stackrel{\mathrm{iid}}{\sim}P_0\text{ for all }i\geq1,\text{ for some }P_0\in\mathcal M,\\
\label{eq: hypothesis2}
H_1 &: X_i \iid P_0 \text{ for } i \le T, X_i \iid P_1\neq P_0 \text{ for } i > T, \text{ for some } T \in \N, P_0,P_1\in \mathcal M.
\end{align}
(Importantly, $P_0$, $P_1$, and $T$ are all unknown and arbitrary.) Let $\mathcal A_n:=\sigma(X_1,\ldots,X_n)$, and let $(M_n,\mathcal A_n)_{n\geq0}$ be nonnegative and adapted, with deterministic $M_0\leq1$. Write $M_n=m_n(X_1,\ldots,X_n)$ for a measurable function $m_n:\Omega^n\to[0,\infty]$. Assume that, for every $P\in\mathcal M$, $\mathbb E_{P^{\otimes\infty}}(M_n\mid\mathcal A_{n-1})
\leq M_{n-1}
\hspace{.2cm} P^{\otimes\infty}\text{-almost surely},
\hspace{.2cm} n\geq1.$
Then $(M_n,\mathcal A_n)_{n\geq0}$ is a test supermartingale under every distribution in $H_1$.
\end{theorem}

 Since under the alternate $M$ is nonnegative supermartingale, Doob's convergence theorem gives $M_n\to M_\infty<\infty$ almost surely. Furthermore, Ville's inequality gives, for every $\alpha\in(0,1)$,
\begin{equation}
\label{eq:fixed-change-ville}
\mathbb P_{P_0,T,P_1}
\left\{
\sup_{n\geq0}M_n\geq\frac1\alpha
\right\}
\leq\alpha.
\end{equation}
Thus a level-$1/\alpha$ alarm based on a universal full-filtration test supermartingale has detection probability at most $\alpha$ under every fixed changepoint alternative covered by the theorem. This is stronger than mere failure of divergence and makes explicit why null validity alone does not provide power against the composite alternative.

The result is also consistent with the fork-convexity perspective of \cite{ramdas2022testing}. When the laws in $\mathcal{M}$ are dominated by a common $\sigma$-finite measure, the corresponding likelihood-ratio processes are well defined relative to that common reference measure, giving the local absolute-continuity structure used in fork-convex formulations. If, in addition, $\mathcal{M}$ is fork-convex, then likelihood-ratio processes under the distribution-shift from $P_0$ to $P_1$ at a fixed time may be viewed as fork-convex combinations of likelihood-ratio processes generated under the two stationary laws $P_0^{\otimes \infty}$ and $P_1^{ \otimes \infty}$. Since fork-convex combinations preserve the nonnegative supermartingale property, the likelihood ratio process remains a supermartingale under the distribution shift and hence converges almost surely to an almost surely finite random variable by Doob's supermartingale convergence theorem. Thus, the impossibility of its divergence implies that a consistent test for the distribution change alternative can not be constructed using the likelihood ratio process. Theorem~\ref{th: impossibility theorem} generalizes this impossibility beyond the specific likelihood ratio process to the general class of universal nonparametric test supermartingales. Since fork convexity implies convexity, assuming $\M$ to be fork convex yields the same impossibility conclusions. 

In Theorem~\ref{th: impossibility theorem}, the requirement that $\M$ is composite and both $P_0, P_1 \in \M$ is crucial. Otherwise, for example, if $P_0,T$ are known, $P_1\ll P_0$ and $0<D_{\mathrm{KL}}(P_1\|P_0)<\infty$, then
$L_{T,n}:=\prod_{i=T+1}^n\frac{dP_1}{dP_0}(X_i),
\hspace{.2cm}n\geq T,$
with $L_{T,T}=1$ is a $P_0$-martingale after time $T$, and the strong law gives $\frac{1}{n-T}\log L_{T,n}
\longrightarrow
\mathbb E_{P_1}\log\frac{dP_1}{dP_0}(X)
=D_{\mathrm{KL}}(P_1\|P_0)
\hspace{.2cm}\mathbb P_{P_0,T,P_1}\text{-a.s,}$
implying $L_{T,n} \to \infty$ almost surely under $H_1$. The next proposition illustrates the impossibility for simple hypotheses.

\begin{proposition}[Exponential decay of a fixed conformal product bet]
\label{prop: convergence of conformal product martingale}
Let $Z_i=s(X_i)$ be independent scalar scores with
$Z_i\stackrel{\mathrm{iid}}{\sim}P_0$ for $i\leq T$ and
$Z_i\stackrel{\mathrm{iid}}{\sim}P_1$ for $i>T$, where $F_{P_0}$ and $F_{P_1}$ are continuous. Let $p_i$ be the randomized conformal ranks in \eqref{eq: randomized pval}, with independent randomizers $\lambda_i\stackrel{\mathrm{iid}}{\sim}U(0,1)$. Let $f:[0,1]\to(0,\infty)$ be continuous, satisfy $\int_0^1f(u)\,du=1$, and not be identically one. For
$M_n:=\prod_{i=1}^nf(p_i)$, $\frac1n\log M_n
\longrightarrow
\int_0^1\log f(u)\,du
=:-\mu_f<0$ almost surely.
\end{proposition}

Thus for the simple null, $M_n\to0$ exponentially fast almost surely. The impossibility results above show that all test supermartingales in the original filtration are trivial, so we have to work in a reduced filtration in order to develop nontrivial test supermartingales.

\section{Auxiliary Lemmas}
\label{Appendix: Lemmas}

The next two lemmas establish the learned-score growth used in Section~\ref{sec:optimality}. The subsequent proposition gives a tail bound with a pair-dependent onset; its calibration term is kept separate.

\begin{lemma}[Stability of ordering by an estimated density]
\label{lem:stable-oracle-ordering}
For a density $h$ on $[0,1]$, define
\[
h^\uparrow(v):=\inf\{z\geq0:U\{u:h(u)\leq z\}\geq v\},\quad0<v<1.
\]
Suppose $\widehat h$ is a density constant on $K$ equal-width bins,
and $\sigma(v)$ is the position of bin $v$ when its height is
ordered increasingly, with ties broken by bin index. Define
$T_\sigma(u)=u+(\sigma(v)-v)/K$ for $u\in I_{K,v}\cap[0,1)$.
This is a Lebesgue-measure-preserving bijection, and
\[
\|h\circ T_\sigma^{-1}-h^\uparrow\|_1\leq2\|\widehat h-h\|_1.
\]
For every bounded measurable $\psi$,
\[
\left|\int_0^1h(u)\psi(T_\sigma(u))\,du
-\int_0^1h^\uparrow(u)\psi(u)\,du\right|
\leq2\|\psi\|_\infty\|\widehat h-h\|_1.
\]
Furthermore, $h^\uparrow$ is a density and has the same entropy
as $h$, whenever that entropy is finite.
\end{lemma}

\begin{lemma}[Lower tail of the learned log increments]
\label{lem:learned-score-growth}
Fix $P_1\ll P_0$, a score $s$, and $H=H_s$ from
\eqref{eq:score-transform}. Consider \eqref{eq: data structure}
along $B,T\to\infty$ with $B=o(T)$. Use the deterministic blocks
in \eqref{eq:learned-histogram}--\eqref{eq:learned-pvalues}.
Suppose every block intersecting $[T/2,2T]$ satisfies
\[
\tfrac12 B^{1/2}\leq a_{j+1}-a_j\leq2B^{1/2},\qquad
\tfrac12 K_B(T/3)\leq K_j\leq2K_B(2T).
\]
For every fixed bounded Lipschitz $\psi$ and every $\epsilon>0$,
there exist $c,C,\delta,A>0$, independent of $B,T,r$, such that
for all sufficiently large $B,T$ and all integers
$AB^{1/2}\leq r\leq\lfloor\delta T/K_B(2T)\rfloor$,
\[
\mathbb P_{P_0,T,P_1}\left\{
\sum_{t=T+1}^{T+r}\psi(p_t)
<r\left(\int_0^1h^\uparrow(u)\psi(u)\,du-\epsilon\right)
\right\}
\leq Ce^{-cr}+2\exp\{-cT/K_B(2T)^2\}.
\]
The constants and the onset may depend on $h,\psi,\epsilon$.
The PFA schedule satisfies the block conditions with $B=\log T$;
the ARL schedule satisfies them with $B=\log\gamma$ whenever
$\log\gamma=o(T)$.
\end{lemma}

\begin{proposition}[Tail bound for learned-score detection delay]
\label{prop:weighted-second-order}
Under Theorem~\ref{th:weighted-KL-upper}, fix $\theta\geq1$
and $0<\epsilon<\int_0^1h^\uparrow\log f_\theta$.
There exist $c,C,\delta,A>0$ such that the following holds
for all sufficiently large $n$. For either
$\tau_n=\tau_{n,S,w}(b_n)$ or $\tau_n=\tau_{n,Z,w}(b_n)$,
if
\[
r\left(\int_0^1h^\uparrow(u)\log f_\theta(u)\,du-\epsilon\right)
\geq\mathcal B_{n,w}(b_n)+\theta\log2,
\]
then in the PFA regime, for
$A(\log T_n)^{1/2}\leq r\leq\lfloor\delta T_n/K_{\log T_n}(2T_n)\rfloor$,
\[
\mathbb P_{P_0,T_n,P_1}(\tau_n>T_n+r)
\leq Ce^{-cr}+2\exp\{-cT_n/K_{\log T_n}(2T_n)^2\}.
\]
In the ARL regime the same assertion holds with $\log b_n$
in place of $\log T_n$ in the range and probability bound.
For $\Bar{\tau}_{n,S,w}(b_n)$, increase the required boundary
by $\log\{1/(1-\rho)\}$ and keep the learned-rank stream and
its schedule fixed.
\end{proposition}

The preceding probability bound contains both a post-change exponential term and a calibration-error term. The following original comparison lemma concerns the fixed-score ranks in \eqref{eq: randomized pval}, not the transformed ranks in \eqref{eq:learned-pvalues}. It remains useful for the fixed-score results.

\begin{lemma}[Local comparison with the oracle transform]
\label{lem:local-comparison}
Under \eqref{eqn:data structure}, suppose $F_{P_0}$ and $H$ are continuous, $H(0)=0$, and form the conformal ranks with iid uniform randomizers independent of the array. For every $n$ and positive integer $r$, define
\[
V_{n,j}:=F_{P_0}(X_{n,j}),
\qquad T_n<j\leq T_n+r.
\]
Then $V_{n,T_n+1},\ldots,V_{n,T_n+r}$ are iid with CDF $H$, and
\begin{equation}
\label{eq:local-comparison}
\max_{T_n<j\leq T_n+r}|p_{n,j}-V_{n,j}|
\leq
\|\widehat F_{P_0,T_n}-F_{P_0}\|_\infty
+\frac{2r+1}{T_n}
\quad\text{almost surely},
\end{equation}
where
\[
\widehat F_{P_0,T_n}(x)
:=
\frac1{T_n}\sum_{i=1}^{T_n}\mathbf 1\{X_{n,i}\leq x\}.
\]
If $r=r_n$ is a deterministic positive-integer sequence with $r_n\to\infty$ and $r_n=o(T_n)$, the left-hand side of \eqref{eq:local-comparison} converges to zero in probability. For every bounded Lipschitz function $\varphi$,
\[
\frac1{r_n}\sum_{j=T_n+1}^{T_n+r_n}\varphi(p_{n,j})
\xrightarrow{\mathbb P}
\int_0^1\varphi(u)\,dH(u).
\]
If, in addition, $T_n/\log n\to\infty$ and $r_n/\log n\to\infty$, the preceding convergences hold almost surely.
\end{lemma}

\begin{proposition}[Null-bet modification]
\label{prop:weighted-null-atom}
Fix $\rho\in(0,1)$ and use \eqref{eq:learned-null-regularization}.
Suppose the conditions of Theorem~\ref{th:weighted-null-validity} hold. If $\|w\|_1<\infty$, then
$\mathbb P\{\Bar{\tau}_{S,w}(b)<\infty\}
\leq \frac{\|w\|_1}{b}.$
For every stopping time $\sigma$,
\[
\mathbb P\{\Bar{\tau}_{S,w}(b)<\infty,\,
\Bar{\tau}_{S,w}(b)\leq\sigma\}
\leq
\frac{\mathbb E W_{\Bar{\tau}_{S,w}(b)\wedge\sigma}}{b}.
\]
If, in addition, $0<\|w\|_\infty<\infty$, then
$\mathbb E\Bar{\tau}_{S,w}(b)
\geq \frac{b}{\|w\|_\infty}.$ Moreover, for every $t\geq1$,
\[
\overline S_t^{(w,\rho)}
=
\rho W_t+(1-\rho)S_t^{(w)}
\geq \rho W_t.
\]
Hence, if $W_t\to\infty$, then $\Bar{\tau}_{S,w}(b)
\leq
\inf\{t\geq1:\rho W_t\geq b\}
<\infty.$

Furthermore, for all $n,r\geq1$,
\[
\overline S_{n,T_n+r}^{(w,\rho)}
\geq
(1-\rho)w_{T_n+1}M_{n,T_n+1,T_n+r}.
\]
Hence every delay bound based on the local restart at $T_n+1$ continues to hold with effective boundary
$B_{n,w}(b_n)+\log\{1/(1-\rho)\}$.
\end{proposition}

Thus, when $W_t\to\infty$, the component $\rho W_t$ forces eventual threshold crossing. The applicable PFA and ARL inequalities remain unchanged. For fixed $\rho$, the added log-boundary term is constant, with no change to the stream or its learning schedule. This yields the same first-order bound wherever the local-restart argument applies.

\section{Proofs}
\label{Appendix: Proofs}

\textbf{Proof of Example~\ref{eg: CTM log T delay}.} For the $n$th row, let $p_{n,j}$ denote the conformal $p$-value
computed at time $j$, and write
\[
M_{n,k}
:=
f_k(p_{n,1},\ldots,p_{n,k})
=
\prod_{j=1}^k f_1(p_{n,j}),
\qquad k\geq 1.
\]
For brevity, set
\[
\tau_n
:=
\tau_n^{\texttt{CTM}}\left(\frac{1}{\alpha}\right)
=
\inf\left\{
k\geq 1:M_{n,k}\geq\frac{1}{\alpha}
\right\}.
\]
Because $X_{n,1},\ldots,X_{n,n}$ are IID under $P_0$, the
distribution-free property of randomized conformal $p$-values gives $p_{n,1},\ldots,p_{n,n}
\stackrel{\mathrm{iid}}{\sim}\operatorname{Unif}(0,1)$ under $\mathbb P_n.$
Indeed, the relative ordering of the first $n$ observations is a
uniformly distributed permutation, the corresponding sequential
insertion ranks are independent discrete uniform random variables,
and the auxiliary randomization makes each conformal $p$-value
uniform on $[0,1]$. Since $\log f_1(u)=u-\frac12-\kappa,
\kappa=\psi(1),$
we have
\[
\frac{1}{n}\log M_{n,n}
=
\frac{1}{n}\sum_{j=1}^n
\left(p_{n,j}-\frac12\right)-\kappa.
\]
For every $\varepsilon>0$, Chebyshev's inequality gives
\[
\mathbb P_n\left(
\left|
\frac{1}{n}\sum_{j=1}^n
\left(p_{n,j}-\frac12\right)
\right|>\varepsilon
\right)
\leq
\frac{1}{12n\varepsilon^2}
\longrightarrow 0.
\]
Consequently,
\[
\frac{1}{n}\log M_{n,n}
\xrightarrow{\mathbb P_n}
-\kappa.
\]
Moreover, $(M_{n,k})_{k=0}^n$ is a nonnegative martingale under
$\mathbb P_n$. Indeed, $\int_0^1f_1(u)\,du=1$ and the first $n$
conformal $p$-values are independent and uniform. Ville's inequality
therefore implies
\[
\mathbb P_n(\tau_n\leq n)
=
\mathbb P_n\left(
\max_{1\leq k\leq n}M_{n,k}\geq\frac{1}{\alpha}
\right)
\leq\alpha,
\]
and hence $\mathbb P_n(\tau_n>n)\geq1-\alpha.$
We next note that $0<\kappa<\frac12.$
The first inequality follows from strict Jensen's inequality:
\[
\int_0^1e^{u-1/2}\,du
>
\exp\left\{\int_0^1\left(u-\frac12\right)\,du\right\}
=1.
\]
The second follows from
\[
\int_0^1e^{u-1/2}\,du<e^{1/2}.
\]
Thus $b:=\frac12-\kappa>0,
c_{\mathrm{CTM}}=\frac{\kappa}{b}.$
Fix $c\in(0,c_{\mathrm{CTM}})$. Then $\kappa-cb>0$. Define
\[
\delta:=\frac{\kappa-cb}{2}>0
\]
and the event
\[
A_n
:=
\left\{
\log M_{n,n}\leq-(\kappa-\delta)n
\right\}.
\]
The convergence above gives $\mathbb P_n(A_n)\to1$. Since
$\mathbb P_n(\tau_n>n)\geq1-\alpha$,
\[
\mathbb P_n(A_n^c\mid\tau_n>n)
\leq
\frac{\mathbb P_n(A_n^c)}{1-\alpha}
\longrightarrow0.
\]
For every $u\in[0,1]$,
\[
\log f_1(u)
=
u-\frac12-\kappa
\leq
\frac12-\kappa
=b.
\]
Therefore, on $A_n$, uniformly over
$0\leq r\leq\lfloor cn\rfloor$,
\begin{align*}
\log M_{n,n+r}
&=
\log M_{n,n}
+
\sum_{j=n+1}^{n+r}\log f_1(p_{n,j})\\
&\leq
-(\kappa-\delta)n+rb\\
&\leq
-(\kappa-\delta)n+cnb =
-\delta n
<0.
\end{align*}
Because $\alpha\in(0,1)$, we have
$0<\log(1/\alpha)$. Hence, on
$A_n\cap\{\tau_n>n\}$, the CTM cannot cross the threshold
$1/\alpha$ at any time between $n+1$ and
$n+\lfloor cn\rfloor$. It follows that
\begin{align*}
&\mathbb P_n\left(
\tau_n>n+\lfloor cn\rfloor
\;\middle|\;
\tau_n>n
\right)\\
&\qquad\geq
\mathbb P_n(A_n\mid\tau_n>n)
\longrightarrow1.
\end{align*}
Finally, let $(r_n)$ be any deterministic nonnegative sequence
satisfying $r_n=o(n)$. For any fixed
$c\in(0,c_{\mathrm{CTM}})$, we have
$r_n\leq\lfloor cn\rfloor$ for all sufficiently large $n$. Therefore,
\[
\mathbb P_n\left(
\tau_n-n>r_n
\;\middle|\;
\tau_n>n
\right)
\geq
\mathbb P_n\left(
\tau_n>n+\lfloor cn\rfloor
\;\middle|\;
\tau_n>n
\right)
\longrightarrow1.
\]
Substituting back
$\tau_n=\tau_n^{\texttt{CTM}}(1/\alpha)$ proves the result. \hfill$\blacksquare$
\bigskip

\textbf{Proof of example \ref{ex:arl-single-bet-separation}.}
Fix $n$, write $T:=T_n$, $\gamma:=\gamma_n$, and $\mathbb P_n:=\mathbb P_{P_0,T_n,P_1}$, and let $\mathbb E_n$ denote expectation under $\mathbb P_n$. Suppress the row index in the observations, randomizers, and conformal p-values, and write $\tau^A:=\tau_n^A(\gamma_n)$ for $A\in\{\texttt{CUSUM},\texttt{SR}\}$. The conformal p-values are computed using the identity score and the independent uniform randomizers assumed throughout the paper. By \eqref{eq:VovkSR} and cancellation within the product martingale, the corresponding statistics are
\[
R_t:=\sum_{k=1}^t\prod_{j=k}^t f(p_j),
\qquad
C_t:=\max_{1\le k\le t}\prod_{j=k}^t f(p_j).
\]
In particular, $C_t\le R_t$ and $\tau^{\texttt{SR}}\le\tau^{\texttt{CUSUM}}$ pathwise.

Set $R_0:=0$ and $\mathcal G_t:=\sigma(p_1,\ldots,p_t)$. Under $\mathbb P_{P_0,\infty}$, the conformal p-values are iid uniform. Since $\int_0^1 f(u)\,du=1$, the recursion $R_t=f(p_t)(1+R_{t-1})$ gives
\[
\mathbb E_{P_0,\infty}[R_t\mid\mathcal G_{t-1}]
=
R_{t-1}+1.
\]
Thus $(R_t-t)_{t\ge0}$ is a martingale. At either alarm, $R_{\tau^A}\ge\gamma$, so optional sampling at the bounded stopping time $\tau^A\wedge T$ yields
\[
\gamma\,\mathbb P_{P_0,\infty}(\tau^A\le T)
\le
\mathbb E_{P_0,\infty}R_{\tau^A\wedge T}
=
\mathbb E_{P_0,\infty}(\tau^A\wedge T)
\le T.
\]
The alternative and no-change laws agree through time $T$. Consequently,
\[
\mathbb P_n(\tau^A>T)
\ge
1-\frac{T}{\gamma}
>0,
\]
and this lower bound tends to one.

The density $h$ is symmetric about $1/2$, and therefore $\mathbb E_{P_0}f(X)=\mathbb E_{P_1}f(X)=1$. Define
\[
q:=
\max\left\{
\int_0^1 f(u)^{3/4}\,du,\,
\int_0^1 h(u)f(u)^{3/4}\,du
\right\}.
\]
Strict concavity of $x\mapsto x^{3/4}$ and nonconstancy of $f(X)$ under both laws imply $0<q<1$. Put
\[
\eta:=-\frac{\log q}{6},
\qquad
c_*:=\frac{\eta}{4}=-\frac{\log q}{24},
\qquad
\bar q:=\sqrt q,
\qquad
r:=\lfloor c_*T\rfloor.
\]
The constants $\eta$, $c_*$, and $\bar q$ do not depend on $n$, and $r\ge1$ for all sufficiently large $n$.

Let $\widehat F_T$ denote the empirical CDF of $X_1,\ldots,X_T$, and define $E_T:=\{\|\widehat F_T-F_{P_0}\|_\infty\le\eta\}$. The Dvoretzky--Kiefer--Wolfowitz inequality gives
\[
\mathbb P_n(E_T^c)\le 2e^{-2T\eta^2}.
\]
Both distributions are continuous, so ties occur with probability zero. For $\ell\ge1$, splitting the global conformal rank into pre-change and post-change contributions gives
\[
p_{T+\ell}
=
\frac{T\widehat F_T(X_{T+\ell})+V_\ell}{T+\ell},
\qquad
V_\ell:=
\sum_{i=T+1}^{T+\ell-1}
\mathbf{1}\{X_i<X_{T+\ell}\}
+\lambda_{T+\ell}
\in[0,\ell].
\]
Since $F_{P_0}(X_{T+\ell})=X_{T+\ell}$ almost surely, it follows that
\[
\begin{aligned}
|p_{T+\ell}-X_{T+\ell}|
&\le
\frac{T}{T+\ell}
\|\widehat F_T-F_{P_0}\|_\infty
+
\frac{|V_\ell-\ell X_{T+\ell}|}{T+\ell}\\
&\le
\|\widehat F_T-F_{P_0}\|_\infty
+
\frac{\ell}{T+\ell}.
\end{aligned}
\]
In particular,
\[
\max_{T<j\le T+r}|p_j-X_j|
\le
\|\widehat F_T-F_{P_0}\|_\infty
+\frac{2r+1}{T}.
\]
On $E_T$, the right-hand side is at most $\eta+2c_*+1/T\le2\eta$ for all sufficiently large $n$. Since $\log f$ is $2$-Lipschitz on $[0,1]$, we obtain, simultaneously for $T<j\le T+r$,
\[
f(p_j)\le e^{4\eta}f(X_j)
\qquad\text{on }E_T.
\]

Define the comparison multipliers
\[
Y_j:=
\begin{cases}
f(p_j),&j\le T,\\
e^{4\eta}f(X_j),&j>T.
\end{cases}
\]
These variables are independent under $\mathbb P_n$: the pre-change conformal p-values are iid uniform, and the post-change observations are independent of the entire pre-change sample and its randomizers. Moreover,
\[
\mathbb E_nY_j^{3/4}
\le
\begin{cases}
q\le\bar q,&j\le T,\\
e^{3\eta}q=\bar q,&j>T.
\end{cases}
\]
Put $\widetilde R_t:=\sum_{k=1}^t\prod_{j=k}^tY_j$. On $E_T$, we have $R_t\le\widetilde R_t$ simultaneously for $1\le t\le T+r$. Subadditivity of $x\mapsto x^{3/4}$ and independence give the unconditional moment bound
\[
\begin{aligned}
\mathbb E_n\widetilde R_t^{3/4}
&\le
\sum_{k=1}^t
\mathbb E_n\left(\prod_{j=k}^tY_j\right)^{3/4}\\
&=
\sum_{k=1}^t
\prod_{j=k}^t\mathbb E_nY_j^{3/4}\\
&\le
\sum_{\ell=1}^t\bar q^\ell
\le
\frac{\bar q}{1-\bar q}.
\end{aligned}
\]
Hence Markov's inequality and a union bound yield
\[
\begin{aligned}
\mathbb P_n(\tau^{\texttt{SR}}\le T+r)
&\le
\mathbb P_n(E_T^c)
+
\sum_{t=1}^{T+r}
\mathbb P_n(\widetilde R_t\ge\gamma)\\
&\le
2e^{-2T\eta^2}
+
\frac{(T+r)\bar q}{(1-\bar q)\gamma^{3/4}}
\longrightarrow0,
\end{aligned}
\]
because $T=\lfloor\sqrt\gamma\rfloor$ and $r\le c_*T$. Since $\tau^{\texttt{SR}}\le\tau^{\texttt{CUSUM}}$, the same conclusion holds for $\tau^{\texttt{CUSUM}}$. Therefore, for each $A\in\{\texttt{CUSUM},\texttt{SR}\}$,
\[
\mathbb P_n(\tau^A>T+r\mid\tau^A>T)
=
\frac{\mathbb P_n(\tau^A>T+r)}
{\mathbb P_n(\tau^A>T)}
\longrightarrow1.
\]

The conditional expectations are finite. To see this, consider the iid no-change law $\mathbb P_{P_1,\infty}$ and put $L:=\lceil\log\gamma/\log(5/4)\rceil$. A block of $L$ consecutive conformal p-values in $(3/4,1)$ produces a product greater than $(5/4)^L\ge\gamma$, and hence forces a CUSUM alarm by the end of that block. Under $\mathbb P_{P_1,\infty}$, disjoint blocks have independent success events of probability $4^{-L}$. Thus
\[
\mathbb E_{P_1,\infty}\tau^{\texttt{SR}}
\le
\mathbb E_{P_1,\infty}\tau^{\texttt{CUSUM}}
\le
L4^L
<\infty.
\]
The randomizers have the same independent law under both measures, and the data laws differ only in their first $T$ coordinates. Since $h\ge1/2$ on $[0,1]$,
\[
\frac{d\mathbb P_n}{d\mathbb P_{P_1,\infty}}
=
\prod_{j=1}^T\frac{1}{h(X_j)}
\le 2^T.
\]
Consequently, $\mathbb E_n\tau^A\le2^TL4^L<\infty$. Together with $\mathbb P_n(\tau^A>T)>0$, this establishes finiteness of the conditional expectations for every $n$.

Finally,
\[
\mathbb E_n[\tau^A-T\mid\tau^A>T]
\ge
r\,\mathbb P_n(\tau^A>T+r\mid\tau^A>T).
\]
Since $r/\sqrt{\gamma_n}\to c_*$, the preceding conditional probability limit gives
\[
\liminf_{n\to\infty}
\frac{
\mathbb E_{P_0,T_n,P_1}
[\tau_n^A(\gamma_n)-T_n
 \mid\tau_n^A(\gamma_n)>T_n]
}{\sqrt{\gamma_n}}
\ge c_*.
\]
Also, $r/\log\gamma_n\to\infty$. For every fixed $K>0$, we therefore have $r>K\log\gamma_n$ for all sufficiently large $n$, and hence
\[
\mathbb P_n\left(
\frac{\tau^A-T}{\log\gamma_n}>K
\,\middle|\,
\tau^A>T
\right)
\ge
\mathbb P_n(\tau^A>T+r\mid\tau^A>T)
\longrightarrow1.
\]
Substituting back $\tau^A=\tau_n^A(\gamma_n)$ proves all the conclusions for both procedures.\hfill$\square$
\bigskip

\textbf{Proof of example \ref{ex:pfa-cmm-separation}.}
Put $\mathbb P_n:=\mathbb P_{P_0,T_n,P_1}$, and let $p_{n,j}$ denote the global conformal p-value in row $n$. Write $M_{n,t}$ for the CMM in \eqref{eq: conformal mixture martingale} evaluated on this row, with $M_{n,0}:=1$. For $\theta\in\mathbb R$, define
\[
\psi(\theta):=
\log\int_0^1e^{\theta(u-1/2)}\,du.
\]
Then
\[
f_\theta(u)
=
\exp\{\theta(u-1/2)-\psi(\theta)\},
\qquad
\theta\in\{-1/2,1\}.
\]
The function $\psi$ is even and strictly convex, with $\psi(0)=0$. Indeed, symmetry follows by the substitution $u\mapsto1-u$, and $\psi''(\theta)$ is the strictly positive variance of $u-1/2$ under the corresponding tilted density. Consequently, with $\kappa:=\psi(1/2)$,
\[
\psi(-1/2)=\kappa>0,
\qquad
\psi(1)>\kappa.
\]

Under $\mathbb P_n$, the first $T_n$ conformal p-values are iid uniform. Let
\[
\mathcal E_n:=
\left\{
\left|
\sum_{j=1}^{T_n}(p_{n,j}-1/2)
\right|
\le\frac{\kappa T_n}{2}
\right\}.
\]
Hoeffding's inequality gives
\[
\mathbb P_n(\mathcal E_n^c)
\le
2\exp\left\{-\frac{\kappa^2T_n}{2}\right\}.
\]
On $\mathcal E_n$, for either $\theta\in\{-1/2,1\}$,
\[
\begin{aligned}
\sum_{j=1}^{T_n}\log f_\theta(p_{n,j})
&=
\theta\sum_{j=1}^{T_n}(p_{n,j}-1/2)
-
T_n\psi(\theta)\\
&\le
\frac{\kappa T_n}{2}
-
\kappa T_n
=
-\frac{\kappa T_n}{2}.
\end{aligned}
\]
Averaging the two products therefore yields
\[
M_{n,T_n}\le e^{-\kappa T_n/2}
\qquad\text{on }\mathcal E_n.
\]

For every $u\in[0,1]$ and $\theta\in\{-1/2,1\}$,
\[
\log f_\theta(u)
=
\theta(u-1/2)-\psi(\theta)
\le\frac12.
\]
Thus, for every integer $\ell\ge0$, regardless of the post-change observations,
\[
\begin{aligned}
M_{n,T_n+\ell}
&=
\frac12
\sum_{\theta\in\{-1/2,1\}}
\left\{
\prod_{j=1}^{T_n}f_\theta(p_{n,j})
\prod_{j=T_n+1}^{T_n+\ell}f_\theta(p_{n,j})
\right\}\\
&\le
e^{\ell/2}M_{n,T_n}.
\end{aligned}
\]
Take $c_*:=\kappa/2$. On $\mathcal E_n$, uniformly over $0\le\ell\le\lfloor c_*T_n\rfloor$,
\[
M_{n,T_n+\ell}
\le
\exp\left\{-\frac{\kappa T_n}{2}+\frac{c_*T_n}{2}\right\}
=
e^{-\kappa T_n/4}
<
1
<
\frac1\alpha.
\]
Hence, on $\mathcal E_n$, no threshold crossing can occur between $T_n+1$ and $T_n+\lfloor c_*T_n\rfloor$.

Before time $T_n$, the CMM is a nonnegative martingale starting from one under $\mathbb P_n$. Ville's inequality gives
\[
\mathbb P_n\bigl(\tau_n^{\texttt{CMM}}(1/\alpha)>T_n\bigr)
\ge1-\alpha.
\]
Consequently,
\[
\begin{aligned}
&\mathbb P_n\left(
\tau_n^{\texttt{CMM}}(1/\alpha)>
T_n+\lfloor c_*T_n\rfloor
\,\middle|\,
\tau_n^{\texttt{CMM}}(1/\alpha)>T_n
\right)\\
&\qquad\ge
1-
\frac{\mathbb P_n(\mathcal E_n^c)}
{\mathbb P_n(\tau_n^{\texttt{CMM}}(1/\alpha)>T_n)}\\
&\qquad\ge
1-\frac{2e^{-\kappa^2T_n/2}}{1-\alpha}
\longrightarrow1.
\end{aligned}
\]
If $r_n=o(T_n)$ is deterministic and nonnegative, then $r_n\le\lfloor c_*T_n\rfloor$ for all sufficiently large $n$, so the same conditional probability limit holds with $r_n$ in place of $\lfloor c_*T_n\rfloor$. Taking $r_n=K\log T_n$ proves the final assertion.\hfill$\square$
\bigskip 

\textbf{Proof of example \ref{ex:arl-cmm-separation}.}
We prove the assertions for $\tau_n^{\texttt{CMM-SR}}(\gamma_n)$ and then deduce the corresponding conclusions for $\tau_n^{\texttt{CMM-CUSUM}}(\gamma_n)$. Put $\mathbb P_n:=\mathbb P_{P_0,T_n,P_1}$, and let $\mathbb E_n$ denote expectation under $\mathbb P_n$. Let $p_{n,j}$ denote the global conformal p-values and $\mathcal G_{n,t}:=\sigma(p_{n,1},\ldots,p_{n,t})$. For $\theta\in\{-1/2,1\}$, put
\[
L_{n,t}^{\theta}:=
\prod_{j=1}^t f_\theta(p_{n,j}),
\qquad
L_{n,0}^{\theta}:=1.
\]
The rowwise CMM is
\[
M_{n,t}
=
\frac12L_{n,t}^{-1/2}
+
\frac12L_{n,t}^{1}
>
0,
\]
so every ratio in \eqref{eq:CMM CUSUM SR} is well defined. Write
\[
R_{n,t}:=
\sum_{i=0}^{t-1}\frac{M_{n,t}}{M_{n,i}},
\qquad
R_{n,0}:=0.
\]
Since $\gamma_n>1$, the term $i=t$ in the CUSUM maximum cannot cause an alarm. Every CUSUM threshold crossing therefore implies an SR threshold crossing at that time or earlier, giving
\[
\tau_n^{\texttt{CMM-SR}}(\gamma_n)
\le
\tau_n^{\texttt{CMM-CUSUM}}(\gamma_n)
\qquad\text{pathwise}.
\]

Under any iid no-change law, the conformal p-values are iid uniform and $M_{n,t}$ is a positive martingale. The recursion
\[
R_{n,t}
=
\frac{M_{n,t}}{M_{n,t-1}}(1+R_{n,t-1})
\]
therefore implies
\[
\mathbb E_{P_0,\infty}[R_{n,t}\mid\mathcal G_{n,t-1}]
=
R_{n,t-1}+1.
\]
Optional sampling at $\tau_n^{\texttt{CMM-SR}}(\gamma_n)\wedge T_n$ gives
\[
\begin{aligned}
&\gamma_n\,
\mathbb P_{P_0,\infty}
\bigl(\tau_n^{\texttt{CMM-SR}}(\gamma_n)\le T_n\bigr)\\
&\qquad\le
\mathbb E_{P_0,\infty}
R_{n,\tau_n^{\texttt{CMM-SR}}(\gamma_n)\wedge T_n}\\
&\qquad=
\mathbb E_{P_0,\infty}
\bigl[\tau_n^{\texttt{CMM-SR}}(\gamma_n)\wedge T_n\bigr]
\le T_n.
\end{aligned}
\]
The alternative and no-change laws agree through time $T_n$. Hence
\[
\mathbb P_n\bigl(\tau_n^{\texttt{CMM-SR}}(\gamma_n)>T_n\bigr)
\ge
1-\frac{T_n}{\gamma_n}
>0,
\]
and this lower bound tends to one.

Define
\[
\psi(\theta):=
\log\int_0^1e^{\theta(u-1/2)}\,du,
\qquad
d:=\psi(1)-\psi(1/2)>0.
\]
The positivity of $d$ follows from the evenness and strict convexity of $\psi$. Let
\[
B_{n,t}:=\frac{L_{n,t}^{1}}{L_{n,t}^{-1/2}}.
\]
Since
\[
\log\frac{f_1(u)}{f_{-1/2}(u)}
=
\frac32(u-1/2)-d,
\]
we have
\[
\log B_{n,T_n}
=
\frac32\sum_{j=1}^{T_n}(p_{n,j}-1/2)-dT_n.
\]
Define
\[
\mathcal A_n:=
\left\{
\log B_{n,T_n}\le-\frac{dT_n}{2}
\right\}.
\]
The pre-change conformal p-values are iid uniform, so Hoeffding's inequality gives
\[
\mathbb P_n(\mathcal A_n^c)
\le
\mathbb P_n\left(
\sum_{j=1}^{T_n}(p_{n,j}-1/2)>\frac{dT_n}{3}
\right)
\le
\exp\left\{-\frac{2d^2T_n}{9}\right\}.
\]

We next choose a window on which the component $f_1$ remains negligible relative to $f_{-1/2}$. Define
\[
q:=
\max\left\{
\int_0^1f_{-1/2}(u)^{3/4}\,du,\,
\int_0^1f_1(u)f_{-1/2}(u)^{3/4}\,du
\right\}.
\]
Strict concavity of $x\mapsto x^{3/4}$ and nonconstancy of $f_{-1/2}$ give
\[
\int_0^1f_{-1/2}(u)^{3/4}\,du<1.
\]
Also,
\[
\int_0^1f_1(u)f_{-1/2}(u)\,du
=
\exp\{\psi(1/2)-\psi(1)-\psi(-1/2)\}
=
e^{-\psi(1)}
<
1.
\]
Jensen's inequality under $P_1$ consequently gives
\[
\int_0^1f_1(u)f_{-1/2}(u)^{3/4}\,du
\le
\left(
\int_0^1f_1(u)f_{-1/2}(u)\,du
\right)^{3/4}
<
1.
\]
Thus $0<q<1$. Put
\[
\eta:=-\frac23\log q,
\qquad
\bar q:=\sqrt q,
\qquad
c_*:=\min\left\{\eta,\frac d3\right\},
\qquad
r_n:=\lfloor c_*T_n\rfloor.
\]
All these constants, except $r_n$, are independent of $n$.

For every $u\in[0,1]$,
\[
\log\frac{f_1(u)}{f_{-1/2}(u)}
=
\frac32(u-1/2)-d
\le\frac34.
\]
Therefore, on $\mathcal A_n$, uniformly over $0\le\ell\le r_n$,
\[
\log B_{n,T_n+\ell}
\le
-\frac{dT_n}{2}+\frac{3\ell}{4}
\le
-\frac{dT_n}{4}.
\]
In particular, $B_{n,t}\le1$ for $T_n\le t\le T_n+r_n$ on $\mathcal A_n$. For these $t$ and every $0\le i<t$,
\[
\frac{M_{n,t}}{M_{n,i}}
=
\frac{L_{n,t}^{-1/2}}{L_{n,i}^{-1/2}}
\frac{1+B_{n,t}}{1+B_{n,i}}
\le
2\frac{L_{n,t}^{-1/2}}{L_{n,i}^{-1/2}}.
\]
Hence
\[
R_{n,t}
\le
2\sum_{k=1}^t\prod_{j=k}^t f_{-1/2}(p_{n,j}),
\qquad
T_n\le t\le T_n+r_n,
\qquad
\text{on }\mathcal A_n.
\]

Let $\widehat F_{n,T_n}$ be the empirical CDF of $X_{n,1},\ldots,X_{n,T_n}$, and define
\[
\mathcal D_n:=
\left\{
\|\widehat F_{n,T_n}-F_{P_0}\|_\infty\le\eta
\right\}.
\]
The Dvoretzky--Kiefer--Wolfowitz inequality gives
\[
\mathbb P_n(\mathcal D_n^c)\le2e^{-2T_n\eta^2}.
\]
Both data laws are continuous, so ties occur with probability zero. For $\ell\ge1$, splitting the global rank at $T_n$ gives
\[
p_{n,T_n+\ell}
=
\frac{
T_n\widehat F_{n,T_n}(X_{n,T_n+\ell})+V_{n,\ell}
}{T_n+\ell},
\]
where
\[
V_{n,\ell}
:=
\sum_{j=T_n+1}^{T_n+\ell-1}
\mathbf1\{X_{n,j}<X_{n,T_n+\ell}\}
+
\lambda_{n,T_n+\ell}
\in[0,\ell].
\]
Since $F_{P_0}(X_{n,T_n+\ell})=X_{n,T_n+\ell}$ almost surely,
\[
|p_{n,T_n+\ell}-X_{n,T_n+\ell}|
\le
\|\widehat F_{n,T_n}-F_{P_0}\|_\infty
+
\frac{\ell}{T_n+\ell}.
\]
Thus, for all sufficiently large $n$, on $\mathcal D_n$,
\[
\max_{T_n<j\le T_n+r_n}|p_{n,j}-X_{n,j}|
\le
\eta+\frac{r_n}{T_n}
\le
2\eta.
\]
The function $\log f_{-1/2}$ is $1/2$-Lipschitz. Consequently,
\[
f_{-1/2}(p_{n,j})
\le
e^\eta f_{-1/2}(X_{n,j}),
\qquad
T_n<j\le T_n+r_n,
\qquad
\text{on }\mathcal D_n.
\]

Define comparison multipliers
\[
Y_{n,j}:=
\begin{cases}
f_{-1/2}(p_{n,j}),&j\le T_n,\\
e^\eta f_{-1/2}(X_{n,j}),&j>T_n.
\end{cases}
\]
For each $n$, these variables are independent under $\mathbb P_n$: the pre-change conformal p-values are iid uniform, and the post-change observations are independent of the entire pre-change sample and its randomizers. Moreover,
\[
\mathbb E_nY_{n,j}^{3/4}
\le
\begin{cases}
q\le\bar q,&j\le T_n,\\
e^{3\eta/4}q=\bar q,&j>T_n.
\end{cases}
\]
Put
\[
\widetilde R_{n,t}:=
\sum_{k=1}^t\prod_{j=k}^tY_{n,j}.
\]
Subadditivity of $x\mapsto x^{3/4}$ and independence yield
\[
\begin{aligned}
\mathbb E_n\widetilde R_{n,t}^{3/4}
&\le
\sum_{k=1}^t
\prod_{j=k}^t\mathbb E_nY_{n,j}^{3/4}\\
&\le
\sum_{\ell=1}^t\bar q^\ell
\le
\frac{\bar q}{1-\bar q}.
\end{aligned}
\]
This is an unconditional moment bound; no independence after conditioning on $\mathcal A_n$ or $\mathcal D_n$ is required. On $\mathcal A_n\cap\mathcal D_n$,
\[
R_{n,t}\le2\widetilde R_{n,t},
\qquad
T_n<t\le T_n+r_n.
\]
The early-alarm bound, Markov's inequality, and a union bound therefore give
\[
\begin{aligned}
&\mathbb P_n\bigl(
\tau_n^{\texttt{CMM-SR}}(\gamma_n)\le T_n+r_n
\bigr)\\
&\quad\le
\frac{T_n}{\gamma_n}
+
\mathbb P_n(\mathcal A_n^c)
+
\mathbb P_n(\mathcal D_n^c)
+
\sum_{t=T_n+1}^{T_n+r_n}
\mathbb P_n\left(\widetilde R_{n,t}\ge\frac{\gamma_n}{2}\right)\\
&\quad\le
\frac{T_n}{\gamma_n}
+
e^{-2d^2T_n/9}
+
2e^{-2T_n\eta^2}
+
\frac{2^{3/4}r_n\bar q}
{(1-\bar q)\gamma_n^{3/4}}
\longrightarrow0.
\end{aligned}
\]
Here $T_n=\lfloor\sqrt{\gamma_n}\rfloor$ and $r_n\le c_*T_n$. Dividing the resulting late-alarm probability by the survival probability yields
\[
\mathbb P_n\left(
\tau_n^{\texttt{CMM-SR}}(\gamma_n)>T_n+r_n
\,\middle|\,
\tau_n^{\texttt{CMM-SR}}(\gamma_n)>T_n
\right)
\longrightarrow1.
\]

We next establish finiteness of the expectations. Since $f_{-1/2}$ is strictly decreasing and $f_1$ is strictly increasing, and both integrate to one, $f_{-1/2}(0)>1$ and $f_1(1)>1$. By continuity, there exist constants $\varepsilon\in(0,1/2)$ and $\beta>1$ such that
\[
f_{-1/2}(u)\ge\beta
\quad\text{for }u\in[0,\varepsilon],
\qquad
f_1(u)\ge\beta
\quad\text{for }u\in[1-\varepsilon,1].
\]
Put
\[
L_n:=
\left\lceil\frac{\log(2\gamma_n)}{\log\beta}\right\rceil.
\]
Consider any iid no-change law $\mathbb P_{P,\infty}$ and a block beginning immediately after time $s$. Conditional on $\mathcal G_{n,s}$, at least one of
\[
\frac{L_{n,s}^{-1/2}}{2M_{n,s}},
\qquad
\frac{L_{n,s}^{1}}{2M_{n,s}}
\]
is at least $1/2$. Choose the corresponding interval $[0,\varepsilon]$ or $[1-\varepsilon,1]$. With conditional probability $\varepsilon^{L_n}$, the next $L_n$ conformal p-values all fall in that interval. On this event,
\[
\begin{aligned}
\frac{M_{n,s+L_n}}{M_{n,s}}
&=
\sum_{\theta\in\{-1/2,1\}}
\frac{L_{n,s}^{\theta}}{2M_{n,s}}
\prod_{j=s+1}^{s+L_n}f_\theta(p_{n,j})\\
&\ge
\frac12\beta^{L_n}
\ge
\gamma_n.
\end{aligned}
\]
A CUSUM alarm must therefore occur by the end of that block. Iterating this conditional bound over successive blocks gives
\[
\mathbb P_{P,\infty}\left(
\tau_n^{\texttt{CMM-CUSUM}}(\gamma_n)>mL_n
\right)
\le
(1-\varepsilon^{L_n})^m,
\qquad m\ge0.
\]
Consequently,
\[
\mathbb E_{P,\infty}
\tau_n^{\texttt{CMM-SR}}(\gamma_n)
\le
\mathbb E_{P,\infty}
\tau_n^{\texttt{CMM-CUSUM}}(\gamma_n)
\le
L_n\varepsilon^{-L_n}
<
\infty.
\]

In fact, the null ARL of the SR alarm is of order $\gamma_n$. Since $f_\theta(u)\le e^{1/2}$ for both densities,
\[
\frac{M_{n,t}}{M_{n,t-1}}\le e^{1/2}.
\]
The SR recursion therefore gives
\[
R_{n,\tau_n^{\texttt{CMM-SR}}(\gamma_n)\wedge m}
\le
e^{1/2}(1+\gamma_n)
\qquad\text{for every }m.
\]
Dominated convergence in the stopped-martingale identity yields
\[
\gamma_n
\le
\mathbb E_{P,\infty}
\tau_n^{\texttt{CMM-SR}}(\gamma_n)
=
\mathbb E_{P,\infty}
R_{n,\tau_n^{\texttt{CMM-SR}}(\gamma_n)}
\le
e^{1/2}(1+\gamma_n).
\]
The CUSUM ARL lower bound follows from the pathwise ordering of the alarms.

Under the alternative, the observations and randomizers have the same law as under $\mathbb P_{P_1,\infty}$ except for the first $T_n$ observations. Their likelihood ratio satisfies
\[
\frac{d\mathbb P_n}{d\mathbb P_{P_1,\infty}}
=
\prod_{j=1}^{T_n}\frac1{f_1(X_{n,j})}
=
(e-1)^{T_n}
\exp\left\{-\sum_{j=1}^{T_n}X_{n,j}\right\}
\le
(e-1)^{T_n}.
\]
It follows that
\[
\mathbb E_n\tau_n^{\texttt{CMM-CUSUM}}(\gamma_n)
\le
(e-1)^{T_n}L_n\varepsilon^{-L_n}
<
\infty,
\]
and the same finiteness conclusion holds for CMM-SR. Since the survival probabilities are strictly positive, both conditional expectations are finite.

Finally,
\[
\begin{aligned}
&\mathbb E_n\left[
\tau_n^{\texttt{CMM-SR}}(\gamma_n)-T_n
\,\middle|\,
\tau_n^{\texttt{CMM-SR}}(\gamma_n)>T_n
\right]\\
&\qquad\ge
r_n\,
\mathbb P_n\left(
\tau_n^{\texttt{CMM-SR}}(\gamma_n)>T_n+r_n
\,\middle|\,
\tau_n^{\texttt{CMM-SR}}(\gamma_n)>T_n
\right).
\end{aligned}
\]
Because $r_n/\sqrt{\gamma_n}\to c_*$, this proves the expected-delay lower bound. Also $r_n/\log\gamma_n\to\infty$, so for every fixed $K>0$,
\[
\begin{aligned}
&\mathbb P_n\left(
\frac{\tau_n^{\texttt{CMM-SR}}(\gamma_n)-T_n}
{\log\gamma_n}>K
\,\middle|\,
\tau_n^{\texttt{CMM-SR}}(\gamma_n)>T_n
\right)\\
&\qquad\ge
\mathbb P_n\left(
\tau_n^{\texttt{CMM-SR}}(\gamma_n)>T_n+r_n
\,\middle|\,
\tau_n^{\texttt{CMM-SR}}(\gamma_n)>T_n
\right)
\longrightarrow1.
\end{aligned}
\]

For CMM-CUSUM, the pathwise ordering gives
\[
\mathbb P_n\bigl(
\tau_n^{\texttt{CMM-CUSUM}}(\gamma_n)\le T_n+r_n
\bigr)
\le
\mathbb P_n\bigl(
\tau_n^{\texttt{CMM-SR}}(\gamma_n)\le T_n+r_n
\bigr)
\longrightarrow0,
\]
and its survival probability is at least $1-T_n/\gamma_n$. Thus its conditional late-alarm probability also tends to one. Applying the same conditional tail lower bound and using $r_n/\sqrt{\gamma_n}\to c_*$ and $r_n/\log\gamma_n\to\infty$ proves the expected-delay and logarithmic-scale conclusions for CMM-CUSUM as well.\hfill$\square$
\bigskip

\textbf{Proof of Lemma~\ref{lem:predictable-rank-rearrangement}.}
We first justify the original randomized-rank assertion, including
ties. Conditional on the first $m$ scores, at time $t\leq m$
the fractional randomizer selects uniformly one of the insertion
positions among scores tied with $s(X_t)$, and its fractional
part within the selected position is independently uniform.
Indeed, if $d$ observations among the first $t$ share that score,
then $\lfloor d\lambda_t\rfloor$ is uniform on
$\{0,\ldots,d-1\}$ and independent of
$d\lambda_t-\lfloor d\lambda_t\rfloor\sim U(0,1)$.
These pairs are conditionally independent across $t$.
Successive uniform insertion within a tie group gives a uniform
ordering of its labels. Exchangeability then makes the full
ordering of the $m$ labels uniform over all $m!$ permutations.
The bijection between label permutations and insertion sequences
in $\prod_{t=1}^m[t]$ shows that the insertion positions are
independent and uniform on their respective sets. Their fractional
parts are independent uniforms and independent of the positions.
Thus the original randomized ranks are iid uniform.

For fixed $t$ and a permutation $\pi$, the map that sends
$u\in[(r-1)/t,r/t)$ to $u+(\pi(r)-r)/t$ translates that interval
onto the $\pi(r)$th interval of the same length. It is an
invertible Lebesgue-measure-preserving map. Conditional on the
preceding original ranks, $\pi_t$ is fixed and the next original
rank is uniform. Its image \eqref{eq:learned-pvalues} is therefore
conditionally uniform. Induction proves that the transformed
ranks are iid uniform.

The transformed ranks through time $t$ are functions of the original
ranks through $t$. Conversely, once the original ranks through
$t-1$ have been recovered, $\pi_t$ is known. The inverse formula is
\[
p_t(s(X_1),\ldots,s(X_t);\lambda_t)
=p_t+\frac{\pi_t^{-1}(1+\lfloor tp_t\rfloor)-1-\lfloor tp_t\rfloor}{t}.
\]
This proves equality of the two filtrations by induction, beginning
with the trivial initial sigma-fields. The inverse formula is
algebraic and remains valid under the changepoint law. Finally,
\eqref{eq:learned-histogram} uses only ranks preceding $a_j$;
for $t\geq a_j$, the corresponding permutation is measurable
with respect to the required preceding-rank sigma-field.
\hfill$\blacksquare$
\bigskip

\textbf{Proof of Proposition~\ref{prop:nonoracle-betting-class}.}
Every continuous piecewise-linear function on the compact interval
is bounded and Lipschitz. Subtracting its log-normalizing constant
preserves these properties. This proves the first assertion.

Let $g$ be a nondecreasing density of finite entropy. For an integer
$m\geq1$, define
\[
a_m(u)=\max\{-m,\min(m,\log g(u))\},\qquad \log0=-\infty.
\]
This function is bounded and nondecreasing. Since $e^{a_m}\leq g+1$
and $e^{a_m}\to g$ almost everywhere, dominated convergence gives
$\int e^{a_m}\to1$. On $\{g>0\}$ we also have
$|g a_m|\leq g|\log g|$.
The dominating function is integrable: its negative-entropy part
is bounded by $e^{-1}$ and its positive part is integrable by
assumption. Thus
\[
\int_0^1ga_m-\log\int_0^1e^{a_m}\longrightarrow\int_0^1g\log g.
\]
For fixed $m$, on a rational uniform grid of mesh $1/q$, round
$a_m$ upward to multiples of $1/q$ and interpolate linearly.
These functions are nondecreasing, have rational data, are bounded
by $m+1$ in absolute value, and converge at every continuity point
of $a_m$. The countably many possible discontinuities have zero
measure both for Lebesgue measure and for $g(u)\,du$.
Dominated convergence shows that their normalized exponential
variational values approach $\int ga_m-\log\int e^{a_m}$.
They belong to the enumerated class. Taking first this approximation
and then $m\to\infty$ proves the lower bound on the supremum.
For each $\theta$, nonnegativity of relative entropy gives
\[
\int_0^1g\log f_\theta
=\int_0^1g\log g-D_{\mathrm{KL}}(g(u)\,du\|f_\theta(u)\,du)
\leq\int_0^1g\log g.
\]
This proves equality. If $\theta\leq k$, its term in
\eqref{eq:finite-restart-mixture} is
$2^{-\theta}\prod_{i=k}^tf_\theta(p_i)$; dropping all other
nonnegative terms and taking logarithms proves the final assertion.
\hfill$\blacksquare$
\bigskip

\textbf{Proof of Theorem~\ref{th:weighted-null-validity}.}
Let $\mathbb P$ be the exchangeable null law including the
independent randomizers. Proposition~\ref{prop: pval iid} or
Lemma~\ref{lem:predictable-rank-rearrangement}, respectively,
makes the ranks iid uniform in $\mathcal G_t$.
For fixed $k,\theta$, independence and $\int_0^1f_\theta=1$ give
\[
\mathbb E_{\mathbb P}\left[\prod_{j=k}^tf_\theta(p_j)
\mid\mathcal G_{t-1}\right]=\prod_{j=k}^{t-1}f_\theta(p_j).
\]
The empty product is one. Conditional Tonelli proves the same
identity after integration against any deterministic probability
measure at restart $k$, including the finite measure in
\eqref{eq:finite-restart-mixture}. Thus each $M_{k,t}$ is a
nonnegative martingale of mean one, extended by one before $k$.

For summable weights, Tonelli and conditional Tonelli show that
$\widetilde S_t^{(w)}$ is a nonnegative martingale with mean
$\|w\|_1$. Since $Z_t^{(w)}\leq S_t^{(w)}\leq\widetilde S_t^{(w)}$,
Ville's inequality gives \eqref{eq:weighted-PFA-bound}.
For arbitrary weights every $W_t$ at finite $t$ is finite, and
\[
\mathbb E_{\mathbb P}[S_t^{(w)}\mid\mathcal G_{t-1}]
=S_{t-1}^{(w)}+w_t.
\]
Thus $S_t^{(w)}-W_t$ is an integrable martingale and for any
bounded stopping time $\eta$,
$\mathbb E_{\mathbb P}S_\eta^{(w)}=\mathbb E_{\mathbb P}W_\eta$.
If $\mathbb E_{\mathbb P}\tau<\infty$, then $\tau<\infty$
almost surely, so Fatou and monotone convergence give
\[
\mathbb E_{\mathbb P}S_\tau^{(w)}
\leq\liminf_{m\to\infty}\mathbb E_{\mathbb P}S_{\tau\wedge m}^{(w)}
=\lim_{m\to\infty}\mathbb E_{\mathbb P}W_{\tau\wedge m}
=\mathbb E_{\mathbb P}W_\tau.
\]
The inequality for $Z$ follows by domination.

Fix either alarm $\tau_\star$ and any stopping time $\sigma$.
At $\eta_m=\tau_\star\wedge\sigma\wedge m$, on
$\{\tau_\star\leq\sigma,\tau_\star\leq m\}$ the sum is at least
$b$, even when $\tau_\star$ is the maximum alarm. Therefore
\[
b\,\mathbb P(\tau_\star\leq\sigma,\tau_\star\leq m)
\leq\mathbb E_{\mathbb P}S_{\eta_m}^{(w)}
=\mathbb E_{\mathbb P}W_{\eta_m}.
\]
Let $m\to\infty$. The events increase to
$\{\tau_\star<\infty,\tau_\star\leq\sigma\}$ and the weights
increase to $W_{\tau_\star\wedge\sigma}$. This proves
\eqref{eq:weighted-finite-horizon-bound}.
If $0<\|w\|_\infty<\infty$, use
$W_\tau\leq\|w\|_\infty\tau$ to obtain the $e$-detector
inequalities. If an alarm has infinite expectation the ARL
lower bound is immediate; otherwise it is finite almost surely
and $S_{\tau_\star}^{(w)}\geq b$. Hence
\[
b\leq\mathbb E_{\mathbb P}S_{\tau_\star}^{(w)}
\leq\mathbb E_{\mathbb P}W_{\tau_\star}
\leq\|w\|_\infty\mathbb E_{\mathbb P}\tau_\star.
\]
Finally, $W_{\tau_\star\wedge\sigma}\leq\|w\|_\infty\sigma$
gives the last displayed bound in the theorem.
\hfill$\blacksquare$
\bigskip

\textbf{Proof of Proposition~\ref{prop:fixed-changepoint-impossibility}.}
For the first part, fix $P_0\in\M$. Choose $S\in\M\setminus\{P_0\}$ and, for $\varepsilon\in(0,1)$, define
$$
P_{1,\varepsilon}:=(1-\varepsilon)P_0+\varepsilon S.
$$
By convexity, $P_{1,\varepsilon}\in\M$, and $P_{1,\varepsilon}\neq P_0$. Since
$$
P_{1,\varepsilon}^{\otimes T}\ge(1-\varepsilon)^T P_0^{\otimes T},
$$
we have
$
\mathbb{P}_{P_{1,\varepsilon},\infty}
\ge
(1-\varepsilon)^T
\mathbb{P}_{P_0,T,P_{1,\varepsilon}}.
$
Therefore,
$$
\mathbb{P}_{P_0,T,P_{1,\varepsilon}}(\tau<\infty)
\le
\frac{\mathbb{P}_{P_{1,\varepsilon},\infty}(\tau<\infty)}
{(1-\varepsilon)^T}
\le
\frac{\alpha}{(1-\varepsilon)^T}.
$$
Taking the infimum over $P_1\in\M\setminus\{P_0\}$ and then letting $\varepsilon\downarrow0$ gives
$$
\inf_{P_1\in\M\setminus\{P_0\}}
\mathbb{P}_{P_0,T,P_1}(\tau<\infty)
\le\alpha.
$$
Now suppose $P_1\ll P_0$. Since $T<\infty$,
$
\mathbb{P}_{P_1,\infty}\ll\mathbb{P}_{P_0,T,P_1}.
$
Moreover,
$$
\mathbb{P}_{P_1,\infty}(\tau=\infty)
\ge1-\alpha>0.
$$
Absolute continuity therefore implies
$
\mathbb{P}_{P_0,T,P_1}(\tau=\infty)>0,
$
and hence
$
\mathbb{P}_{P_0,T,P_1}(\tau<\infty)<1.
$
For the second part, fix $P_1\in\M$ and, for $b>0$, define
$$
\sigma_b:=\inf\{n\ge0:E_n\ge b\}.
$$
For every $m\in\mathbb{N}$,
$$
b\,\mathbb{P}_{P_1,\infty}(\sigma_b\le m)
\le
\mathbb{E}_{P_1,\infty}[E_{\sigma_b\wedge m}]
\le1.
$$
Letting $m\to\infty$ gives
$$
\mathbb{P}_{P_1,\infty}(\sigma_b<\infty)\le\frac{1}{b}.
$$
Since
$$
\left\{\sup_{n\ge0}E_n=\infty\right\}
\subseteq
\{\sigma_b<\infty\}
$$
for every $b>0$, letting $b\to\infty$ yields
$$
\mathbb{P}_{P_1,\infty}\left(\sup_{n\ge0}E_n=\infty\right)=0.
$$
If $P_0\ll P_1$, then
$
\mathbb{P}_{P_0,T,P_1}\ll\mathbb{P}_{P_1,\infty}.
$
Consequently,
$$
\mathbb{P}_{P_0,T,P_1}\left(\sup_{n\ge0}E_n=\infty\right)=0,
$$
or equivalently,
$
\mathbb{P}_{P_0,T,P_1}\left(\sup_{n\ge0}E_n<\infty\right)=1.
$
\hfill$\blacksquare$
\bigskip

\textbf{Proof of Theorem~\ref{th:weighted-KL-upper}.}
Lemma~\ref{lem:stable-oracle-ordering} identifies the entropy of
$h^\uparrow$ with $D_{\mathrm{KL}}(P_1\|P_0)$. Fix the stated $A$.
Proposition~\ref{prop:nonoracle-betting-class} supplies a fixed
$\theta$ and an $\epsilon>0$ such that
\[
A\left(\int_0^1h^\uparrow(u)\log f_\theta(u)\,du-\epsilon\right)>1.
\]
Put $r_n=\lfloor A\mathcal B_{n,w}(b_n)\rfloor$.
In the PFA regime apply Lemma~\ref{lem:learned-score-growth}
with $B=\log T_n$ and $T=T_n$. The block conditions hold for
this schedule; the theorem's scale conditions give
$r_n/(\log T_n)^{1/2}\to\infty$ and
$r_nK_{\log T_n}(2T_n)/T_n\to0$. The latter follows from
$K_{\log T_n}(2T_n)\leq(\log T_n)^{1/8}$.
Thus $r_n$ lies in the range of the lemma for all sufficiently
large $n$. Its probability bound tends to zero because
$T_n/K_{\log T_n}(2T_n)^2\to\infty$.
In the ARL regime use $B=\log b_n$ instead. The stated conditions
again put $r_n$ in the admissible range, and
\[
\frac{T_n}{(\log b_n)K_{\log b_n}(2T_n)^2}
\geq2^{-1/4}\left(\frac{T_n}{\log b_n}\right)^{3/4}\to\infty.
\]
In either regime the lemma proves, with probability tending to one,
\[
\sum_{t=T_n+1}^{T_n+r_n}\log f_\theta(p_{n,t})
\geq r_n\left(\int h^\uparrow\log f_\theta-\epsilon\right).
\]
Since $\theta$ is fixed, eventually $\theta\leq T_n+1$.
The finite-mixture lower bound therefore gives
\[
\begin{split}
\log\{w_{T_n+1}M_{n,T_n+1,T_n+r_n}/b_n\}
\geq {}&r_n\left(\int h^\uparrow\log f_\theta-\epsilon\right)\\
&-\mathcal B_{n,w}(b_n)-\theta\log2\longrightarrow\infty
\end{split}
\]
on that event. This local term is in both weighted statistics,
so both crossing probabilities tend to one.

For either alarm $\tau_n$, the null and alternative agree through
$T_n$. The finite-horizon validity bound yields
$\mathbb P_{P_0,T_n,P_1}(\tau_n>T_n)\geq q_{n,w}(b_n)$.
Dividing the vanishing late-alarm probability by this lower bound
proves the conditional assertion when its lower limit is positive.
Finally,
\[
\mathbb P_{P_0,T_n,P_1}(T_n<\tau_n\leq T_n+r_n)
\geq q_{n,w}(b_n)-\mathbb P_{P_0,T_n,P_1}(\tau_n>T_n+r_n),
\]
which proves the correct-window assertion.
\hfill$\blacksquare$
\bigskip

\textbf{Proof of Corollary~\ref{cor:RCMM-pointwise-KL}.}
The polynomial weights give
$\mathcal B_{n,w}(1/\alpha)=(1+\eta)\log T_n+O(1)$.
Both PFA scale conditions in Theorem~\ref{th:weighted-KL-upper}
follow immediately. Choose
$A'\in(1/D_{\mathrm{KL}}(P_1\|P_0),A/(1+\eta))$.
Then $A'\mathcal B_{n,w}(1/\alpha)\leq A\log T_n$ eventually,
so the theorem gives the crossing assertion. The weights have
total mass at most one, and hence
$q_{n,w}(1/\alpha)\geq1-\alpha$. The conditional and correct-window
assertions follow from the same theorem.
\hfill$\blacksquare$
\bigskip

\textbf{Proof of Remark~\ref{remark:near-harmonic-KL}.}
The weights telescope:
$\sum_{k=1}^mw_k=1-1/\log(e+m)\to1$.
Moreover,
\[
w_k=\int_{e+k-1}^{e+k}\frac{dx}{x(\log x)^2}.
\]
Monotonicity of the integrand bounds this integral between its
values at the two endpoints. Thus
$\mathcal B_{n,w}(1/\alpha)=\log T_n+2\log\log T_n+O(1)$.
The same expansion holds for the previously normalized
near-harmonic weights. The PFA scale conditions again hold.
For $A'\in(1/D_{\mathrm{KL}}(P_1\|P_0),A)$, eventually
$A'\mathcal B_{n,w}(1/\alpha)\leq A(\log T_n+2\log\log T_n)$.
Apply Theorem~\ref{th:weighted-KL-upper}, and use
$q_{n,w}(1/\alpha)\geq1-\alpha$ for the remaining assertions.
\hfill$\blacksquare$
\bigskip

\textbf{Proof of Corollary~\ref{cor:ARL-KL-sharp}.}
With unit weights $\mathcal B_{n,w}(\gamma_n)=\log\gamma_n$.
The ARL conditions of Theorem~\ref{th:weighted-KL-upper} hold:
the nontrivial product condition follows from
\[
K_{\log\gamma_n}(2T_n)\leq(2T_n/\log\gamma_n)^{1/8},
\qquad\log\gamma_n=o(T_n).
\]
The theorem gives the crossing conclusion. The null and
alternative agree before the change, and the null horizon
bound yields survival probability at least $1-T_n/\gamma_n$.
When $T_n=o(\gamma_n)$ this tends to one, proving the two
remaining probability conclusions.

For the conditional expectation fix $A>1/D_{\mathrm{KL}}(P_1\|P_0)$.
Choose a fixed $\theta$ and $\epsilon>0$ such that
$A(\int h^\uparrow\log f_\theta-\epsilon)>1$.
Set $r_n=\lceil A\log\gamma_n\rceil$. Eventually, for every
$r\geq r_n$,
\[
r\left(\int h^\uparrow\log f_\theta-\epsilon\right)
\geq\log\gamma_n+\theta\log2+\log\frac1{1-\rho}.
\]
Use Lemma~\ref{lem:learned-score-growth} with $B=\log\gamma_n$
and the local term starting at $T_n+1$. If
$R_n=\lfloor\delta T_n/K_{\log\gamma_n}(2T_n)\rfloor$, then
for all integers $r_n\leq r\leq R_n$,
\[
\mathbb P_{P_0,T_n,P_1}\{\Bar{\tau}_{n,S,w}(\gamma_n)>T_n+r\}
\leq Ce^{-cr}+2\exp\{-cT_n/K_{\log\gamma_n}(2T_n)^2\}.
\]
Here the lower endpoint is eventually admissible because
$r_n/(\log\gamma_n)^{1/2}\to\infty$. Also,
\[
R_n/\log\gamma_n\to\infty,\qquad
\frac{T_n}{(\log\gamma_n)K_{\log\gamma_n}(2T_n)^2}
\geq2^{-1/4}(T_n/\log\gamma_n)^{3/4}\to\infty.
\]
The regularized mixture remains valid under the null, so its
survival probability at $T_n$ is at least $1-T_n/\gamma_n$.
It is bounded pathwise by $\lceil\gamma_n/\rho\rceil$.
Thus $r_n<R_n<\lceil\gamma_n/\rho\rceil$ eventually when
$T_n=o(\gamma_n)$. The tail-sum formula, with its first $r_n$
conditional probabilities bounded by one, gives
\[
\begin{split}
&\mathbb E_{P_0,T_n,P_1}[\Bar{\tau}_{n,S,w}(\gamma_n)-T_n
\mid\Bar{\tau}_{n,S,w}(\gamma_n)>T_n]\\
&\quad\leq r_n+\frac{1}{1-T_n/\gamma_n}
\left[
\frac{Ce^{-cr_n}}{1-e^{-c}}
+C\lceil\gamma_n/\rho\rceil e^{-cR_n}
+4\lceil\gamma_n/\rho\rceil
 e^{-cT_n/K_{\log\gamma_n}(2T_n)^2}
\right].
\end{split}
\]
To obtain this bound, sum the displayed tail estimate for
$r_n\leq r<R_n$, and bound all remaining tail probabilities
by the one at $R_n$. The factor $4$ bounds the calibration
contributions from both ranges. Each bracketed term tends to
zero by the two preceding divergence statements and
$\lceil\gamma_n/\rho\rceil=O(\gamma_n)$.
Divide by $\log\gamma_n$ and let
$A\downarrow1/D_{\mathrm{KL}}(P_1\|P_0)$.

For uniformity over $T\in I_{N_\gamma}(\kappa)$, all relevant
bin counts lie between a constant multiple of
$K_{\log\gamma}(\kappa N_\gamma/3)$ and a constant multiple
of $K_{\log\gamma}(2N_\gamma)$. The former tends to infinity.
Every approximation and concentration step of the growth lemma
therefore holds with common constants and a common onset.
Use the common range endpoint
$\lfloor\delta\kappa N_\gamma/K_{\log\gamma}(2N_\gamma)\rfloor$
and the common calibration bound
$2\exp\{-c\kappa N_\gamma/K_{\log\gamma}(2N_\gamma)^2\}$,
reducing $\delta,c$ to absorb fixed factors.
Both ratios required above diverge because
$\log\gamma=o(N_\gamma)$; the common survival bound is
$1-N_\gamma/\gamma\to1$. The same tail-sum calculation proves
the uniform assertion.
\hfill$\blacksquare$
\bigskip

\textbf{Proof of Theorem~\ref{th:restricted-PFA-minimax}.}
\textbf{Part 1.}
Let $\ell=\log L$ on $\{L>0\}$, a set of $P_1$-probability one.
Since
\[
\mathbb E_{P_1}(-\ell)^+=\int_{\{L<1\}}-L\log L\,dP_0\leq e^{-1}
\]
and $D<\infty$, $\ell\in L^1(P_1)$. Consequently for every
$\epsilon>0$,
\[
\eta_r(\epsilon):=P_1^{\otimes r}
\left\{\sum_{i=1}^r\ell(Y_i)>r(D+\epsilon)\right\}\to0.
\]
Fix $\zeta\in(0,1)$ and put
$r_N=\lfloor(1-\zeta)\log N/D\rfloor$,
$a_N=\lceil\kappa N\rceil$, and
$M_N=\lfloor(N-a_N)/r_N\rfloor+1$.
Then $r_N\to\infty$, $r_N=o(N)$, and $M_N\asymp N/r_N$.
Let $\nu_j=a_N+(j-1)r_N$, $1\leq j\leq M_N$.
For any PFA-feasible stopping time $\tau$ in either class,
the events $E_j=\{\nu_j<\tau\leq\nu_j+r_N\}$ are disjoint.
Thus $\sum_j\mathbb P_{P_0,\infty}(E_j)\leq\alpha$, and for
some $j$ its null probability is at most $\alpha/M_N$.
On the observation and randomizer sigma-field through
$\nu_j+r_N$, the likelihood ratio is
$\prod_{i=\nu_j+1}^{\nu_j+r_N}L(X_i)$, since the randomizers
have the same independent law. Splitting according to whether
its logarithm is at most $r_N(D+\epsilon)$ gives
\[
\mathbb P_{P_0,\nu_j,P_1}(E_j)
\leq e^{r_N(D+\epsilon)}\frac\alpha{M_N}+\eta_{r_N}(\epsilon).
\]
Choose $\epsilon=D\zeta/\{2(1-\zeta)\}$. The first term is
at most $C_{\alpha,\kappa}r_NN^{-\zeta/2}\to0$, and the
second tends to zero by the law of large numbers. The resulting
bound is independent of $\tau$. Since $1-\delta>0$, no feasible
procedure has a common correct-detection window at most $r_N$
for every $\nu\in I_N(\kappa)$ once $N$ is sufficiently large.
Taking the infimum over either class and then
$\zeta\downarrow0$ proves both lower bounds.

\textbf{Part 2.}
A Borel injection on a standard Borel space has a measurable
inverse on its image, so Proposition~\ref{prop:score-data-processing}
gives $D_{\mathrm{KL}}(P_1\|P_0)=D$. The learned alarm is measurable
in the observation and original-randomizer filtration and belongs
to $\mathfrak T_{\rm rand}$. Theorem~\ref{th:weighted-null-validity}
and $\sum_kw_k=1$ give PFA at most $\bar\alpha$.
Fix $c>1/D$. Choose one fixed $\theta$ and $\epsilon>0$ with
$c(\int h^\uparrow\log f_\theta-\epsilon)>1$, and let
$r_N=\lceil c\log N\rceil$. Uniformly for $T\in I_N(\kappa)$,
the PFA blocks have lengths comparable to $(\log N)^{1/2}$
and bin counts comparable to $(\log N)^{1/8}$. Their smallest
bin count diverges. The proof of
Lemma~\ref{lem:learned-score-growth} therefore gives common
constants $c_1,C_1>0$ such that
\[
\begin{split}
\sup_{T\in I_N(\kappa)}\mathbb P_{P_0,T,P_1}
\left\{\sum_{j=T+1}^{T+r_N}\log f_\theta(p_j)
<r_N\left(\int h^\uparrow\log f_\theta-\epsilon\right)\right\}
\leq C_1e^{-c_1r_N}+2e^{-c_1N/(\log N)^{1/4}}\to0.
\end{split}
\]
Also, uniformly over the same interval,
\[
\log\frac1{\bar\alpha w_{T+1}}
=\log N+2\log\log N+O(1).
\]
Eventually $\theta\leq T+1$ for every such $T$, and the strict
choice of $\theta,\epsilon$ implies
\[
r_N\left(\int h^\uparrow\log f_\theta-\epsilon\right)
\geq\log\frac1{\bar\alpha w_{T+1}}+\theta\log2.
\]
The local restart beginning at $T+1$ crosses the weighted boundary
outside the preceding event. This proves the uniform late-alarm
bound for both statistics.
Because the laws agree through $T$,
$\mathbb P_{P_0,T,P_1}(\tau\leq T)\leq\bar\alpha$.
It follows that
\[
\inf_{T\in I_N(\kappa)}\mathbb P_{P_0,T,P_1}(T<\tau\leq T+r_N)
\geq1-\bar\alpha-o(1)>1-\delta
\]
eventually. Hence the upper limit of the minimax window divided
by $\log N$ is at most $c$. Let $c\downarrow1/D$ and use Part 1.
\hfill$\blacksquare$
\bigskip

\textbf{Derivation of the classical ARL bounds}
Let $Y_1,Y_2,\ldots$ be iid $P_1$ and put
$\ell(Y):=\log(dP_1/dP_0)(Y)$. Since $D<\infty$, $\ell(Y)$ is integrable under $P_1$: indeed,
\[\mathbb E_{P_1}[(-\ell(Y))_+] = \int_{\{L<1\}}-L\log L\,dP_0\leq e^{-1}\], and finiteness of $D$ then controls the positive part. Hence
\begin{equation}
\label{eq:raw-oracle-LLN}
\frac1m\sum_{i=1}^m\ell(Y_i)
\longrightarrow D
\quad P_1^{\otimes\infty}\text{-a.s.}
\end{equation}
Fix $\beta\in(0,1)$ and $\epsilon>0$. There is $m_\epsilon$ such that
\[
\eta_m(\epsilon):=
\mathbb P_{P_1}\left\{
\sum_{i=1}^m\ell(Y_i)>m(D+\epsilon)
\right\}
\leq\beta/2,
\qquad m\geq m_\epsilon.
\]

For Lorden's criterion, fix $\tau\in\mathcal C_\gamma^{\mathrm{ARL}}$ and write
$J:=\mathcal J_L^{P_1}(\tau)$. If $J=\infty$, the claim is immediate. Set
\[
m:=\max\left\{m_\epsilon,\left\lceil\frac{J}{1-\beta}\right\rceil\right\}.
\]
For every deterministic $\nu$, Lorden's definition and Markov's inequality give, on $\{\tau>\nu\}$,
\[
\mathbb P_{P_0,\nu,P_1}
\{\nu<\tau\leq\nu+m\mid\mathcal A_\nu\}
\geq\beta.
\]
Let
$L_{\nu,m}:=\sum_{i=\nu+1}^{\nu+m}\ell(X_i)$. Intersecting with
$\{L_{\nu,m}\leq m(D+\epsilon)\}$ loses at most $\eta_m(\epsilon)$ under the alternative. Changing measure back to the no-change law gives, on $\{\tau>\nu\}$,
\[
\mathbb P_{P_0,\infty}
\{\nu<\tau\leq\nu+m\mid\mathcal A_\nu\}
\geq
q_m:=e^{-m(D+\epsilon)}\{\beta-\eta_m(\epsilon)\}.
\]
Applying this at $\nu=jm$ yields
$\mathbb P_{P_0,\infty}(\tau>jm)\leq(1-q_m)^j$ and hence
\[
\mathbb E_{P_0,\infty}\tau
\leq
\frac{m}{q_m}.
\]
Because $\gamma\leq\mathbb E_{P_0,\infty}\tau<\infty$,
\[
\log\gamma
\leq
m(D+\epsilon)+\log m-
\log\{\beta-\eta_m(\epsilon)\}.
\]
This first forces $J\to\infty$ along any sequence $\gamma\to\infty$, so eventually $m=\lceil J/(1-\beta)\rceil$. Consider any sequence of admissible stopping times indexed by $\gamma$. If $J/\log\gamma$ is unbounded along a subsequence, the desired lower bound is immediate there. Along every subsequence on which $J/\log\gamma$ is bounded, one has $m=O(\log\gamma)$ and therefore $\log m=o(\log\gamma)$. Rearranging the preceding display then gives
\[
J
\geq
(1-\beta)\frac{\log\gamma-o(\log\gamma)}{D+\epsilon}.
\]
Taking the infimum over $\tau$ and then letting $\beta\downarrow0$ and $\epsilon\downarrow0$ proves Lorden's bound.

For Pollak's criterion, let $J:=\mathcal J_P^{P_1}(\tau)$ and use the same definition of $m$. For every $\nu$ with
$\mathbb P_{P_0,\nu,P_1}(\tau>\nu)>0$, Markov's inequality gives
\[
\mathbb P_{P_0,\nu,P_1}
\{\nu<\tau\leq\nu+m\mid\tau>\nu\}
\geq\beta.
\]
The event $\{\tau>\nu\}$ is $\mathcal A_\nu$-measurable, so the same likelihood-ratio argument gives the same no-change block hazard $q_m$. The geometric-survival argument and the preceding algebra are unchanged, proving Pollak's bound. \hfill$\blacksquare$
\bigskip

\textbf{Proof of Theorem~\ref{thm:restricted-arl-lower}.}
\textbf{Part 1.}
As in the preceding likelihood-ratio argument,
$\ell=\log L\in L^1(P_1)$ because
$\int_{\{L<1\}}-L\log L\,dP_0\leq e^{-1}$ and $D<\infty$.
For every $\epsilon>0$,
\[
\eta_r(\epsilon):=P_1^{\otimes r}
\left\{\sum_{i=1}^r\ell(Y_i)>r(D+\epsilon)\right\}\to0.
\]
Fix $\zeta\in(0,1)$ and put
$r_\gamma=\lfloor(1-\zeta)\log\gamma/D\rfloor$ and
$\epsilon_\zeta=D\zeta/\{2(1-\zeta)\}$.
Then $r_\gamma\to\infty$, $r_\gamma=o(N_\gamma)$, and
$r_\gamma(D+\epsilon_\zeta)\leq(1-\zeta/2)\log\gamma$.
For a feasible $\tau$ in either class let
\[
a_\gamma=\lceil\kappa N_\gamma\rceil,\quad
M_\gamma=\lfloor(N_\gamma-a_\gamma)/r_\gamma\rfloor+1,\quad
\nu_j=a_\gamma+(j-1)r_\gamma.
\]
The events $E_j=\{\nu_j<\tau\leq\nu_j+r_\gamma\}$ are disjoint,
and all $\nu_j$ belong to $I_{N_\gamma}(\kappa)$. Thus
\[
\sum_{j=1}^{M_\gamma}\mathbb P_{P_0,\infty}(E_j)
\leq\mathbb P_{P_0,\infty}(\tau\leq N_\gamma+r_\gamma)
\leq(N_\gamma+r_\gamma)/\gamma.
\]
Since $M_\gamma\asymp N_\gamma/r_\gamma$, some $j$ has
$\mathbb P_{P_0,\infty}(E_j)\leq C_\kappa r_\gamma/\gamma$,
with $C_\kappa$ independent of the procedure.
The likelihood ratio through $\nu_j+r_\gamma$, including the
independent randomizers, is
$\prod_{i=\nu_j+1}^{\nu_j+r_\gamma}L(X_i)$.
Changing measure and splitting at its upper-typical logarithm
yields
\[
\mathbb P_{P_0,\nu_j,P_1}(E_j)
\leq C_\kappa r_\gamma\gamma^{-\zeta/2}
+\eta_{r_\gamma}(\epsilon_\zeta)=:\zeta_\gamma\to0.
\]
This upper bound is uniform over feasible procedures. Also,
\[
\mathbb P_{P_0,\nu_j,P_1}(\tau>\nu_j)
=\mathbb P_{P_0,\infty}(\tau>\nu_j)
\geq1-N_\gamma/\gamma.
\]
Consequently,
\[
\begin{split}
\mathcal J_\gamma^\kappa(\tau)
&\geq r_\gamma\mathbb P_{P_0,\nu_j,P_1}
(\tau>\nu_j+r_\gamma\mid\tau>\nu_j)\\
&\geq r_\gamma\left(1-\frac{\zeta_\gamma}{1-N_\gamma/\gamma}\right).
\end{split}
\]
Take the infimum over the class, divide by $\log\gamma$, and
use \eqref{eq:restricted-arl-regime}; then let $\zeta\downarrow0$.
This proves Part 1.

\textbf{Part 2.}
The alarm is a $\mathcal G_t$-stopping time and
$\mathcal G_t\subseteq\mathcal H_t$, so it belongs to
$\mathfrak T_{\rm rand}$. Adding mass $\rho$ to the constant
density and multiplying every previous mass by $1-\rho$
preserves the local martingale property. By
Theorem~\ref{th:weighted-null-validity},
\[
\mathbb P_{P_0,\infty}\{\Bar{\tau}_{S,w}(\gamma)\leq m\}
\leq m/\gamma,\qquad
\mathbb E_{P_0,\infty}\Bar{\tau}_{S,w}(\gamma)\geq\gamma.
\]
The pathwise inequality $\overline S_t^{(w,\rho)}\geq\rho t$
gives $\Bar{\tau}_{S,w}(\gamma)\leq\lceil\gamma/\rho\rceil$.
Thus it has finite null expectation and is in the required class.
The injective score retains $D$ by
Proposition~\ref{prop:score-data-processing}.
The uniform conditional-mean conclusion of
Corollary~\ref{cor:ARL-KL-sharp} now proves the upper bound.
Finally,
\[
\mathcal R_\gamma^{\rm lin}(\mathfrak T_{\rm rand};\kappa)
\leq\mathcal J_\gamma^\kappa\{\Bar{\tau}_{S,w}(\gamma)\},
\]
and Part 1 bounds the lower limit of the left side divided by
$\log\gamma$ by $1/D$. Combining the inequalities proves both limits.
\hfill$\blacksquare$
\bigskip

\textbf{Proof of Proposition~\ref{prop:score-data-processing}.}
Let $F=F_{s,P_0}$ and $Q(u)=\inf\{x:F(x)\geq u\}$ in this proof.
For $u\in(0,1)$, $F(Q(u)-)\leq u\leq F(Q(u))$.
In the randomized transform of a variable with CDF $F$,
values below $Q(u)$ contribute probability $F(Q(u)-)$ to the
event that the transform is at most $u$. When $F$ has a jump
at $Q(u)$, the values at that point contribute exactly
$u-F(Q(u)-)$, by uniform randomization over the jump interval.
When there is no jump, $F(Q(u)-)=F(Q(u))=u$ and this additional
contribution is zero. Equality cases arising from flat intervals
have probability zero. The transform is therefore uniform.
The generalized-inverse identity also gives
$Q(T_s(X,\lambda))=s(X)$ almost surely under $P_0\otimes U$.

Let $L_s=d(P_1\circ s^{-1})/d(P_0\circ s^{-1})$. Absolute
continuity follows from $P_1\ll P_0$, and
$L_s(s(X))=\mathbb E_{P_0}[L(X)\mid s(X)]$ almost surely.
For bounded measurable $a$, independence of $\lambda$ and
conditional expectation yield
\[
\begin{split}
\mathbb E_{P_1\otimes U}a(T_s(X,\lambda))
&=\mathbb E_{P_0\otimes U}
[L_s(s(X))a(T_s(X,\lambda))]\\
&=\int_0^1L_s(Q(u))a(u)\,du.
\end{split}
\]
Thus $h_s(u)=L_s(Q(u))$ is a density of $H_s$. The inverse
identity gives
\[
\int_0^1h_s\log h_s
=\int L_s\log L_s\,d(P_0\circ s^{-1}).
\]
Conditional Jensen for the strictly convex function $x\log x$
gives the upper bound by $\mathbb E_{P_0}L\log L$.
Its negative part is bounded by $e^{-1}$ and its positive part
is integrable. Equality in this conditional Jensen inequality
holds precisely when $L=\mathbb E_{P_0}[L\mid s]$ almost surely,
which is equivalent to the asserted measurability condition.
It holds for $s=L$. If $s$ has a measurable inverse, $X$ and
hence $L(X)$ can be recovered measurably from $s(X)$, proving
all remaining claims.
\hfill$\blacksquare$
\bigskip

\textbf{Proof of Theorem~\ref{thm:local-post-change-pvalues}:}
Write $T:=T_n$, $r:=r_n$, and $m:=T+r$. Let
\[
\widehat U_n(t):=\frac1T\sum_{j=1}^T\mathbf 1\{p_{n,j}\leq t\},
\qquad
\widehat P_n(t):=\frac1r\sum_{j=T+1}^{T+r}\mathbf 1\{p_{n,j}\leq t\}.
\]
The pre-change conformal ranks are iid uniform, so $\widehat U_n$ is the empirical CDF of $T$ iid $U(0,1)$ variables. For $j>T$, put $V_{n,j}:=F_{P_0}(X_{n,j})$ and let
\[
\widehat H_n(t):=\frac1r\sum_{j=T+1}^{T+r}\mathbf 1\{V_{n,j}\leq t\}.
\]
Then the $V_{n,j}$ are iid with CDF $H$.

The deterministic rank comparison used later in Lemma~\ref{lem:local-comparison} gives
\begin{equation}
\label{eq:local-proof-rank-bound}
\eta_n:=\max_{T<j\leq T+r}|p_{n,j}-V_{n,j}|
\leq
\|\widehat F_{P_0,T}-F_{P_0}\|_\infty+\frac{2r+1}{T},
\end{equation}
where $\widehat F_{P_0,T}$ is the empirical CDF of $X_{n,1},\ldots,X_{n,T}$. Hence $\eta_n\to0$ in probability. If
$\omega_H(\delta):=\sup\{|H(u)-H(v)|:|u-v|\leq\delta\}$, then monotonicity gives
\[
\|\widehat P_n-H\|_\infty
\leq
\|\widehat H_n-H\|_\infty+
\omega_H(\eta_n).
\]
By DKW and continuity of $H$, the right-hand side converges to zero in probability.

Since
\[
\widehat F_{n,m}
=
\frac{T}{m}\widehat U_n+\frac{r}{m}\widehat P_n,
\]
we have the exact identity
\[
\frac{m}{r}(\widehat F_{n,m}-F_{\mathrm U})
=
\frac{T}{r}(\widehat U_n-F_{\mathrm U})
+(\widehat P_n-H)
+(H-F_{\mathrm U}).
\]
Consequently,
\begin{align*}
\left|
\frac{m}{r}\|\widehat F_{n,m}-F_{\mathrm U}\|_\infty-
\Delta
\right|
&\leq
\frac{T}{r}\|\widehat U_n-F_{\mathrm U}\|_\infty
+
\|\widehat P_n-H\|_\infty.
\end{align*}
The first term is $O_{\mathbb P}(\sqrt T/r)=o_{\mathbb P}(1)$, proving \eqref{eq:conv to Delta}. Since $m/T\to1$ and $r/\sqrt T\to\infty$,
\[
\sqrt m\|\widehat F_{n,m}-F_{\mathrm U}\|_\infty
=
\frac{r}{\sqrt m}
\left\{\frac{m}{r}\|\widehat F_{n,m}-F_{\mathrm U}\|_\infty\right\}
\longrightarrow\infty
\]
in probability.

Now assume $\frac{r_n^2}{T_n\log n}\to\infty,$. Together with $r/T\to0$, it implies $\frac{T}{\log n}\to\infty,
\frac{r}{\log n}\to\infty.$
For every fixed $\epsilon>0$, DKW gives
\[
\mathbb P\left\{\frac{T}{r}\|\widehat U_n-F_{\mathrm U}\|_\infty>\epsilon\right\}
\leq
2\exp\left(-2\epsilon^2\frac{r^2}{T}\right),
\]
and the right-hand side is summable. The DKW bounds for $\widehat H_n$ and $\widehat F_{P_0,T}$ are also summable because $r/\log n\to\infty$ and $T/\log n\to\infty$. Borel--Cantelli, \eqref{eq:local-proof-rank-bound}, and uniform continuity of $H$ therefore make every preceding convergence almost sure. \hfill$\blacksquare$
\bigskip

\textbf{Proof of Theorem~\ref{thm:CMM-growth-corrected} under Assumption~\ref{assn:uniform-local-mixture-richness}:}
Put $m_n:=T_n+r_n$ and
\[
d_n:=\|\widehat F_{n,m_n}-F_{\mathrm U}\|_\infty.
\]
Theorem~\ref{thm:local-post-change-pvalues} and \eqref{eq:r_n} imply that, with probability tending to one,
\[
\frac{\Delta r_n}{2m_n}
\leq d_n\leq
\frac{2\Delta r_n}{m_n}.
\]
The same event occurs eventually almost surely under the additional condition \eqref{eq:CMM-growth-as-condition}. Since $r_n=o(T_n)$, the upper bound is eventually below $\epsilon_0$. On this event, Assumption~\ref{assn:uniform-local-mixture-richness} gives
\[
\log M_{n,m_n}
\geq
\kappa m_nd_n^2-K\log(1/d_n)-K
\geq
\frac{\kappa\Delta^2}{4}\frac{r_n^2}{m_n}
-K\log\frac{2m_n}{\Delta r_n}-K.
\]
Because $m_n\asymp T_n$ and $r_n^2/(T_n\log T_n)\to\infty$, the logarithmic terms are
$o(r_n^2/T_n)$. Hence there is $c_0>0$ such that
\[
\log M_{n,T_n+r_n}\geq c_0\frac{r_n^2}{T_n}
\]
with probability tending to one, and eventually almost surely under \eqref{eq:CMM-growth-as-condition}. \hfill$\blacksquare$
\bigskip

\textbf{Proof of Theorem~\ref{thm:CMM-growth-corrected} under Assumption~\ref{assn:directional-local-detectability}:}
Again put $m_n:=T_n+r_n$ and $a_n:=r_n/m_n$. Let
$s:=\operatorname{sign}(J)$ and choose $\gamma>0$ such that
\[
d_\gamma:=\gamma|J|-C\gamma^2>0.
\]
Set $\theta_n:=\gamma s a_n$. Then $|\theta_n|\leq\theta_0/2$ for all sufficiently large $n$, and
\[
\sum_{j=1}^{m_n}\log f_{\theta_n}(p_{n,j})
\geq
\theta_n\sum_{j=1}^{m_n}g(p_{n,j})
-Cm_n\theta_n^2.
\]
The pre-change ranks are iid uniform with mean-zero $g$, so
\[
\frac1{r_n}\sum_{j=1}^{T_n}g(p_{n,j})
\xrightarrow{\mathbb P}0.
\]
For the post-change block, the deterministic rank comparison, Lipschitz continuity of $g$, and the weak law give
\[
\frac1{r_n}\sum_{j=T_n+1}^{T_n+r_n}g(p_{n,j})
\xrightarrow{\mathbb P}
\int g\,dH=J.
\]
Thus
\[
\sum_{j=1}^{m_n}g(p_{n,j})=r_nJ+o_{\mathbb P}(r_n),
\]
and therefore
\[
\sum_{j=1}^{m_n}\log f_{\theta_n}(p_{n,j})
\geq
\{d_\gamma+o_{\mathbb P}(1)\}\frac{r_n^2}{m_n}.
\]
Let $B_+:=\max\{B,1\}$. For
$|\vartheta-\theta_n|\leq(2B_+m_n)^{-1}$, the derivative bound gives a total log-likelihood loss at most $1/2$. The local prior condition then yields
\[
\log M_{n,m_n}
\geq
\{d_\gamma+o_{\mathbb P}(1)\}\frac{r_n^2}{m_n}
-
\log m_n-O(1).
\]
The last two terms are negligible under \eqref{eq:r_n}, proving the in-probability assertion.

Under \eqref{eq:CMM-growth-as-condition}, Hoeffding's inequality makes the pre-change deviation probability summable at rate
$\exp\{-c r_n^2/T_n\}$. It also makes the post-change oracle-average deviation summable at rate $\exp(-cr_n)$, while DKW controls the rank approximation; as shown in the proof of Theorem~\ref{thm:local-post-change-pvalues}, the additional condition implies both $T_n/\log n\to\infty$ and $r_n/\log n\to\infty$. Borel--Cantelli therefore upgrades every $o_{\mathbb P}$ term above to an almost-sure $o(1)$ term. The same mixture bound then holds eventually almost surely. \hfill$\blacksquare$
\bigskip

\textbf{Proof of Theorem~\ref{th: upper bound on detection delay}:}
Let
\[
r_n:=\left\lceil C\sqrt{T_n\log b_n}\right\rceil,
\]
where $C>0$ will be chosen below. Since $b_n\to\infty$, we have $r_n\to\infty$. The assumption $\log b_n=o(T_n)$ gives
\[
\frac{r_n}{T_n}
\leq
C\sqrt{\frac{\log b_n}{T_n}}+o(1)
\to0.
\]
Moreover,
\[
\frac{r_n^2}{T_n\log T_n}
\geq
C^2\frac{\log b_n}{\log T_n}+o(1)
\to\infty
\]
because $\log T_n=o(\log b_n)$. Hence $(r_n)$ satisfies the conditions of Theorem~\ref{thm:CMM-growth-corrected}. There exists $c_0>0$ such that
\[
\mathbb P\left\{
\log M_{n,T_n+r_n}
\geq
c_0\frac{r_n^2}{T_n}
\right\}
\to1.
\]
Choose $C>c_0^{-1/2}$. Then
\[
c_0\frac{r_n^2}{T_n}
\geq
c_0C^2\log b_n+o(\log b_n)
\geq
\log b_n
\]
for all sufficiently large $n$. Thus
\[
\mathbb P\{M_{n,T_n+r_n}\geq b_n\}\to1.
\]
By definition of $\tau_n^{\texttt{CMM}}(b_n)$,
\[
\mathbb P\{\tau_n^{\texttt{CMM}}(b_n)\leq T_n+r_n\}
\to1.
\]
Before time $T_n$ the changepoint law agrees with the no-change law, so Ville's inequality gives
\[
\mathbb P_{P_0,T_n,P_1}\{\tau_n^{\texttt{CMM}}(b_n)\leq T_n\}
\leq b_n^{-1}\to0.
\]
Subtracting this pre-change probability and absorbing the ceiling into $C$ proves \eqref{eq:prob conv}. Finally,
\[
\frac{\sqrt{T_n\log b_n}}{T_n}
=
\sqrt{\frac{\log b_n}{T_n}}
\to0,
\]
which gives the relative-delay conclusion. \hfill$\blacksquare$
\bigskip

\textbf{Proof of Theorem~\ref{th:CMM-late-alarm-tail}:}
Write $m:=T_n+r, \hspace{.2cm}
a:=\frac{r}{m}.$
Let $g$ and $\{f_\theta:|\theta|\leq\theta_0\}$ be as in Assumption~\ref{assn:directional-local-detectability}, and write
\[
J:=\int_0^1g(u)\,dH(u)\neq0.
\]
Let $L_g$ be a Lipschitz constant of $g$. Enlarge $L_g$ and the derivative bound $B$ if necessary so that $L_g\geq1$ and $B\geq1$. Choose $\delta_0\in(0,1)$ sufficiently small that, for every $r\leq\delta_0T_n$ and all sufficiently large $n$,
\[
\frac{2r+1}{T_n}
\leq
\frac{|J|}{16L_g}.
\]
Let $s:=\operatorname{sign}(J),
\hspace{.2cm}
\theta:=\gamma s a,$
where $\gamma>0$ is fixed sufficiently small that
\[
\gamma\frac{r}{T_n+r}\leq\frac{\theta_0}{2}
\]
for $r\leq\delta_0T_n$ and
\[
c_*:=\frac{\gamma|J|}{2}-C\gamma^2>0.
\]
Assumption~\ref{assn:directional-local-detectability} gives
\[
\sum_{j=1}^m\log f_\theta(p_{n,j})
\geq
\theta\sum_{j=1}^m g(p_{n,j})-Cm\theta^2.
\]
Split the score sum at $T_n$. Since the pre-change conformal $p$-values are iid uniform and $\int g=0$, Hoeffding's inequality gives
\[
\mathbb P\left\{
\left|\sum_{j=1}^{T_n}g(p_{n,j})\right|>\frac{|J|r}{8}
\right\}
\leq
2\exp\left(-c_1\frac{r^2}{T_n}\right).
\]
For $j>T_n$, put $V_{n,j}=F_{P_0}(X_{n,j})$. The variables $g(V_{n,j})$ are iid, bounded, and have mean $J$, so
\[
\mathbb P\left\{
\left|\sum_{j=T_n+1}^m g(V_{n,j})-rJ\right|>\frac{|J|r}{8}
\right\}
\leq
2e^{-c_2r}.
\]
By Lemma~\ref{lem:local-comparison},
\[
\eta_{n,r}
:=
\max_{T_n<j\leq T_n+r}|p_{n,j}-V_{n,j}|
\leq
\|\widehat F_{P_0,T_n}-F_{P_0}\|_\infty+\frac{2r+1}{T_n}.
\]
The DKW inequality gives
\[
\mathbb P\left\{
\|\widehat F_{P_0,T_n}-F_{P_0}\|_\infty>\frac{|J|}{16L_g}
\right\}
\leq
2e^{-c_3T_n}.
\]
Outside the union of these events, $s\sum_{j=1}^m g(p_{n,j})
\geq
\frac{|J|r}{2}.$
Since $r\leq\delta_0T_n$, the terms $e^{-c_2r}$ and $e^{-c_3T_n}$ are bounded by a constant multiple of $\exp(-c r^2/T_n)$. Hence there exist constants $C_1,C_2>0$ such that
\[
\mathbb P\left\{
s\sum_{j=1}^m g(p_{n,j})<\frac{|J|r}{2}
\right\}
\leq
C_1\exp\left(-C_2\frac{r^2}{T_n}\right).
\]
On the complementary event,
\[
\theta\sum_{j=1}^m g(p_{n,j})-Cm\theta^2
\geq
c_*\frac{r^2}{m}.
\]
For every $\vartheta$ satisfying $|\vartheta-\theta|\leq1/(2Bm)$, the derivative bound gives
\[
\sum_{j=1}^m\log f_\vartheta(p_{n,j})
\geq
\sum_{j=1}^m\log f_\theta(p_{n,j})-\frac12.
\]
The local prior mass condition therefore yields
\[
\log M_{n,m}
\geq
c_*\frac{r^2}{m}-\log m-O(1)
\]
outside an event of probability at most $C_1\exp(-C_2r^2/T_n)$. Since $m\leq(1+\delta_0)T_n$, a sufficiently large $C_0$ makes
\[
r\geq C_0\sqrt{T_n\{\log T_n+\log b_n\}}
\]
imply
\[
c_*\frac{r^2}{m}-\log m-O(1)
\geq
\log b_n.
\]
Thus
\[
\mathbb P\{M_{n,T_n+r}<b_n\}
\leq
C_1\exp\left(-C_2\frac{r^2}{T_n}\right).
\]
The event $\{\tau_n^{\texttt{CMM}}(b_n)>T_n+r\}$ implies $M_{n,T_n+r}<b_n$, proving the theorem. \hfill$\blacksquare$
\bigskip

\textbf{Proof of Corollary~\ref{cor:CMM-fixed-threshold}:}
Theorem~\ref{thm:CMM-growth-corrected} gives
\[
\mathbb P\left\{
\log M_{n,T_n+r_n}\geq c_0r_n^2/T_n
\right\}\to1.
\]
Since \eqref{eq:r_n} implies $r_n^2/T_n\to\infty$, the right-hand side inside the event eventually exceeds the fixed value $\log b$. Hence
\[
\mathbb P\{\tau_n^{\texttt{CMM}}(b)\leq T_n+r_n\}\to1.
\]
Moreover,
\[
\mathbb P_{P_0,T_n,P_1}\{\tau_n^{\texttt{CMM}}(b)\leq T_n\}
\leq b^{-1},
\]
so subtracting the pre-change event gives the stated lower bound for a correct post-change alarm. Under \eqref{eq:CMM-growth-as-condition}, the growth inequality holds eventually almost surely, which gives the asserted almost-sure crossing-time conclusion.

If $T_n=\Theta(n)$ and $r_n=a_n\sqrt{n\log n}$ with
$a_n\to\infty$ and $a_n=o\{\sqrt{n/\log n}\}$, then $r_n=o(T_n)$ and both \eqref{eq:r_n} and \eqref{eq:CMM-growth-as-condition} hold. \hfill$\blacksquare$
\bigskip

\textbf{Proof of Lemma~\ref{lem:late-alarm-negligible}:}
Let $C_0,C_1,C_2,\delta_0$ be the constants from Theorem~\ref{th:CMM-late-alarm-tail}. Choose $A>C_0$ such that $C_2A^2>2$ and define
\[
d_n:=\left\lceil A\sqrt{T_n\log(T_nU_n)}\right\rceil.
\]
Condition~\eqref{eq:threshold size} gives
\[
\frac{d_n}{T_n^{1-\gamma}}
\leq
A\sqrt{\frac{\log(T_nU_n)}{T_n^{1-2\gamma}}}
+
\frac{1}{T_n^{1-\gamma}}
\longrightarrow0.
\]
Hence $d_n=o(T_n)$, $d_n\to\infty$, and $d_n\leq\delta_0T_n$ for all sufficiently large $n$. Since $b>1$ is fixed, eventually
\[
d_n
\geq
C_0\sqrt{T_n\{\log T_n+\log b\}}.
\]
Theorem~\ref{th:CMM-late-alarm-tail} therefore yields
\[
\mathbb P\{\tau_n^{\texttt{CMM}}(b)>T_n+d_n\}
\leq
C_1\exp\left(-C_2\frac{d_n^2}{T_n}\right)
\leq
C_1(T_nU_n)^{-C_2A^2}.
\]
Multiplying by $U_n$ and using $U_n\leq T_nU_n$ gives
\[
U_n\,\mathbb P\{\tau_n^{\texttt{CMM}}(b)>T_n+d_n\}
\leq
C_1(T_nU_n)^{1-C_2A^2}
\longrightarrow0.
\]
\hfill$\blacksquare$
\bigskip

\textbf{Proof of Proposition~\ref{ex:threshold-betting-family}:} 
First, observe that each $f_{a,\theta}$ is a probability density on $[0,1]$.
Indeed,
\[
\int_0^1 g_a(u)\,du
=
\int_0^1\{\mathbf 1(u\leq a)-a\}\,du
=
a-a
=
0,
\]
and hence
\[
\int_0^1 f_{a,\theta}(u)\,du
=
1+\theta\int_0^1g_a(u)\,du
=
1.
\]
Also, since $g_a(u)\in[-1,1]$ and $|\theta|\leq1/2$, we have $f_{a,\theta}(u)\geq\frac12.$
Thus $f_{a,\theta}$ is a valid betting density.
Let $G_m$ be any empirical CDF on $[0,1]$ and write
\[
\delta:=\|G_m-F_{\mathrm U}\|_\infty.
\]
If $\delta=0$, there is nothing to prove. Assume $\delta>0$. Define
\[
D(a):=G_m(a)-a .
\]
Choose $a_0\in[0,1]$ such that $|D(a_0)|\geq \frac34\delta .$
Let $s:=\operatorname{sign}(D(a_0))$. Suppose first that $s=1$. Then $D(a_0)\geq3\delta/4$. Since
$G_m(a_0)\leq1$, we must have $a_0\leq1-3\delta/4$. Thus
\[
I_a:=[a_0,a_0+\delta/4]\subseteq[0,1].
\]
For every $a\in I_a$, since $G_m(a)\geq G_m(a_0)$,
\[
D(a)=G_m(a)-a
\geq
G_m(a_0)-a
=
D(a_0)-(a-a_0)
\geq
\frac34\delta-\frac14\delta
=
\frac12\delta .
\]
If $s=-1$, then $G_m(a_0)\geq0$ and $D(a_0)\leq-3\delta/4$ imply $a_0\geq3\delta/4$. Hence the same argument with
\(
I_a:=[a_0-\delta/4,a_0]
\)
gives
\(
D(a)\leq -\frac12\delta
\text{ for every }a\in I_a.
\)
In both cases, $|I_a|=\delta/4$ and
\(
sD(a)\geq \frac12\delta
\text{ for every }a\in I_a.
\)
Now define
\[
I_\theta
:=
\left\{\theta:s\theta\in[\delta/16,\delta/8]\right\}.
\]
Then $|I_\theta|=\delta/16$ and, for every $a\in I_a$ and
$\theta\in I_\theta$,
\[
\theta D(a)\geq \frac{\delta^2}{32},
\qquad
\theta^2\leq \frac{\delta^2}{64}.
\]
Since $|\theta g_a(u)|\leq1/2$, the inequality
\(
\log(1+x)\geq x-x^2, \text{ for } |x|\leq1/2,
\)
implies
\[
\begin{split}
\int_0^1\log f_{a,\theta}(u)\,dG_m(u)
&=
\int_0^1\log(1+\theta g_a(u))\,dG_m(u)\\
&\geq
\theta\int_0^1g_a(u)\,dG_m(u)
-
\theta^2\int_0^1g_a(u)^2\,dG_m(u)\\
&\geq
\theta D(a)-\theta^2 \geq
\frac{\delta^2}{64}.
\end{split}
\]
Therefore
\[
\begin{split}
&\int_{\mathcal F_{\mathrm{thr}}}
\exp\left\{
m\int_0^1\log f(u)\,dG_m(u)
\right\}
d\Pi_{\mathrm{thr}}(f)\\
&\qquad\geq
\int_{I_a}\int_{I_\theta}
\exp\left\{
m\int_0^1\log f_{a,\theta}(u)\,dG_m(u)
\right\}
d\theta\,da\\
&\qquad\geq
|I_a||I_\theta|
\exp\left\{\frac{m\delta^2}{64}\right\}\\
&\qquad=
\frac{\delta^2}{64}
\exp\left\{\frac{m\delta^2}{64}\right\}.
\end{split}
\]
Taking logarithms gives
\[
\log\int_{\mathcal F_{\mathrm{thr}}}
\exp\left\{
m\int_0^1\log f(u)\,dG_m(u)
\right\}
d\Pi_{\mathrm{thr}}(f)
\geq
\frac{m\delta^2}{64}
-
2\log\frac1\delta
-
\log 64.
\]
Thus Assumption~\ref{assn:uniform-local-mixture-richness} holds with, for example, $\kappa=\frac1{64},
\hspace{.2cm}
K=2+\log64,
\hspace{.2cm}
\epsilon_0=1.$
This proves the proposition. \hfill $\blacksquare$
\bigskip \\
\textbf{Proof of Proposition~\ref{ex:countable-exp-tilt-family}:}
Let $\nu:=H-U$ be the finite signed Borel measure induced by the difference of the two probability laws on $[0,1]$. Since $H\neq F_{\mathrm U}$, $\nu\neq0$. Finite signed Borel measures on the compact metric space $[0,1]$ are regular, and continuous functions separate them; hence there is $\varphi\in C[0,1]$ such that $\int\varphi\,d\nu\neq0$. Set
\[
\varphi_0:=\varphi-\int_0^1\varphi(u)\,du,
\qquad
c:=\max\{1,\|\varphi_0\|_\infty\},
\qquad
g:=\varphi_0/c.
\]
Then $g\in\mathcal C_{0,1}$ and $\int g\,d\nu\neq0$. Choose $\ell$ so that $\|g_\ell-g\|_\infty
<
\frac14\left|\int g\,d\nu\right|.$
Since $|\nu|([0,1])\leq2$,
\[
\left|\int(g_\ell-g)\,d\nu\right|
\leq
\|g_\ell-g\|_\infty |\nu|([0,1])
<
\frac12\left|\int g\,d\nu\right|.
\]
Thus
\[
J_\ell:=\int_0^1g_\ell(u)\,dH(u)
=
\int g_\ell\,d\nu
\neq0,
\]
where the equality uses $\int_0^1g_\ell(u)\,du=0$.

Fix this $\ell$ and consider the path $f_\theta:=f_{\ell,\theta}$. We have $f_0\equiv1$. Moreover,
\[
\psi_\ell(0)=\psi_\ell'(0)=0,
\qquad
\psi_\ell''(\theta)
=
\operatorname{Var}_{f_{\ell,\theta}}\{g_\ell(U)\}
\leq1.
\]
Taylor's theorem therefore gives
\[
\sup_{u\in[0,1]}
|\log f_\theta(u)-\theta g_\ell(u)|
=|\psi_\ell(\theta)|
\leq\frac12\theta^2.
\]
Also
\[
\left|\frac{\partial}{\partial\theta}\log f_\theta(u)\right|
=|g_\ell(u)-\psi_\ell'(\theta)|
\leq2.
\]
Finally, for $\theta\in[-\theta_0/2,\theta_0/2]$ and all sufficiently small $\epsilon>0$, the prior mass of the $\epsilon$-neighborhood on this path is at least
$w_\ell\frac{2\epsilon}{2\theta_0}
=
\frac{w_\ell}{\theta_0}\epsilon.$
Thus Assumption~\ref{assn:directional-local-detectability} holds with direction $g_\ell$, $C=1/2$, $B=2$, and $c_\Pi=w_\ell/\theta_0$. \hfill$\blacksquare$
\bigskip

\textbf{Proof of Theorem~\ref{th: impossibility theorem}:}
For $n\leq T$, the history law under $\mathbb P_{P_0,T,P_1}$ is $P_0^{\otimes(n-1)}$ and the conditional law of $X_n$ is $P_0$. The asserted transition inequality therefore follows directly from the supermartingale property under $P_0^{\otimes\infty}$.

Fix $n>T$ and write
\[
\nu_{n-1}:=P_0^{\otimes T}\otimes P_1^{\otimes(n-1-T)}
\]
for the law of $(X_1,\ldots,X_{n-1})$ under $H_1$. Let $\Lambda:=\mathbb Q\cap(0,1)$ and, for $\lambda\in\Lambda$, put
\[
P_\lambda:=\lambda P_0+(1-\lambda)P_1\in\mathcal M.
\]
The supermartingale property under $P_\lambda^{\otimes\infty}$ implies that, outside a $P_\lambda^{\otimes(n-1)}$-null set,
\begin{equation}
\label{eq:mixture-transition-proof}
\int_{\Omega}m_n(x_1,\ldots,x_{n-1},x)\,dP_\lambda(x)
\leq
m_{n-1}(x_1,\ldots,x_{n-1}).
\end{equation}
Because $P_\lambda\geq\lambda P_0$ and $P_\lambda\geq(1-\lambda)P_1$, one has
\(
\nu_{n-1}\ll P_\lambda^{\otimes(n-1)}.
\)
Thus the exceptional set in \eqref{eq:mixture-transition-proof} is also $\nu_{n-1}$-null. Since $\Lambda$ is countable, there is a single $\nu_{n-1}$-null set outside which \eqref{eq:mixture-transition-proof} holds for every $\lambda\in\Lambda$. On this common full-measure set, nonnegativity of $m_n$ gives
\[
(1-\lambda)
\int_{\Omega}m_n(x_1,\ldots,x_{n-1},x)\,dP_1(x)
\leq
m_{n-1}(x_1,\ldots,x_{n-1}).
\]
Letting rational $\lambda\downarrow0$ yields
\[
\int_{\Omega}m_n(x_1,\ldots,x_{n-1},x)\,dP_1(x)
\leq
m_{n-1}(x_1,\ldots,x_{n-1})
\quad\nu_{n-1}\text{-a.e.}
\]
This is exactly
$\mathbb E_{P_0,T,P_1}(M_n\mid\mathcal A_{n-1})\leq M_{n-1}$ for $n>T$. Since $P_0$, $P_1$, and $T$ were arbitrary, $(M_n,\mathcal A_n)$ is a test supermartingale under every distribution in $H_1$. \hfill$\blacksquare$
\bigskip

\textbf{Proof of Proposition~\ref{prop: convergence of conformal product martingale}:}
For $k\geq1$, let $q_k$ be the randomized conformal rank of $Z_{T+k}$ within the post-change block $Z_{T+1},\ldots,Z_{T+k}$, using the same randomizer $\lambda_{T+k}$ as $p_{T+k}$. Proposition~\ref{prop: pval iid}, applied to the iid post-change sequence, gives
$q_1,q_2,\ldots\stackrel{\mathrm{iid}}{\sim}U(0,1).$
Writing the numerator of $p_{T+k}$ as the contribution from the $T$ pre-change observations plus $kq_k$, one obtains the deterministic bound $|p_{T+k}-q_k|
\leq
\frac{T}{T+k}.$
Because $f$ is continuous and strictly positive on the compact interval $[0,1]$, $\log f$ is bounded and uniformly continuous. A standard split into $k\leq K$ and $k>K$, followed by $K\to\infty$, therefore gives
\[
\frac1n\sum_{k=1}^{n-T}
\left\{\log f(p_{T+k})-\log f(q_k)\right\}
\longrightarrow0
\quad\text{almost surely}.
\]
The first $T$ terms contribute $o(1)$ after division by $n$, and the strong law applied to $(q_k)$ yields
\[
\frac1n\log M_n
\longrightarrow
\int_0^1\log f(u)\,du
\quad\text{almost surely}.
\]
Finally, strict Jensen inequality gives
\[
\int_0^1\log f(u)\,du
<
\log\int_0^1f(u)\,du
=0,
\]
because $f$ is not identically one. Hence the result follows. \hfill$\blacksquare$
\bigskip

\textbf{Proof of Proposition~\ref{prop:weighted-second-order}.}
Apply Lemma~\ref{lem:learned-score-growth} with
$\psi=\log f_\theta$, using $B=\log T_n$ in the PFA regime
or $B=\log b_n$ in the ARL regime. Its schedule conditions
were verified in that lemma. Eventually $\theta\leq T_n+1$.
Outside its lower-tail event, the finite local mixture gives
\[
\begin{split}
\log\{w_{T_n+1}M_{n,T_n+1,T_n+r}\}
&\geq\log w_{T_n+1}
+r\left(\int h^\uparrow\log f_\theta-\epsilon\right)-\theta\log2\\
&\geq\log b_n.
\end{split}
\]
This term occurs in both the sum and the maximum, so either
late-alarm event is contained in the lemma's lower-tail event.
For the regularized sum, retain only
$(1-\rho)w_{T_n+1}M_{n,T_n+1,T_n+r}$; this increases the
required logarithmic boundary by $\log\{1/(1-\rho)\}$.
It does not change the stream or its learning schedule.
\hfill$\blacksquare$
\bigskip

\textbf{Proof of Lemma~\ref{lem:local-comparison}:}
Fix $j=T_n+\ell$ with $1\leq\ell\leq r$. Continuity of $F_{P_0}$ makes $P_0$ atomless, while continuity of the pushforward CDF $H$ makes $P_1$ atomless on every level set of $F_{P_0}$ and hence atomless. Cross-ties also have probability zero because $P_0$ is atomless. Thus, almost surely, there are no ties between $X_{n,j}$ and the other observations used in its rank. Hence $p_{n,j}
=
\frac1j\sum_{i=1}^j\mathbf 1\{X_{n,i}<X_{n,j}\}
+
\frac{\lambda_{n,j}}j.$
With $V_{n,j}:=F_{P_0}(X_{n,j})$ and $\widehat F_{P_0,T_n}(x):=\frac1{T_n}\sum_{i=1}^{T_n}\mathbf 1\{X_{n,i}\leq x\},$
splitting the rank into pre- and post-change parts gives
\begin{align*}
|p_{n,j}-V_{n,j}|
&\leq
\|\widehat F_{P_0,T_n}-F_{P_0}\|_\infty
+
\left|\frac{T_n}{j}-1\right|
+
\frac{\ell-1}{j}
+
\frac1j \\
&\leq
\|\widehat F_{P_0,T_n}-F_{P_0}\|_\infty
+
\frac{2r+1}{T_n}.
\end{align*}
This proves \eqref{eq:local-comparison}. Now take $r=r_n$. DKW and $r_n=o(T_n)$ imply convergence to zero in probability. For any bounded $L_\varphi$-Lipschitz function $\varphi$, \[\left|
\frac1{r_n}\sum_{j=T_n+1}^{T_n+r_n}
\{\varphi(p_{n,j})-\varphi(V_{n,j})\}
\right|
\leq
L_\varphi
\max_{T_n<j\leq T_n+r_n}|p_{n,j}-V_{n,j}|.\]
The oracle variables are iid with CDF $H$, so the weak law yields the asserted convergence in probability.

If $T_n/\log n\to\infty$ and $r_n/\log n\to\infty$, the DKW bound for $\widehat F_{P_0,T_n}$ and Hoeffding's bound for the oracle average are summable in $n$. Borel--Cantelli and the displayed deterministic comparison give the almost-sure assertions. \hfill$\blacksquare$
\bigskip

\textbf{Proof of Lemma~\ref{lem:stable-oracle-ordering}.}
Each bin is translated onto a different interval of the same
length; the inverse translates it back. This proves invertibility
and measure preservation. For nonnegative integrable $f,g$,
Tonelli and the pointwise layer-cake identity give
\[
\|f-g\|_1=\int_0^\infty U(\{f>z\}\mathbin{\triangle}\{g>z\})\,dz.
\]
The superlevel sets of increasing rearrangements are upper
intervals, with the same lengths as the original superlevel sets.
It follows that
\[
\|f^\uparrow-g^\uparrow\|_1
=\int_0^\infty|U\{f>z\}-U\{g>z\}|\,dz\leq\|f-g\|_1.
\]
Since $\widehat h\circ T_\sigma^{-1}=\widehat h^\uparrow$
almost everywhere, the triangle inequality gives
\[
\begin{split}
\|h\circ T_\sigma^{-1}-h^\uparrow\|_1
&\leq\|(h-\widehat h)\circ T_\sigma^{-1}\|_1
+\|\widehat h^\uparrow-h^\uparrow\|_1\\
&\leq2\|\widehat h-h\|_1.
\end{split}
\]
Changing variables and integrating against $\psi$ gives the stated
bounded-function inequality. Finally $h(U)$ and $h^\uparrow(U)$
have the same law for a uniform variable $U$. Applying this
identity to $z$ shows that the rearrangement integrates to one;
applying it to $z\log z$, with bounded negative part and integrable
positive part, proves equality of the entropies.
\hfill$\blacksquare$
\bigskip

\textbf{Proof of Lemma~\ref{lem:learned-score-growth}.}
If $\psi$ is constant the event is empty. Suppose otherwise and
write $C_\psi=\|\psi\|_\infty>0$ and let $L_\psi$ be a
Lipschitz constant. All quantities introduced below are local to
this proof. The definition of $K_B$ gives
\[
K_B(T/3)\to\infty,\quad
K_B(2T)\log(eK_B(2T))=o(B^{1/2}),\quad
B^{1/2}K_B(2T)=o(T).
\]
For example, $K_B(2T)\leq B^{1/8}$ and, eventually,
$K_B(2T)\leq(2T/B)^{1/8}$. The first convergence follows because
both $B$ and $T/B$ diverge.

\paragraph{Comparison with independent population transforms.}
Put $V_t=T_s(X_t,\lambda_t)$, using \eqref{eq:score-transform}.
The post-change variables $V_t$, $t>T$, are iid with density $h$.
Let
\[
\widehat F_T(x)=T^{-1}\sum_{i=1}^T\mathbf1\{s(X_i)\leq x\}.
\]
In the numerator of the original randomized rank at time $t>T$,
the contribution from the first $T$ observations equals
\[
T\left[(1-\lambda_t)\widehat F_T(s(X_t)-)
+\lambda_t\widehat F_T(s(X_t))\right].
\]
The contribution from observations $T+1,\ldots,t$ lies between
zero and $t-T$. Comparing the displayed convex combination with
its population counterpart gives the deterministic inequality
\[
|p_t(s(X_1),\ldots,s(X_t);\lambda_t)-V_t|
\leq\|\widehat F_T-F_{s,P_0}\|_\infty+(t-T)/t.
\]
The supremum norm controls left limits as well by passage to
limits from below. Fix $0<d<1/32$, to be specified later, and let
\[
\mathcal E_T=\{\|\widehat F_T-F_{s,P_0}\|_\infty
\leq d/K_B(2T)\}.
\]
The Dvoretzky--Kiefer--Wolfowitz inequality gives
\[
\mathbb P(\mathcal E_T^c)\leq
2\exp\{-2d^2T/K_B(2T)^2\}.
\]
Choose $\delta\leq d$. On $\mathcal E_T$, for all
$T<t\leq T+\delta T/K_B(2T)$, the original rank differs from
$V_t$ by at most $2d/K_B(2T)$. The midpoint of its rank interval
differs from that rank by at most $1/(2t)\leq d/K_B(2T)$
for all sufficiently large $B,T$.

\paragraph{Bin boundaries and approximation.}
Choose $\eta>0$ small enough that
$(1+22C_\psi)\eta\leq\epsilon/8$. Since $h$ is integrable, choose
$M>0$ so that $\int_{\{h>M\}}h\leq \eta$, and then choose
$d\in(0,1/32)$ so that $32Md\leq \eta$.
For any $K\leq2K_B(2T)$ define the deterministic set
\[
E_{K,T}:=\{u\in[0,1]:
\min_{0\leq v\leq K}|u-v/K|\leq4d/K_B(2T)\}.
\]
Its measure is at most $32d$ for all sufficiently large $B,T$,
so $\int_{E_{K,T}}h\leq32Md+\eta\leq2\eta$.
On $\mathcal E_T$, when $V_t\notin E_{K,T}$, the original
rank, its interval midpoint, and $V_t$ are in the same $K$-bin.

Let $h_K$ be the bin-average density:
\[
h_K(u)=K\sum_{v=1}^K\left(\int_{I_{K,v}}h(x)\,dx\right)
\mathbf1\{u\in I_{K,v}\}.
\]
For every integer sequence $K\to\infty$, $\|h_K-h\|_1\to0$,
uniformly over all $K$ above a diverging lower bound. To see this,
choose a continuous function $g$ with arbitrarily small
$\|g-h\|_1$. Bin averaging is an $L^1$ contraction, and thus
$\|h_K-h\|_1\leq2\|g-h\|_1+\|g_K-g\|_1$.
The last term tends to zero uniformly by uniform continuity of
$g$. In particular, for all relevant $K\geq K_B(T/3)/2$,
$\|h_K-h\|_1\leq \eta$ eventually.

\paragraph{Independent events on complete post-change blocks.}
Fix an admissible $r$. For each complete block
$\{a_j,\ldots,a_{j+1}-1\}\subseteq\{T+1,\ldots,T+r\}$,
require the following three conditions:
\[
\left\|
\frac{K_{j+1}}{a_{j+1}-a_j}
\sum_{v=1}^{K_{j+1}}\sum_{t=a_j}^{a_{j+1}-1}
\mathbf1\{V_t\in I_{K_{j+1},v}\}\mathbf1_{I_{K_{j+1},v}}
-h_{K_{j+1}}\right\|_1\leq \eta;
\]
\[
\frac1{a_{j+1}-a_j}\sum_{t=a_j}^{a_{j+1}-1}
\mathbf1\{V_t\in E_{K,T}\}\leq3\eta,
\qquad K\in\{K_j,K_{j+1}\};
\]
and
\[
\max_{\sigma\in\mathfrak S_{K_j}}
\left|\frac1{a_{j+1}-a_j}\sum_{t=a_j}^{a_{j+1}-1}
\psi(T_\sigma(V_t))-
\int_0^1h(u)\psi(T_\sigma(u))\,du\right|\leq \eta.
\]
Here $\mathfrak S_K$ is the set of permutations of $[K]$, and
$T_\sigma$ is the map in Lemma~\ref{lem:stable-oracle-ordering}.
Call a complete block good when all three conditions hold.
These events depend only on the independent latent variables
within their own blocks and on deterministic partition sizes.
They are therefore independent across complete blocks, without
conditioning on $\mathcal E_T$.
All bin counts used in these events obey the bounds in the
statement: the considered interval and the next block start lie
in $[T/2,2T]$ eventually, since $\delta<1/32$ and
$B^{1/2}=o(T)$.

For a histogram based on $a_{j+1}-a_j$ iid observations, the expected
$L^1$ error from its bin-average density is at most
$\{K_{j+1}/(a_{j+1}-a_j)\}^{1/2}$. This follows by bounding
absolute bin-frequency errors by their standard deviations and
applying Cauchy--Schwarz. Changing one observation changes that
error by at most $2/(a_{j+1}-a_j)$. Its expectation is eventually
at most $\eta/2$, so the bounded-differences inequality bounds the
failure probability of the first condition by
$\exp\{-(a_{j+1}-a_j)\eta^2/8\}$.
For either boundary condition, Hoeffding's inequality gives
$\exp\{-2(a_{j+1}-a_j)\eta^2\}$, since its population probability
is at most $2\eta$. For a fixed permutation, the summands in the
last condition lie in $[-C_\psi,C_\psi]$. Its total failure
probability is at most
\[
2K_j!\exp\{-(a_{j+1}-a_j)\eta^2/(2C_\psi^2)\}.
\]
Since $K_j\log K_j=o(B^{1/2})$ uniformly and block lengths
are at least $B^{1/2}/2$, the sum of these bounds is at most
$\exp(-c_0B^{1/2})$ for a fixed $c_0>0$ and all sufficiently
large $B,T$.

\paragraph{Good training and scoring blocks.}
Suppose a complete post-change block immediately preceding block
$j$ is good. Let $\widehat h_j^V$ denote
\eqref{eq:learned-histogram} with each original rank replaced by
$V_t$ on that training block. On $\mathcal E_T$, bin assignments
can change only at its boundary observations. Their proportion
is at most $3\eta$. Changing one assignment changes histogram
$L^1$ distance by at most twice its empirical mass. Therefore
\[
\|\widehat h_j-\widehat h_j^V\|_1\leq6\eta,\qquad
\|\widehat h_j-h\|_1\leq6\eta+\eta+\eta=8\eta.
\]
Let $\sigma$ order the bins of $\widehat h_j$. Each $K_j$-bin
contains $t/K_j$ rank-grid midpoints with counting error at most
one. Counting the grid midpoints in preceding image bins shows
that, whenever $z$ and the original rank midpoint lie in the
same bin,
\[
|p_t-T_\sigma(z)|\leq
|p_t(s(X_1),\ldots,s(X_t);\lambda_t)-z|+(K_j+1)/t.
\]
Indeed, the error in the number of midpoints in preceding image
bins is at most $K_j-1$, and the original-bin endpoint contributes
at most one further grid spacing. Division by $t$ gives the bound.
Thus, on $\mathcal E_T$, outside $E_{K_j,T}$, the difference
between $\psi(p_t)$ and $\psi(T_\sigma(V_t))$ is at most
\[
L_\psi\{2d/K_B(2T)+(2K_B(2T)+1)/T\}.
\]
On the boundary set it is at most $2C_\psi$.
If the scoring block is also good, its empirical condition holds
for this actual random permutation because it holds for every
permutation. Lemma~\ref{lem:stable-oracle-ordering} and the
histogram bound therefore give
\[
\begin{split}
\frac1{a_{j+1}-a_j}\sum_{t=a_j}^{a_{j+1}-1}\psi(p_t)
\geq {}&\int_0^1h^\uparrow\psi
-(1+22C_\psi)\eta\\
&-L_\psi\{2d/K_B(2T)+(2K_B(2T)+1)/T\}.
\end{split}
\]
The last term tends to zero. By the choice of $\eta$ this lower
bound is at least $\int h^\uparrow\psi-\epsilon/4$ eventually.

\paragraph{Summing complete blocks.}
Let $N$ be the number of complete post-change blocks in the
interval and $F$ the number of these whose good events fail.
Two possible incomplete endpoint blocks and the first complete
block, whose training data need not be fully post-change, contain
at most $6B^{1/2}$ observations. Each failed block can invalidate
at most itself and its successor, at most $4B^{1/2}$ observations.
Every other observation belongs to a block with the preceding
lower bound, while every summand is at least $-C_\psi$.
Consequently, on $\mathcal E_T$,
\[
\sum_{t=T+1}^{T+r}\psi(p_t)
\geq r\left(\int h^\uparrow\psi-\epsilon/4\right)
-2C_\psi B^{1/2}(6+4F).
\]
Take $a=\min\{1/4,\epsilon/(64C_\psi)\}$ and
$A\geq\max\{8,48C_\psi/\epsilon\}$.
For $r\geq AB^{1/2}$ and $F\leq aN$, the endpoint loss is at
most $\epsilon r/4$. Since $N\leq2r/B^{1/2}$, the failed-block
loss is also at most $\epsilon r/4$. Hence the sum is at least
$r(\int h^\uparrow\psi-3\epsilon/4)$.
For $r\geq8B^{1/2}$, $N\geq r/(4B^{1/2})$.
Independence of the good events gives
\[
\begin{split}
\mathbb P(F>aN)
&\leq2^N\exp\{-c_0B^{1/2}\lceil aN\rceil\}\\
&\leq\exp\{-c_0aNB^{1/2}/2\}
\leq\exp(-c_0ar/8)
\end{split}
\]
for all sufficiently large $B,T$. The target event is contained
in $\mathcal E_T^c\cup\{F>aN\}$. Combining these bounds proves
the assertion, decreasing $c$ and increasing $C$ if necessary.

Finally, the ARL block lengths are $\lceil B^{1/2}\rceil$ when
$B=\log\gamma$, and a block intersecting $[T/2,2T]$ starts
between $T/3$ and $2T$ eventually. Monotonicity of $K_B(a)$
gives the stated bin bounds. For PFA and $B=\log T$, the relevant
block starts have logarithms $\log T+O(1)$. The second argument
of the minimum in \eqref{eq:learning-K} is then larger than the
first, so the bin counts are dyadic roundings of
$\{\log(e+a_j)\}^{1/8}$. These lie between half and twice the
corresponding $K_{\log T}$ values, and the block-length bounds
follow in the same way. This proves the final schedule assertions, and the proof of Lemma~\ref{lem:learned-score-growth} follows.
\hfill$\blacksquare$
\bigskip

\textbf{Proof of Proposition~\ref{prop:weighted-null-atom}.}
At each restart add mass $\rho$ to the density $1$ and multiply
the previous probability measure by $1-\rho$. The resulting
probability measure is deterministic and generates
$\overline M_{k,t}^\rho=\rho+(1-\rho)M_{k,t}$.
Theorem~\ref{th:weighted-null-validity} therefore applies unchanged.
Nonnegativity gives
$\overline S_t^{(w,\rho)}=\rho W_t+(1-\rho)S_t^{(w)}\geq\rho W_t$,
proving the pathwise bound when $W_t\to\infty$.
Dropping all but the restart $T_n+1$ gives
\[
\overline S_{n,T_n+r}^{(w,\rho)}
\geq(1-\rho)w_{T_n+1}M_{n,T_n+1,T_n+r}.
\]
This reaches $b_n$ as soon as the logarithm of the local process
reaches $\mathcal B_{n,w}(b_n)+\log\{1/(1-\rho)\}$.
Thus Proposition~\ref{prop:weighted-null-atom} follows.
\hfill$\blacksquare$
\bigskip

\section{Fixed-score implementation}
\label{apd:fixed-score-algorithm}
This is the original common-prior, fixed-score construction used in Section~\ref{sec:simulations}. In those experiments the specified betting class is finite, so the displayed prior integral is a finite sum. No learned permutation is applied.

\begin{algorithm2e}[h!]
\LinesNumbered
\small
\caption{Fixed-score weighted conformal implementation used in the simulations}
\label{alg:fixed-score-RCMM}
\KwIn{data stream $X_1,X_2,\ldots$ observed online; score function $s:\mathcal X\to\R$;
betting class $\mathcal F=\{f_\theta:\theta\in\Theta\}$ with prior $\Pi$;
regime $\texttt{PFA}$ at level $\alpha\in(0,1)$ or $\texttt{ARL}$ at level $\gamma\geq1$.}
\KwOut{alarm times $\tau_{S,w}(b)$ and $\tau_{Z,w}(b)$.}
\uIf{regime $=\texttt{PFA}$}{
  Choose weights $w=(w_k)_{k\geq1}$ with $w_k\geq0$ and
  $\|w\|_1=\sum_{k\geq1}w_k\leq1$,  set $b\leftarrow1/\alpha$\;
}
\Else{
  Choose weights $w=(w_k)_{k\geq1}$ with $w_k\geq0$ and
  $\|w\|_\infty\leq1$; set
  $b\leftarrow\gamma$\;
}
Initialize $S_0^{(w)}\leftarrow0$, $Z_0^{(w)}\leftarrow0$, and
$\tau_{S,w}(b)\leftarrow\infty$, $\tau_{Z,w}(b)\leftarrow\infty$\;
\For{$t=1,2,\ldots$}{
  Observe $X_t$ and draw $\lambda_t\sim U(0,1)$ independently of the data and preceding randomizers\;
  Compute the conformal $p$-value $p_t\leftarrow p_t(s(X_1),\ldots,s(X_t);\lambda_t)$
  as in \eqref{eq: randomized pval}\;
  Initialize the new restart: $G_{t,t}(\theta)\leftarrow f_\theta(p_t)$ for
  $\theta\in\Theta$\;
  \For{$k=1,\ldots,t-1$}{
    Update the local product
    $G_{k,t}(\theta)\leftarrow G_{k,t-1}(\theta)\,f_\theta(p_t)$\;
  }
  \For{$k=1,\ldots,t$}{
    Compute the local CMM
    $M_{k,t}\leftarrow\int_\Theta G_{k,t}(\theta)\,d\Pi(\theta)$\;
  }
  Form the weighted statistics
  $S_t^{(w)}\leftarrow\sum_{k=1}^t w_kM_{k,t}$ and
  $Z_t^{(w)}\leftarrow\max_{1\leq k\leq t}w_kM_{k,t}$\;
  \lIf{$S_t^{(w)}\geq b$ and $\tau_{S,w}(b)=\infty$}{$\tau_{S,w}(b)\leftarrow t$}
  \lIf{$Z_t^{(w)}\geq b$ and $\tau_{Z,w}(b)=\infty$}{$\tau_{Z,w}(b)\leftarrow t$}
  \lIf{$\tau_{S,w}(b)<\infty$ and $\tau_{Z,w}(b)<\infty$}{\Return $\tau_{S,w}(b)$, $\tau_{Z,w}(b)$}
}
\end{algorithm2e}
\end{document}